\documentclass[12pt,reqno]{amsart}
\usepackage[numbers,sort&compress]{natbib}
\usepackage{amsfonts,bbm}
\usepackage{amssymb,color}

\usepackage{fancyhdr}
\usepackage[titletoc]{appendix}
\usepackage{enumitem}
\usepackage{amsgen}
\usepackage{amscd}
\usepackage{amsmath}
\usepackage{mathtools}
\usepackage{mathrsfs}
\usepackage{cases}
\usepackage{accents} % This replaces the HARPOON package.

\usepackage[colorlinks=true,backref]{hyperref}
\hypersetup{urlcolor=blue, citecolor=green, linkcolor=blue}

\usepackage[margin=3cm, a4paper]{geometry}

\makeatletter
\renewcommand{\section}{\@startsection{section}{1}{\z@}{1cm}{0.5cm}{\normalfont\scshape\centering}}
\makeatother

\newtheorem{thm}{Theorem}[section]
\newtheorem{cor}[thm]{Corollary}
\newtheorem{lem}[thm]{Lemma}
\newtheorem{prop}[thm]{Proposition}
\newtheorem{defn}[thm]{ \bf{Definition}}

\theoremstyle{remark}
\newtheorem{remark}[thm]{Remark}

\newcommand{\EQ}[1]{\begin{align*} #1 \end{align*}}
\newcommand{\EQn}[1]{\begin{align}\begin{split} #1 \end{split}\end{align}}

\newcommand{\EQnnsub}[1]{\begin{subequations}\begin{align} #1 \end{align}\end{subequations}}
\newcommand{\Del}[1]{}

\def\norm#1{\left\|#1\right\|}
\def\normo#1{\|#1\|}
\def\normb#1{\big\|#1\big\|}

\def\normB#1{\Big\|#1\Big\|}

\def\abs#1{\left|#1\right|}
\def\abso#1{|#1|}
\def\absb#1{\big|#1\big|}

\def\brko#1{(#1)}

\def\mbrk#1{\left[#1\right]}
\def\mbrko#1{[#1]}
\def\mbrkb#1{\big[#1\big]}
\def\mbrkbb#1{\bigg[#1\bigg]}

\def\fbrk#1{\left\lbrace#1\right\rbrace}
\def\fbrko#1{\lbrace#1\rbrace}

\def\jb#1{\langle#1\rangle}

\def\wt#1{\widetilde{#1}}
\def\wh#1{\widehat{#1}}

\def\pd{\partial}

\newcommand{\ra}{{\rightarrow}}
\newcommand{\hra}{{\hookrightarrow}}

\def\lsm{\lesssim}

\newcommand{\N}{{\mathbb N}}

\newcommand{\R}{{\mathbb R}}
\newcommand{\C}{{\mathbb C}}
\newcommand{\Z}{{\mathbb Z}}

\newcommand{\A}{{\mathcal{A}}}

\newcommand{\F}{{\mathcal{F}}}

\newcommand{\J}{{\mathcal{J}}}
\newcommand{\E}{{\mathcal{E}}}

\newcommand{\W}{{\mathcal{W}}}
\newcommand{\dd}{{\mathrm{d}}}

\newcommand{\TT}{{\mathcal{T}}}

\newcommand{\cha}{{\mathbbm{1}}}

\newcommand{\re}{{\mathrm{Re}}}
\newcommand{\im}{{\mathrm{Im}}}

\def\ep{\varepsilon}
\def\al{\alpha}

\def\De{\Delta}

\def\la{\lambda}
\def\ga{\gamma}

\newcommand{\I}{\infty}
\def\rev#1{\frac{1}{#1}}

\def\Lam{\Lambda}

\def\Dy{\Delta}
\numberwithin{equation}{section}

\allowdisplaybreaks
\begin{document}
%\raggedbottom \activatedisplayskips
\title[Modified wave operators]{Modified wave operators for nonlinear Schr\"odinger equations in the full subcritical long range regime}


\author{Jia Shen}
\address{(J. Shen) School of Mathematical Sciences and LPMC\\
Nankai University\\ Tianjin 300071, China}
\email{shenjia@nankai.edu.cn}
\thanks{}

\author{Yifei Wu}
\address{(Y.Wu) School of Mathematical Sciences\\
	Nanjing Normal University\\
	Nanjing 210046, China}
\email{yerfmath@gmail.com}
\thanks{}

\date{}

\begin{abstract}
We construct modified wave operators for the nonlinear Schr\"odinger equation $i\partial_tu+\frac12\Delta u=|u|^pu$ in the full subcritical long-range case $0<p<2/d$, with small,  nonvanishing, analytic final data $W(x)$ with bounded logarithmic gradients. Previous results established large-time asymptotics for selected classes of Cauchy data. Moreover, the exact asymptotic expansion for $p<1/d$ remained unknown.

When $1/d<p<2/d$, our result gives the approximation
\EQ{
\frac{1}{(it)^{\frac d2}}e^{\frac{i|x|^2}{2t}}
W\brko{\frac{x}{t}}
\exp\mbrko{ -i\frac{t^{1-\frac{dp}{2}}-1}{1-\frac {dp}2}\abso{W\brko{\frac{x}{t}}}^p}.
}
The wave operator is constructed by an iteration in the analytic spaces with decreasing radius.

When $p\le 1/d$, we construct the profile from a finite truncation of a Fuchsian equation coupled with a transport equation. This profile still leaves a long-range triangular coupling whose terminal integral does not preserve the required fast decay class. The construction yields quantitative $L^q$ asymptotics for $2\le q\le\infty$ and uniqueness in the prescribed analytic asymptotic classes. The central new ingredients are a nonlinear final-state normal form that removes this long-range coupling and a mixed iteration in particular analytic spaces.
\end{abstract}

\maketitle

\tableofcontents

\section{Introduction and the main theorem}%\label{sec:intro}

Let $d\ge1$. We consider the nonlinear Schr\"odinger equation
\EQn{\label{eq:nls}
i\pd_tu+\frac12\De u=\la |u|^pu,
\qquad \la\in\R,
\qquad
x\in\R^d ,
}
in the subcritical long-range regime
\EQ{
0<p<\frac2d.
}
For short-range nonlinearities, that is $p>2/d$, the wave operator prescribes a final state $u_\pm\in L^2$ and constructs a nonlinear solution satisfying linear scattering:
\EQn{\label{eq:linear-scattering-def}
\normb{u(t)-e^{\frac12it\De}u_\pm}_{L^2}\ra0,
\qquad t\ra\pm\I.
}
The study of scattering theory includes the following:
\begin{itemize}
\item (Existence) Given $u_\pm$, prove the existence of a nonlinear solution satisfying \eqref{eq:linear-scattering-def}.
\item (Uniqueness) Prove uniqueness of the nonlinear solution with the prescribed final state $u_\pm$.
\item (Completeness) Given a global nonlinear solution, prove the existence of a final state $u_\pm$ satisfying \eqref{eq:linear-scattering-def}.
\end{itemize}
The first two problems define the wave operator
\EQ{
\Omega_\pm^{\rm lin}:u_\pm\longmapsto u(0),
}
whereas the third is usually referred to as the scattering problem. In other words, the construction of a wave operator is a final-state problem, while completeness is an initial-value problem.

Heuristically, the construction of a scattering solution relies on the decay of the free evolution, and the problem becomes harder as $p$ decreases. When $p\ge4/d$, standard Strichartz arguments yield small-data scattering and wave operators in suitable Sobolev spaces. There are extensive, important results for large data scattering in the range $p\ge4/d$; we do not discuss this in details. We refer to \cite{Cazenave2003,Tao2006,KochTataruVisan2014,Dodson2019} for general references. 

The short-range but mass-subcritical regime $2/d<p<4/d$ is more delicate. Available deterministic scattering results generally impose additional spatial localization, decay, or symmetry. Small-data scattering under spatial decay assumptions goes back to Strauss \cite{Strauss1981}, while radial-data scattering was studied by Hidano \cite{Hidano2008} and Guo-Wang \cite{GuoWang2014}. For the defocusing equation, Tsutsumi and Yajima \cite{TsutsumiYajima1984} proved $L^2$ scattering for initial data $u_0\in H^1$ satisfying $xu_0\in L^2$.  Cazenave-Weissler \cite{CazenaveWeissler1992} used the pseudo-conformal transformation to construct wave operators in the weighted energy space for $p>\max\fbrko{2/d,4/(d+2)}$. Ginibre-Ozawa-Velo \cite{GinibreOzawaVelo1994} subsequently used fractional weighted regularity and Besov estimates to reach every short-range power $p>2/d$ in dimension $d=3$ and to improve the admissible range in dimensions $d\ge4$. For the final-state existence problem in the short-range regime, Nakanishi \cite{Nakanishi2001} constructed asymptotically free nonlinear solutions for arbitrary prescribed free $L^2$ data.

Killip-Masaki-Murphy-Visan \cite{KillipMasakiMurphyVisan2017} proved that a solution is global and scatters if  $\sup_{t\in I}\normo{|x|^{|d/2-2/p|}e^{-\frac12it\Delta}u(t)}_{L^2}<\infty$ on its maximal lifespan $I$, in the range $\max\fbrko{2/d,4/(d+2)}<p<4/d$. For radial defocusing solutions in dimensions $d\ge3$, Killip-Masaki-Murphy-Visan \cite{KillipMasakiMurphyVisan2019} later proved an analogous conditional scattering result under an a priori bound in $\dot H^{s_c}$ for a subrange of $2/d<p<4/d$. In 3D defocusing quadratic case, our previous work \cite{ShenWuQuadratic2024} established global well-posedness for arbitrary radial data in the sharp critical weighted space $\mathcal F\dot H^{1/2}$, thereby removing the a priori critical weighted bound from the global-existence conclusion of \cite{KillipMasakiMurphyVisan2017}. For the focusing equation, Masaki \cite{Masaki2015Sharp} constructed a minimal non-scattering solution with respect to the critical weighted norm and proved that it is not the ground-state standing wave. 

There is further development along the study of Tsutsumi and Yajima \cite{TsutsumiYajima1984} in recent years. Burq-Georgiev-Tzvetkov-Visciglia \cite{BurqGeorgievTzvetkovVisciglia2023} upgraded this result to $H^1$ scattering for the same class of initial data throughout $2/d<p<4/d$. Another direction is to reduce the weighted condition. Lee \cite{Lee2021} showed that, above the Strauss exponent, there exists $0<\beta<p$ such that neither the wave operator nor the map from the initial datum to the scattering state admits an extension to a $C^{1+\beta}$ map from $L^2$ to $L^2$ near the origin. This motivates the study of intermediate weight assumptions between $u_0\in L^2$ and $\langle x\rangle u_0\in L^2$. In our recent work \cite{ShenWuMassSubcritical}, we obtained large-data $L^2$ scattering for $\jb{x}^su_0\in L_x^2$ with some $0<s<1$ and almost-sure scattering for randomized $L^2$-data. A more detailed review of scattering in the mass-subcritical case can be found in the introduction of \cite{ShenWuMassSubcritical}. 

For $p\le2/d$, Strauss \cite{Strauss1974} and Barab \cite{Barab1984} proved that an $L^2$ solution satisfying \eqref{eq:linear-scattering-def} must be identically zero. Hence $2/d$ is the threshold for linear scattering, and the case $p\le2/d$ is referred as long range. In the long-range regime, the free evolution has to be replaced by some modification:
\EQ{
\norm{u(t)-u_{\rm mod}(t;u_\pm)}_{L^2}\ra0,
\qquad t\ra\pm\I.
}

At the critical exponent $p=2/d$, modified scattering is well understood. Ozawa \cite{Ozawa1991} constructed the first modified wave operator for the one-dimensional cubic equation with small final data, and Ginibre-Ozawa \cite{GinibreOzawa1993} extended the construction to dimensions $d=2,3$. Carles \cite{Carles2001} later gave a new proof for the one-dimensional equation by methods of geometric optics. For the initial-value problem, Hayashi-Naumkin \cite{HayashiNaumkin1998} proved asymptotic completeness for small data in dimensions $d\le3$. Other approaches to the one-dimensional cubic problem include the final-state argument of Lindblad-Soffer \cite{LindbladSoffer2006}, the space-time resonance method of Kato-Pusateri \cite{KatoPusateri2011}, and the wave-packet method of Ifrim-Tataru \cite{IfrimTataru2014,IfrimTataru2024WavePackets}.

All above mentioned long-range results concern small final or initial data. A modified wave operator with large final data for the one-dimensional defocusing cubic NLS was constructed by Kawamoto-Mizutani \cite{KawamotoMizutani2025}. Georgiev-Ozawa \cite{GeorgievOzawa2025} obtained a completeness criterion under an additional $L^\I$ control assumption. For the integrable one-dimensional cubic NLS, Deift-Zhou \cite{DeiftZhou2002} used the Riemann-Hilbert steepest descent method introduced in \cite{DeiftZhou1993} to derive long-time asymptotics for large initial data.

The subcritical range $0<p<2/d$ is less developed. Hayashi-Kaikina-Naumkin \cite{HayashiKaikinaNaumkin1999} studied the one-dimensional generalized derivative NLS. Applied to \eqref{eq:nls} with $d=1$ and $\la=1$, their result \cite[pp.~95--96]{HayashiKaikinaNaumkin1999} gives, for $1<p<2$, after a time-independent phase renormalization, the explicit ODE-type profile 
\EQ{%\label{eq:intro-HKN-ODE-profile}
u_{\rm ODE}(t,x)
=\frac{1}{(it)^{\frac12}}e^{\frac{i|x|^2}{2t}}
W\brko{\frac{x}{t}}
\exp\mbrko{
-i\abso{W\brko{\frac{x}{t}}}^p
\frac{t^{1-\frac p2}-1}{1-\frac p2}
}.
}
Here $W$ is determined by the initial value, using the notation in this paper. Their results concern a class of Cauchy data satisfying weighted analytic bounds and a quantitative polynomial lower bound after removing a quadratic phase. For $0<p\le1$, their corresponding formula is
\EQ{%\label{eq:intro-HKN-lower-profile}
u(t,x)
&=\frac{1}{(it)^{\frac12}}e^{\frac{i|x|^2}{2t}}
W(\frac{x}{t})
\exp\mbrko{
-i\frac{t^{1-\frac p2}}{1-\frac p2}
|W(\frac{x}{t})|^p
-i\mathcal R_p(t,\frac{x}{t})
} +O\brko{t^{-\frac{1+p}{2}}},
}
uniformly in $x$, where $\mathcal R_p$ is real-valued. Their proof yields
\EQ{
\sup_{y\in\R}|\mathcal R_p(t,y)|
\lsm
\begin{cases}
1+t^{1-p},&0<p<1,
\\
1+\log t,&p=1.
\end{cases}
}
Related Hartree and time-dependent cubic models were treated in \cite{HayashiKaikinaNaumkin1998,HayashiKatoNaumkin1998}.
Following the quadratic-phase mechanism used by Cazenave-Weissler \cite{CazenaveWeissler1992}, Han \cite{Han2024} studied the standard NLS in dimensions $d\ge2$. He constructed a class of arbitrarily large nonvanishing initial data whose solutions decay at the linear rate $t^{-d/2}$, and obtained phase-modified asymptotics when $4/(d+\sqrt{d^2+8d})<p<2/d$.

In \cite[p.~95]{HayashiKaikinaNaumkin1999}, the authors stated: ``The existence of the modified wave operators for the sub-critical case is an open problem.'' Moreover, the exact asymptotic expansion for $p<1/d$ remained unknown. In this paper, we solve the modified wave operator problem in the whole subcritical long-range case $0<p<2/d$ for small, nonvanishing, analytic final data with bounded logarithmic gradients. Moreover, for $p\le 1/d$, we explicitly construct finite-order asymptotic profiles with quantitative remainder estimates.
\subsection{Main results}

We consider the case when $\la=1$ for simplicity, and our argument also works for general $\la\in\R$. We now fix $d\ge1$ and $0<p<2/d$, and set
\EQn{\label{eq:alpha-def}
\al:=\frac{dp}{2}.
}
Thus $0<\al<1$. Define the ray profile by
\EQn{\label{eq:ray-transform}
u(t,x)=\frac{1}{(it)^{\frac d2}}e^{\frac{i|x|^2}{2t}}
a\brko{t,\frac{x}{t}}.
}
Then \eqref{eq:nls} is equivalent to
\EQn{\label{eq:ray-equation}
i\pd_ta+\frac1{2t^2}\Dy a
=t^{-\al}|a|^pa.
}
For $2\le q\le\I$, with the convention $1/\I=0$, the ray transform satisfies the exact scaling identity
\EQn{\label{eq:ray-Lq-scaling}
\normB{
\frac1{(it)^{\frac d2}}e^{\frac{i|x|^2}{2t}}
F\brko{\frac{x}{t}}
}_{L^q_x}
=t^{-d\brko{\frac12-\frac1q}}\norm{F}_{L^q_y},
\qquad t>0,\quad F\in L^q(\R^d).
}
The ODE profile is
\EQn{\label{eq:ray-ode}
i\pd_ta=t^{-\al}|a|^pa.
}
For a prescribed $W$, the ODE-type profile contains the phase
\EQ{%\label{eq:Lambda-def}
\Lam(t)|W|^p,
\qquad
\Lam(t):=\int_1^t s^{-\al}\dd s
=\frac{t^{1-\al}-1}{1-\al}.
}

For $\rho\ge0$ and $\sigma\in\R$, we define the Wiener analytic space
\EQn{\label{eq:Arho-def}
\norm{f}_{\A^\sigma_\rho}
:=\int_{\R^d}e^{\rho\jb{\xi}}\jb{\xi}^{\sigma}
\absb{\widehat f(\xi)}\dd\xi .
}

\begin{defn}[Admissible final data]\label{def:final-data}
Fix $\rho_0>0$ and $R>0$. For $0<\ep\le1$, we call $W=e^\Psi$ an admissible final data of size $\ep$ if $\Psi\in C^\I(\R^d)$ satisfies
\EQn{\label{eq:R-bound}
\norm{\nabla\Psi}_{\A^6_{\rho_0}}
\le R,
}
and
\EQn{\label{eq:eps-small}
\norm{e^\Psi}_{L^2}
+\norm{e^\Psi}_{\A^6_{\rho_0}}
+\norm{e^{p\re\Psi}}_{\A^6_{\rho_0}}
\le\ep.
}
\end{defn}

This class contains infinite-dimensional families of complex-valued and possibly nonradial data. More precisely, Corollary \ref{cor:admissible-data-richness} in Appendix \ref{app:admissible-data} constructs $W_h=W_0e^h$ for every $h$ in a fixed open ball of $\A^7_{\rho_0}$, with all $W_h$ admissible for the same $(\rho_0,R,\ep)$.

\begin{thm}[Modified wave operators]\label{thm:main}
Let $d\ge1$ and $0<p<2/d$. Fix
\EQ{
0<\rho_*<\rho_1<\rho_0,
\qquad
R>0.
}
Let
\EQ{
W=e^\Psi\in L^2\cap\A^6_{\rho_0}
}
be an admissible final data of sufficiently small size.
\begin{enumerate}
\item If $1/d<p<2/d$, there exists a global $L^2$ solution of \eqref{eq:nls} such that
\EQ{
\norm{u(t)-u_{\rm ODE}(t)}_{L^2_x}\ra0,
\qquad t\ra\I,
}
where
\EQ{%\label{eq:main-ode-modifier}
u_{\rm ODE}(t,x)
=\frac{1}{(it)^{\frac d2}}e^{\frac{i|x|^2}{2t}}
W\brko{\frac{x}{t}}
\exp\mbrko{
-i\Lam(t)\abso{W\brko{\frac{x}{t}}}^p
}.
}
The solution is unique in the ODE-type solution class specified in Theorem \ref{thm:ode}.

\item If $0<p\le1/d$, then there exist $L$ and real-valued $S$ depending explicitly on $\Psi$, $d$, and $p$, and there is a global $L^2$ solution of \eqref{eq:nls} such that
\EQ{
\norm{u(t)-u_{\rm eik}(t)}_{L^2_x}\ra0,
\qquad t\ra\I,
}
where
\EQ{%\label{eq:main-eikonal-modifier}
u_{\rm eik}(t,x)
=\frac{1}{(it)^{\frac d2}}e^{\frac{i|x|^2}{2t}}
\exp\mbrko{
L\brko{t^{-\al},\frac{x}{t}}
-it^{1-\al}
S\brko{t^{-\al},\frac{x}{t}}
}.
}
The solution is unique in the Fuchsian solution class specified in Theorem \ref{thm:eikonal}. 
\end{enumerate}
\end{thm}

\begin{remark}[Modified wave operators]
The solutions in Theorem \ref{thm:main} define the modified wave operator
\EQn{\label{eq:wave-operator-def}
\W_+(\Psi):=u(0),
\qquad 0<p<\frac2d.
}
The construction of the ODE-type profile is given in Theorem \ref{thm:ode}. In the eikonal case, the functions $L$ and $S$ are finite truncations of the Fuchsian amplitude-eikonal system.  Theorem \ref{thm:main} uses the restriction $C_+=0$ in the detailed construction below, and Theorem \ref{thm:eikonal} proves that $u(0)$ is independent of the truncation order.
\end{remark}

The recursive construction of $L$ and $S$ is stated in Proposition \ref{prop:eikonal-prep}. It is inspired by the analogous construction for the long-range Hartree equation. Ginibre and Velo developed an amplitude-phase construction for
\EQ{%\label{eq:intro-hartree-model}
	i\pd_t u+\frac12\Delta u
	=\la\brko{|\cdot|^{-\delta}*|u|^2}u,
	\qquad
	0<\delta<1,
}
in appropriate Sobolev and Gevrey settings \cite{GinibreVelo2000I,GinibreVelo2000,GinibreVelo2001}. Writing
\EQ{
	u(t,x)
	=\frac{1}{(it)^{\frac d2}}e^{\frac{i|x|^2}{2t}}
	e^{-i\phi(t,x/t)}w(t,x/t),
}
the amplitude and phase satisfy the system
\EQ{%\label{eq:intro-GV-system}
	\left\{
	\begin{aligned}
		\pd_t w
		&=\frac{i}{2t^2}\Delta w
		+\frac1{2t^2}
		\brko{2\nabla\phi\cdot\nabla+\Delta\phi}w,
		\\
		\pd_t\phi
		&=\frac1{2t^2}|\nabla\phi|^2
		+t^{-\delta}\la\brko{|\cdot|^{-\delta}*|w|^2},
	\end{aligned}
	\right.
}
They then constructed the amplitude and the phase successively, adding terms until the remaining forcing became integrable. 

The Hartree construction does not directly resolve the family of classical nonlinearities $|u|^pu$ considered here. The Hartree interaction remains cubic: the exponent $\delta$ controls the long-range decay through the kernel, whereas $\la\brko{|\cdot|^{-\delta}*|w|^2}$ is always quadratic in the amplitude. For the nonlinearity $|u|^pu$, as $p$ approaches zero, the number of nondecaying corrections increases precisely while the local coefficient loses regularity at the zero set. Although our final profiles are nonvanishing, their spatial decay prevents any uniform positive lower bound. This is the additional difficulty that has to be overcome in the standard NLS with the nonlinearity $|u|^pu$.

Higher-order large-time asymptotics for scattering solutions of NLS with local nonlinearities have been studied in both the short-range and long-range regimes. In the short-range regime, Masaki \cite{Masaki2009} constructed recursive expansions and showed that their accuracy depends on the fractional part of the power. When the power is an integer, the resulting series converges to the solution and gives a complete expansion. At the long-range threshold, Deift-Zhou \cite{DeiftZhou1994} obtained expansions to all orders for the integrable one-dimensional cubic equation with Schwartz initial data by nonlinear steepest descent. Kita-Wada \cite{KitaWada2002} rigorously identified the second asymptotic term for the 1D equation with a cubic term and a higher-power perturbation. Lindblad-Soffer \cite{LindbladSoffer2006} constructed asymptotic profiles to arbitrary order for the 1D cubic--quintic final-state problem. More recently, Jendrej-Salvi \cite{JendrejSalvi2026} justified arbitrary finite-order expansions for small localized solutions of one-dimensional NLS with gauge-invariant polynomial nonlinearities, including the logarithmic phase generated by the cubic term. A related but different question concerns expansions of the wave and scattering operators with respect to the size of the data. For analytic short-range nonlinearities, Carles-Gallagher \cite{CarlesGallagher2009} proved real analyticity of these operators and gave a procedure for computing their Taylor coefficients. For the 1D cubic long-range equation, Carles \cite{Carles2024Dynamics} showed that the modified wave operator, its inverse, and the modified scattering operator are the identity to leading order near the origin. The first cubic corrections with quantitative error bounds have been computed and the method can yield higher-order terms.

\subsection{Main ideas}

We first discuss the modified wave operator with ODE-type modifier. When $1/d<p<2/d$, the leading modifier is obtained directly from \eqref{eq:ray-ode}. Theorem \ref{thm:ode} constructs the corresponding final-state correction by exploiting the following structure.
\begin{itemize}
\item 
We write the logarithm of the ray profile as $\Phi=\Psi-i\Lam(t)|W|^p+Z$ and split the correction as $Z=X+iY$, $X=\re Z$, and $Y=\im Z$. This separation reveals a triangular structure: the nonlinear term $-i t^{-\al}|W|^p(e^{pX}-1)$ is purely imaginary and depends only on $X$, while the real and imaginary components of the differential forcing have different decay rates. 
\item 
We next measure $X$ and $Y$ with the respective weights $t^\al$ and $t^{2\al-1}$, which correspond to the decay rates $t^{-\al}$ and $t^{1-2\al}$. 
\item 
For the final-state construction, we decompose the Duhamel map into the explicit source, the same-radius integral of the triangular term, and the differential and dispersive remainders. The same-radius term is small because $|W|^p$ is small, while the remaining nonlinear terms are time-integrable but cost analytic radius. 
\item 
We then expand $Z=\sum_{n\ge0}Z_n$ and distribute the loss of analytic radius among the successive Picard orders. The resulting order-by-order estimates prove absolute convergence of the Picard series. 
\end{itemize}

We next discuss the modified wave operator with the eikonal type modifier. Our construction is inspired by the Hamilton-Jacobi expansion of Ginibre-Velo for Hartree equations in \cite{GinibreVelo2000I,GinibreVelo2000,GinibreVelo2001}. 
We use the eikonal form
\EQ{
	a(t,y)=\exp\mbrko{
		L(\tau,y)-it^{1-\al}S(\tau,y)+Z(t,y)},
	\qquad
	\tau=t^{-\al}.
}
Let
\EQ{
	\gamma=\frac{1-\al}{\al}
	=\frac{2-dp}{dp}.
}
Since $t^{1-\al}=\tau^{-\gamma}$, future infinity is the regular singular endpoint $\tau=0$. The logarithmic amplitude $L$ and the real eikonal function $S$ satisfy
\EQn{\label{eq:intro-eikonal-system}
	\al\tau\pd_\tau S
	&=(1-\al)S-\frac{\tau}{2}\nabla S\cdot\nabla S
	-e^{p\re L},
	\\
	\al\pd_\tau L
	&=-\nabla S\cdot\nabla L-\frac12\Dy S
	-\frac{i}{2}\tau^\gamma
	\brko{\Dy L+\nabla L\cdot\nabla L}.
}
We refer to the first equation for $S$ as the Hamilton--Jacobi, eikonal, or Fuchsian equation, and to the second equation for $L$ as a transport equation with a small dispersive perturbation. We call a finite Frobenius truncation of \eqref{eq:intro-eikonal-system} a Fuchsian profile; equivalently, it may also be called an eikonal or Hamilton--Jacobi profile. In this paper, we use the term Fuchsian profile throughout. 

We do not solve \eqref{eq:intro-eikonal-system} exactly. A finite Frobenius expansion at $\tau=0$ is sufficient, provided that its residual has enough decay. The coefficients contain terms of the form
\EQ{
	\tau^{m+n\gamma}(\log\tau)^j.
}
Logarithmic factors appear when a generated exponent meets the homogeneous Fuchsian exponent. 

Up to the required truncation order, the Fuchsian formulation identifies all nonintegrable power-law terms and the logarithmic terms generated at resonant exponents. This does not by itself solve the final-state problem: after the finite Fuchsian profile is removed, the remainder equation still contains a long-range triangular coupling whose terminal integral does not preserve the required fast decay class. The main innovation is to remove this coupling by a nonlinear final-state normal form and then to close the transformed equation by a mixed analytic iteration:
\begin{itemize}
\item \emph{(A nonlinear final-state normal form removes the long-range triangular coupling)}. Let
\EQ{
\Phi^{(\mu)}(t,y)
=L^{(\mu)}\brko{t^{-\al},y}
-it^{1-\al}S^{(\mu)}\brko{t^{-\al},y}
}
be the logarithmic phase associated with a finite Fuchsian profile. If $Z=X+iY$, the correction equation contains
\EQ{
-it^{-\al}e^{p\re\Phi^{(\mu)}}(e^{pX}-1).
}
This term is purely imaginary and depends on the real part $X$. If $X=O(t^{-\kappa})$, then integration from $t$ to infinity gives the slower size $O(t^{1-\al-\kappa})$ and does not preserve the desired decay class. We therefore solve for the fast variable $\widetilde Z$ and reconstruct
\EQ{
Z=\widetilde Z+iK_\mu(\re\widetilde Z),
}
where $K_\mu$ solves a final-state transport-reaction equation determined by the truncated Fuchsian profile.

\item \emph{(The transformed equation requires a new Wiener-analytic iteration)}. The normal-form equation contains arbitrary nonlinear degree, ordered time integrations, and a loss of analytic radius at every differentiation. We work in exponentially weighted Wiener algebras and estimate the nonlinear degree and the number of time integrations simultaneously, while allocating analytic radii across the Picard orders. The resulting estimates prove absolute convergence of the Picard series. The same estimates yield compatibility between different Fuchsian truncations and uniqueness in the prescribed solution class.

\end{itemize}

In summary, the combination of the Fuchsian profile with the nonlinear final-state normal form is the central new ingredient of this paper.

\section{Preliminaries}%\label{sec:prelim}

\subsection{Notation}
We write $A\lsm B$ if $A\le CB$ for a constant $C>0$. A subscript, as in $A\lsm_\Lambda B$, indicates the dependencies relevant to the argument. Unless otherwise stated, the implicit constant may also depend on the fixed structural parameters. The real and imaginary parts of a complex-valued function are denoted by $\re f$ and $\im f$. For the correction variable we usually write
\EQ{
Z=X+iY,\qquad X=\re Z,\quad Y=\im Z .
}
We use $\wh f$ or $\F f$ to denote the Fourier transform of $f$:
\EQ{
	\wh f(\xi)=\F f(\xi):=
	\rev{(2\pi)^{d/2}}\int_{\R^d} e^{-ix\cdot\xi}f(x)\dd x.
}
We also define the inverse Fourier transform:
\EQ{
	\F^{-1} g(x):=
	\rev{(2\pi)^{d/2}}\int_{\R^d} e^{ix\cdot\xi}g(\xi)\dd \xi.
}
Using the Fourier transform, we use $\jb{\nabla}:=\F^{-1}\jb{\xi}\F$, where the Japanese bracket $\jb{\cdot}$ is defined by
\EQ{
\jb{\xi} := \brko{1+|\xi|^2}^{1/2}.
}  
For $\rho\ge0$ and $\sigma\in\R$, we use the analytic Wiener norm
\EQ{
\norm{f}_{\A^\sigma_\rho}
:=\int_{\R^d}e^{\rho\jb{\xi}}\jb{\xi}^{\sigma}
\absb{\widehat f(\xi)}\dd\xi .
}
The space $\A^\sigma_\rho$ consists of all tempered distributions for which this norm is finite.  The parameter $\rho>0$ measures the analytic radius, and the Fourier $L^1$ structure gives $\A^0_\rho\hra L^\I$.  All spatial function spaces are taken over $\R^d$, and we omit this for simplicity.

A function $F:\C\to\C$ is entire if it is holomorphic on $\C$, namely
\EQ{
	F(z)=\sum_{n=0}^{\I}a_nz^n,
	\qquad z\in\C,
}
with infinite radius of convergence.

\subsection{A real analytic implicit function lemma}

\begin{lem}[Real analytic implicit function theorem]\label{lem:real-analytic-ift}
Let $A_0\in\R$. Let $F$ be real analytic near $(0,A_0)$ in the following sense: there exist $r_0,\delta_0>0$ and $c_{mn}\in\R$ such that
\EQ{%\label{eq:real-analytic-expansion}
F(z,A)
=\sum_{m,n\ge0}c_{mn}z^m(A-A_0)^n,
\qquad |z|<r_0,\quad |A-A_0|<\delta_0,
}
and
\EQ{%\label{eq:real-analytic-absolute-convergence}
\sum_{m,n\ge0}|c_{mn}|r^m\delta^n<\I,
\qquad 0\le r<r_0,\quad 0\le\delta<\delta_0.
}
Assume that
\EQ{%\label{eq:real-analytic-ift-hypothesis}
F(0,A_0)=0,\qquad \pd_AF(0,A_0)\ne0.
}
Then there exist $r>0$, $\delta>0$, and a unique real analytic function $A:(-r,r)\to(A_0-\delta,A_0+\delta)$ such that
\EQn{\label{eq:real-analytic-ift-conclusion}
A(0)=A_0,\qquad F(z,A(z))=0,\qquad |z|<r.
}
If $A_0>0$, then $r$ can also be chosen so that
\EQn{\label{eq:real-analytic-ift-positive}
\frac{A_0}{2}\le A(z)\le2A_0,\qquad 0\le z<r.
}
\end{lem}

\begin{proof}
Let
\EQ{
c=\pd_AF(0,A_0),\qquad a=A-A_0,
}
and, for $\sigma>0$, write
\EQ{
D_\sigma:=\fbrk{w\in\C: |w|<\sigma}.
}
We use the same notation for the holomorphic extension
\EQ{
F(z,A_0+a)=\sum_{m,n\ge0}c_{mn}z^ma^n,
\qquad (z,a)\in D_{r_0}\times D_{\delta_0}.
}
Define
\EQn{\label{eq:ift-G-def}
G(z,a):=a-\frac{F(z,A_0+a)}{c}.
}
Then
\EQ{%\label{eq:ift-G-origin}
G(0,0)=0,\qquad \pd_aG(0,0)=0.
}
Choose $0<r<r_0$ and $0<\delta<\delta_0$ such that
\EQn{\label{eq:ift-contraction-constants}
\max_{\overline D_r\times\overline D_\delta}|\pd_aG|\le\frac12,
\qquad
\max_{|z|\le r}|G(z,0)|\le\frac{\delta}{2}.
}
For $|z|<r$ and $|a|,|b|\le\delta$, \eqref{eq:ift-contraction-constants} gives
\EQ{%\label{eq:ift-contraction-map}
|G(z,a)|
&\le |G(z,0)|+\max_{\overline D_r\times\overline D_\delta}|\pd_aG|\,|a|
\le \delta,
\\
|G(z,a)-G(z,b)|
&\le \frac12|a-b|.
}
Hence, by contraction mapping, for each $z\in D_r$, the equation
\EQn{\label{eq:ift-fixed-point-equation}
a=G(z,a),\qquad |a|\le\delta,
}
has a unique solution. More precisely, if
\EQn{\label{eq:ift-picard-iterates}
a_0(z)=0,\qquad a_{k+1}(z)=G(z,a_k(z)),
}
then
\EQ{%\label{eq:ift-picard-sequence}
|a_{k+1}(z)-a_k(z)|\le 2^{-k-1}\delta,
\qquad |z|<r.
}
Therefore $a_k\to a$ uniformly on compact subsets of $D_r$, and $a$ is holomorphic in $D_r$. By \eqref{eq:ift-G-def} and \eqref{eq:ift-fixed-point-equation},
\EQ{%\label{eq:ift-F-zero}
F(z,A_0+a(z))=0.
}
For real $z$, the iterates in \eqref{eq:ift-picard-iterates} are real-valued. Hence $a(z)\in\R$ for $z\in(-r,r)$, and $A(z)=A_0+a(z)$ satisfies \eqref{eq:real-analytic-ift-conclusion}. If $A_1$ and $A_2$ are two such solutions, then for $a_j=A_j-A_0$
\EQ{
|a_1(z)-a_2(z)|=|G(z,a_1(z))-G(z,a_2(z))|\le\frac12|a_1(z)-a_2(z)|,
}
so $A_1=A_2$. Finally, if $A_0>0$, choose $r$ so that
\EQ{
|A(z)-A_0|\le\frac{A_0}{2},\qquad 0\le z<r.
}
Then \eqref{eq:real-analytic-ift-positive} follows.
\end{proof}

\subsection{Analytic Wiener estimates}

\begin{lem}[Analytic Wiener algebra and composition]\label{lem:fourier-algebra}
Let $r\ge0$ and $\rho\ge0$. The following statements hold.
\begin{enumerate}[label=\textup{(\roman*)},leftmargin=*]
\item The space $\A^r_\rho$ is a Banach algebra:
\EQn{\label{eq:fourier-algebra}
\norm{fg}_{\A^r_\rho}
\le C_r
\norm{f}_{\A^r_\rho}
\norm{g}_{\A^r_\rho}.
}
\item The Wiener norm controls the $L^\I$ norm:
\EQn{\label{eq:wiener-embedding}
\norm{f}_{L^\I}\le (2\pi)^{-d/2}\norm{f}_{\A^0_\rho}.
}
\item The space is invariant under complex conjugation and taking real or imaginary parts:
\EQn{\label{eq:wiener-real-conjugate}
\norm{\overline f}_{\A^r_\rho}
=\norm{f}_{\A^r_\rho},
\qquad
\norm{\re f}_{\A^r_\rho}
+\norm{\im f}_{\A^r_\rho}
\le2\norm{f}_{\A^r_\rho}.
}
\item If $f\in\A^0_\rho$ and $g\in L^2$, then
\EQn{\label{eq:wiener-L2-product}
\norm{fg}_{L^2}
\le (2\pi)^{-d/2}\norm{f}_{\A^0_\rho}\norm{g}_{L^2}.
}
\item Let $F$ be an entire function with $F(0)=0$. For each $R>0$, if
\EQ{
\norm{u}_{\A^r_\rho}+\norm{v}_{\A^r_\rho}\le R,
}
then
\EQn{\label{eq:zero-composition}
\norm{F(u)}_{\A^r_\rho}\lsm_R\norm{u}_{\A^r_\rho},
\qquad
\norm{F(u)-F(v)}_{\A^r_\rho}\lsm_R\norm{u-v}_{\A^r_\rho}.
}
The implicit constants may also depend on $F$.
\item Under the same bound on $u$ and $v$, if $B\in \A^r_\rho$ and $G$ is entire, then
\EQn{\label{eq:product-composition}
\norm{BG(u)}_{\A^r_\rho}\lsm_R\norm{B}_{\A^r_\rho},
}
and
\EQn{\label{eq:product-composition-lip}
\norm{B\brko{G(u)-G(v)}}_{\A^r_\rho}
\lsm_R\norm{B}_{\A^r_\rho}\norm{u-v}_{\A^r_\rho}.
}
The implicit constant may also depend on $G$.
\item If $0\le\rho'<\rho$ and $\sigma\ge0$, then
\EQn{\label{eq:radius-gap-estimate}
\norm{|D|^\sigma f}_{\A^r_{\rho'}}
\lsm_\sigma(\rho-\rho')^{-\sigma}\norm{f}_{\A^r_\rho}.
}
Consequently,
\EQ{
\norm{f}_{\A^{r+\sigma}_{\rho'}}
\lsm_\sigma\mbrko{1+(\rho-\rho')^{-\sigma}}
\norm{f}_{\A^r_\rho}.
}
In particular, one spatial derivative gives one inverse radius gap:
\EQ{
\norm{\nabla f}_{\A^r_{\rho'}}
\lsm(\rho-\rho')^{-1}\norm{f}_{\A^r_\rho}.
}
\end{enumerate}
\end{lem}
\begin{proof}
The Fourier identity
\EQ{
\widehat{fg}=(2\pi)^{-d/2}\widehat f*\widehat g
}
and the inequalities
\EQ{
e^{\rho\jb{\xi}}
\le e^{\rho\jb{\eta}}e^{\rho\jb{\xi-\eta}},
\qquad
\jb{\xi}^r\le C_r\jb{\eta}^r\jb{\xi-\eta}^r
}
give \eqref{eq:fourier-algebra} by the $L^1$ convolution inequality. Completeness follows from the completeness of the corresponding weighted Fourier $L^1$ space. Fourier inversion gives \eqref{eq:wiener-embedding}, and the $L^1*L^2\to L^2$ convolution inequality gives \eqref{eq:wiener-L2-product}.  Since \(\widehat{\overline f}(\xi)=\overline{\widehat f(-\xi)}\), \eqref{eq:wiener-real-conjugate} follows from the symmetry of the weight.  Since $F$ is entire and $F(0)=0$, write
\EQ{
F(z)=\sum_{n\ge1}a_nz^n,
\qquad z\in\C .
}
Thus
\EQ{%\label{eq:composition-power-bound}
\norm{u^n}_{\A^r_\rho}
\le C_r^{n-1}\norm{u}_{\A^r_\rho}^n
\le C_r^{n-1}R^{n-1}\norm{u}_{\A^r_\rho}.
}
Since the series of $F$ is entire,
\EQ{
\norm{F(u)}_{\A^r_\rho}
\le
\sum_{n\ge1}|a_n|C_r^{n-1}R^{n-1}\norm{u}_{\A^r_\rho}
\lsm_R\norm{u}_{\A^r_\rho}.
}
The Lipschitz estimate in \eqref{eq:zero-composition} follows from
\EQ{
u^n-v^n=(u-v)\sum_{m=0}^{n-1}u^{n-1-m}v^m
}
and
\EQ{
\norm{u^n-v^n}_{\A^r_\rho}
\le nC_r^{n-1}R^{n-1}\norm{u-v}_{\A^r_\rho}.
}
For \eqref{eq:product-composition}, write
\EQ{
BG(u)=G(0)B+B\brko{G(u)-G(0)}
}
and apply \eqref{eq:fourier-algebra} and \eqref{eq:zero-composition} to $G-G(0)$.  The Lipschitz estimate \eqref{eq:product-composition-lip} follows from \eqref{eq:fourier-algebra} and the Lipschitz part of \eqref{eq:zero-composition} applied to $G-G(0)$.

It remains to prove \eqref{eq:radius-gap-estimate}. Set $a=\rho-\rho'>0$. Since $\abs{\xi}\le\jb{\xi}$,
\EQ{
\abs{\xi}^{\sigma}e^{\rho'\jb{\xi}}
&\le
\jb{\xi}^{\sigma}e^{-a\jb{\xi}}e^{\rho\jb{\xi}}
\\
&=
a^{-\sigma}\brko{a\jb{\xi}}^\sigma e^{-a\jb{\xi}}e^{\rho\jb{\xi}}.
}
The function $x^\sigma e^{-x}$ is bounded on $[0,\I)$. Hence
\EQ{
\abs{\xi}^{\sigma}e^{\rho'\jb{\xi}}
\lsm_\sigma a^{-\sigma}e^{\rho\jb{\xi}}
=(\rho-\rho')^{-\sigma}e^{\rho\jb{\xi}}.
}
Multiplying this pointwise inequality by $\jb{\xi}^{r}\abso{\widehat f(\xi)}$ and taking the $L^1_\xi$ norm gives \eqref{eq:radius-gap-estimate}. The preceding consequence follows from
\EQ{
\jb{\xi}^\sigma\lsm_\sigma\brko{1+|\xi|^\sigma}.
}
\end{proof}

\begin{lem}[Analytic product estimates]%\label{lem:analytic-radius-products}
Let $r\ge1$ and $0\le\rho'<\rho<\rho_1$. Then
\EQn{\label{eq:analytic-first-order-product}
\norm{a\nabla f}_{\A^r_{\rho'}}
\lsm_{r,\rho_1}(\rho-\rho')^{-1}
\norm{a}_{\A^r_\rho}\norm{f}_{\A^r_\rho},
}
and
\EQn{\label{eq:analytic-gradient-product}
\norm{\nabla f\cdot\nabla g}_{\A^r_{\rho'}}
\lsm_{r,\rho_1}(\rho-\rho')^{-1}
\norm{f}_{\A^r_\rho}\norm{g}_{\A^r_\rho}.
}
If $\rho'<\rho<\rho''<\rho_1$, then
\EQn{\label{eq:analytic-second-derivative}
\norm{\Dy f}_{\A^r_{\rho'}}
\lsm_{r,\rho_1}(\rho-\rho')^{-1}(\rho''-\rho)^{-1}
\norm{f}_{\A^r_{\rho''}}.
}
\end{lem}

\begin{proof}
By \eqref{eq:fourier-algebra} and \eqref{eq:radius-gap-estimate},
\EQ{
\norm{a\nabla f}_{\A^r_{\rho'}}
\le C_r\norm{a}_{\A^r_{\rho'}}\norm{\nabla f}_{\A^r_{\rho'}}
\lsm_{r,\rho_1}(\rho-\rho')^{-1}\norm{a}_{\A^r_\rho}\norm{f}_{\A^r_\rho},
}
which proves \eqref{eq:analytic-first-order-product}. Applying \eqref{eq:radius-gap-estimate} to the product with polynomial weight $r-1$ gives
\EQ{
\norm{\nabla f\cdot\nabla g}_{\A^r_{\rho'}}
&\lsm\mbrko{\norm{\nabla f\cdot\nabla g}_{\A^{r-1}_{\rho'}}+\norm{|D|(\nabla f\cdot\nabla g)}_{\A^{r-1}_{\rho'}}}
\\
&\lsm\brko{1+(\rho-\rho')^{-1}}\norm{\nabla f\cdot\nabla g}_{\A^{r-1}_\rho}
\\
&\lsm_{\rho_1}(\rho-\rho')^{-1}\norm{\nabla f\cdot\nabla g}_{\A^{r-1}_\rho}.
}
Since $r-1\ge0$, \eqref{eq:fourier-algebra} yields
\EQ{
\norm{\nabla f\cdot\nabla g}_{\A^{r-1}_\rho}
\le C_r\norm{f}_{\A^r_\rho}\norm{g}_{\A^r_\rho},
}
and hence \eqref{eq:analytic-gradient-product}. Applying the one-derivative estimate in \eqref{eq:radius-gap-estimate} first from $\rho''$ to $\rho$ and then from $\rho$ to $\rho'$ proves \eqref{eq:analytic-second-derivative}.
\end{proof}
\subsection{Linear and Duhamel estimates}
For $T\le t\le s$, let
\EQ{%\label{eq:U-def}
U(t,s)=\exp\brko{i\brko{\frac1{2s}-\frac1{2t}}\Dy}.
}
Then
\EQ{%\label{eq:U-eq}
\pd_tU(t,s)=\frac{i}{2t^2}\Dy U(t,s),
\qquad
U(s,s)=I.
}
We first record linear propagator estimates for $U$. The estimate for $U(t,s)-I$ follows from the multiplier bound and \eqref{eq:radius-gap-estimate}.
\begin{lem}[The linear estimates]%\label{lem:linear-propagator}
For every $\rho\ge0$ and $\sigma\in\R$,
\EQn{\label{eq:U-isometry}
\norm{U(t,s)f}_{\A^\sigma_\rho}
=\norm{f}_{\A^\sigma_\rho}.
}
Moreover, for $0<\nu<1$, $0\le\rho'<\rho$, and $T\le t\le s$,
\EQn{\label{eq:UminusI}
\norm{\brko{U(t,s)-I}f}_{\A^\sigma_{\rho'}}
\lsm_\nu t^{-\nu}
(\rho-\rho')^{-2\nu}\norm{f}_{\A^\sigma_\rho}.
}
At a fixed analytic radius one also has the polynomial-weight version
\EQn{\label{eq:UminusI-sobolev}
\norm{\brko{U(t,s)-I}f}_{\A^\sigma_{\rho}}
\lsm_a t^{-a}
\norm{f}_{\A^{\sigma+2a}_{\rho}},
\qquad
0<a<1.
}
\end{lem}

\begin{proof}
The Fourier multiplier of $U(t,s)$ is $\exp\mbrko{-i\brko{\frac1{2s}-\frac1{2t}}\abso{\xi}^2}$, which has modulus one; this proves \eqref{eq:U-isometry}. Since $\abso{\frac1{2s}-\frac1{2t}}\le \frac1{2t}$,
\EQ{
\abs{\exp\mbrko{-i\brko{\frac1{2s}-\frac1{2t}}\abso{\xi}^2}-1}
\lsm_\nu t^{-\nu}\abso{\xi}^{2\nu}.
}
Combining this with \eqref{eq:radius-gap-estimate} proves \eqref{eq:UminusI}. The same multiplier bound, without changing the analytic radius, gives
\EQ{
\norm{\brko{U(t,s)-I}f}_{\A^\sigma_{\rho}}
\lsm_a
t^{-a}
\norm{|D|^{2a}f}_{\A^\sigma_\rho}
\le
t^{-a}
\norm{f}_{\A^{\sigma+2a}_\rho},
}
which proves \eqref{eq:UminusI-sobolev}.
\end{proof}

The following Duhamel estimate is used below.
\begin{lem}[Duhamel estimate]\label{lem:two-weight-duhamel}
Let $F$ be complex valued, and
\EQ{
V(t)=-\int_t^\I U(t,s)F(s)\dd s.
}
Let $0\le\rho'<\rho$ and $0<\nu<1$. Then
\EQn{\label{eq:two-weight-real}
\norm{\jb{\nabla}\re V(t)}_{\A^0_{\rho'}}
\lsm_\nu&
\int_t^\I\norm{\jb{\nabla}\re F(s)}_{\A^0_{\rho'}}\dd s
\\
&+t^{-\nu}(\rho-\rho')^{-2\nu}
\int_t^\I
\norm{\jb{\nabla}F(s)}_{\A^0_\rho}\dd s,
}
and
\EQn{\label{eq:two-weight-imag}
\norm{\jb{\nabla}\im V(t)}_{\A^0_{\rho'}}
\lsm_\nu&
\int_t^\I\norm{\jb{\nabla}\im F(s)}_{\A^0_{\rho'}}\dd s
\\
&+t^{-\nu}(\rho-\rho')^{-2\nu}
\int_t^\I
\norm{\jb{\nabla}F(s)}_{\A^0_\rho}\dd s.
}
\end{lem}

\begin{proof}
Since
\EQ{
U(t,s)F(s)=F(s)+\brko{U(t,s)-I}F(s),
}
one has
\EQ{
\re V(t)&=-\int_t^\I\re F(s)\dd s
-\int_t^\I\re\brko{\brko{U(t,s)-I}F(s)}\dd s,
}
and similar identity for $\im$ also holds. By the commutation of $U(t,s)$ with $\nabla$ and \eqref{eq:UminusI},
\EQ{
\norm{\jb{\nabla}\brko{U(t,s)-I}F(s)}_{\A^0_{\rho'}}
&\lsm_\nu
t^{-\nu}(\rho-\rho')^{-2\nu}
\norm{\jb{\nabla}F(s)}_{\A^0_\rho}.
}
Integrating this estimate in $s$ proves \eqref{eq:two-weight-real} and \eqref{eq:two-weight-imag}.
\end{proof}

\subsection{The logarithmic phase equation}%\label{sec:phase}
Let
\EQ{%\label{eq:Q-def}
Q(\Phi)=\Dy\Phi+\nabla\Phi\cdot\nabla\Phi.
}
If $a=e^\Phi$, then
\EQ{
\Dy a=e^\Phi Q(\Phi).
}
Consequently, whenever the preceding differentiation identity is justified, the ray equation \eqref{eq:ray-equation} is equivalent to
\EQn{\label{eq:phase-equation}
i\pd_t\Phi+\frac1{2t^2}Q(\Phi)
=t^{-\al}e^{p\re\Phi}.
}
For existence, we solve \eqref{eq:phase-equation} through its integral formulation with the condition imposed at infinity and then define $a=e^\Phi$.

\section{The ODE-type profile for
\texorpdfstring{$1/d<p<2/d$}{1/d < p < 2/d}}
\label{sec:ode-proof}

Throughout this section, let $1/d<p<2/d$, or equivalently $1/2<\al<1$. Fix $\nu$ satisfying
\EQn{\label{eq:nu-choice}
1-\al<\nu<\frac12.
}
For a final profile $W=e^\Psi$, let
\EQn{\label{eq:ode-Phi0-def}
\Phi_0(t,y)=\Psi(y)-i\Lam(t)|W(y)|^p.
}
Then $e^{\Phi_0}$ solves the ODE \eqref{eq:ray-ode}. We write the full phase as
\EQn{\label{eq:ode-phase-correction-ansatz}
\Phi=\Phi_0+Z,
\qquad
Z=X+iY,
\qquad
Z(t)\ra0.
}
Using \eqref{eq:ode-phase-correction-ansatz} in \eqref{eq:phase-equation},
\EQn{\label{eq:ode-phase-correction-before-cancel}
i\pd_tZ+\frac1{2t^2}\brko{Q(\Phi_0)+\Dy Z+2\nabla\Phi_0\cdot\nabla Z+\nabla Z\cdot\nabla Z}
=t^{-\al}|W|^pe^{pX}-i\pd_t\Phi_0.
}
By \eqref{eq:ode-Phi0-def},
\EQ{%\label{eq:ode-Phi0-cancel}
i\pd_t\Phi_0=t^{-\al}|W|^p.
}
By \eqref{eq:ode-phase-correction-before-cancel},
\EQn{\label{eq:ode-Z-equation}
\pd_tZ=\frac{i}{2t^2}\Dy Z+N(t,Z),
}
where the nonlinearity is given by
\EQ{%\label{eq:ode-N-decomp}
N(t,Z):=N_0(t)+N_1(t,Z)+N_2(t,Z)+N_{\rm tri}(t,Z),
}
with
\EQn{\label{eq:ode-N0-def}
N_0(t)=\frac{i}{2t^2}Q(\Phi_0),
}
the first-order differential term
\EQn{\label{eq:ode-N1-def}
N_1(t,Z)=\frac{i}{t^2}\nabla\Phi_0\cdot\nabla Z,
}
the quadratic differential term
\EQn{\label{eq:ode-N2-def}
N_2(t,Z)=\frac{i}{2t^2}\nabla Z\cdot\nabla Z,
}
and the triangular type term
\EQn{\label{eq:ode-Ntri-def}
N_{\rm tri}(t,Z)=-i t^{-\al}|W|^p(e^{pX}-1),
\qquad X=\re Z.
}
The final-state Duhamel map associated with \eqref{eq:ode-Z-equation} is
\EQn{\label{eq:ode-Duhamel-map}
(\TT Z)(t):=-\int_t^\I U(t,s)N(s,Z(s))\dd s.
}

\begin{defn}[ODE-type solution class]\label{def:ode-asymptotic}
Fix $0<\rho_*<\rho_1$. For $T\ge2$, $0\le\rho<\rho_1$, and $Z=X+iY$, where $X$ and $Y$ are real-valued, define
\EQ{%\label{eq:XT-rho-norm}
\norm{Z}_{\mathcal X_T(\rho)}
:=
\sup_{t\ge T}
\mbrkb{
t^\al\norm{\jb{\nabla}X(t)}_{\A^0_{\rho}}
+t^{2\al-1}\norm{\jb{\nabla}Y(t)}_{\A^0_{\rho}}
}.
}
Let
\EQ{
\mathcal X_T(\rho)
:=\fbrk{Z=X+iY:\ \jb{\nabla}Z\in C([T,\I);\A^0_{\rho}),\
\ \norm{Z}_{\mathcal X_T(\rho)}<\I}.
}
We call $Z$ a final-state correction if there exist $T\ge2$ and $\rho_*<\rho<\rho_1$ such that
\EQ{
Z&\in\mathcal X_T(\rho),
\\
Z&=\TT Z\quad\hbox{in }\mathcal X_T(\rho')
\quad\hbox{for every }\rho_*<\rho'<\rho.
}
For a logarithm $\Psi$, the ODE-type solution class consists of the corresponding ray profiles
\EQ{%\label{eq:phase-class-form}
a(t,y)=\exp\mbrko{\Psi(y)-i\Lam(t)|W(y)|^p+Z(t,y)}.
}
\end{defn}

\begin{thm}[Modified wave operator for the ODE-type profile]
\label{thm:ode}
Let $d\ge1$ and $1/d<p<2/d$. Fix $0<\rho_*<\rho_1<\rho_0$ and $R>0$. There exists $0<\ep_0\le1$, depending on these parameters, such that the following holds whenever $0<\ep\le\ep_0$ and $W=e^\Psi$ is an admissible final profile of size $\ep$ in the sense of Definition \ref{def:final-data}.
\begin{enumerate}
\item There exist $T\ge2$ and a final-state correction $Z\in\mathcal X_T\brko{(\rho_*+\rho_1)/2}$ such that
\EQ{%\label{eq:ode-solution-form}
u(t,x)=\frac{1}{(it)^{\frac d2}}e^{\frac{i|x|^2}{2t}}
\exp\mbrko{
\Psi\brko{\frac{x}{t}}
-i\Lam(t)\abso{W\brko{\frac{x}{t}}}^p
+Z\brko{t,\frac{x}{t}}
}
}
solves \eqref{eq:nls} on $[T,\I)$ and belongs to the ODE-type solution class of Definition \ref{def:ode-asymptotic}.
\item For every $2\le q\le\I$, the solution has the quantitative modified asymptotics
\EQn{\label{eq:ode-Lq-remainder}
\norm{u(t)-u_{\rm ODE}(t)}_{L^q_x}
\lsm_R\norm{W}_{L^q_y}t^{-d\brko{\frac12-\frac1q}-(2\al-1)},
\qquad t\ge T,
}
where
\EQ{%\label{eq:ode-modifier-intro}
u_{\rm ODE}(t,x)
=\frac{1}{(it)^{\frac d2}}e^{\frac{i|x|^2}{2t}}
W\brko{\frac{x}{t}}
\exp\brko{-i\Lam(t)\abso{W\brko{\frac{x}{t}}}^p}.
}
Moreover,
\EQn{\label{eq:ode-Lq-modulus}
\norm{|u(t)|-|u_{\rm ODE}(t)|}_{L^q_x}
\lsm_R\norm{W}_{L^q_y}t^{-d\brko{\frac12-\frac1q}-\al},
}
and
\EQn{\label{eq:ode-Lq-leading}
\absb{t^{d\brko{\frac12-\frac1q}}\norm{u(t)}_{L^q_x}-\norm{W}_{L^q_y}}
\lsm_R\norm{W}_{L^q_y}t^{-\al}.
}
\item The solution is unique within the ODE-type solution class of Definition \ref{def:ode-asymptotic} for the same logarithm $\Psi$.
\item The solution constructed on $[T,\I)$ extends uniquely to a global $L^2$ solution on $[0,\I)$, and its value at time zero is independent of the admissible lower endpoint $T$.
\end{enumerate}
\end{thm}

\begin{remark}[Amplitude and $L^q$ asymptotics]
The estimates \eqref{eq:ode-Lq-modulus} and \eqref{eq:ode-Lq-leading} describe the amplitude asymptotics. The modulus error in \eqref{eq:ode-Lq-modulus} decays faster than the complex error in \eqref{eq:ode-Lq-remainder}. Thus the slower part of the complex remainder is a phase correction. Moreover, \eqref{eq:ode-Lq-leading} can be written as
\EQ{
\norm{u(t)}_{L^q_x}
=t^{-d\brko{\frac12-\frac1q}}\norm{W}_{L^q_y}
+O_R\brko{\norm{W}_{L^q_y}t^{-d\brko{\frac12-\frac1q}-\al}}.
}
Hence $t^{-d\brko{\frac12-\frac1q}}$ is the leading $L^q$ decay rate, with coefficient $\norm{W}_{L^q_y}$, and the long-range phase does not change this leading amplitude.
\end{remark}

\subsection{Nonlinear estimates}%\label{sec:ode-nonlinear-estimates}

All norms in this section are the analytic norms defined in \eqref{eq:Arho-def}. Constants may depend on $p,\rho_0,\rho_1,\rho_*$ and $R$, but not on $T$ or on the small parameter $\ep$. For the pointwise-in-time norm associated with $\mathcal X_T(\rho)$, write
\EQn{\label{eq:instant-norm}
	\norm{Z(t)}_{E_\rho(t)}
	:=t^\al\norm{\jb{\nabla}X(t)}_{\A^0_\rho}
	+t^{2\al-1}\norm{\jb{\nabla}Y(t)}_{\A^0_\rho}.
}

\begin{lem}[Nonlinear estimates for $N_0$]\label{lem:ode-N0}
Let $N_0$ be defined by \eqref{eq:ode-N0-def}. Under the assumptions of Theorem \ref{thm:ode},
\EQn{\label{eq:N0}
\norm{N_0(t)}_{E_{\rho_1}(t)}
\lsm_Rt^{-1}.
}
Consequently,
\EQn{\label{eq:N0-integrated}
\normb{\int_t^\I N_0(s)\dd s}_{E_{\rho_1}(t)}
\lsm_R 1.
}
Moreover, for every $0\le\rho<\rho_1$,
\EQn{\label{eq:N0-duhamel-bound}
\normb{\int_t^\I U(t,s)N_0(s)\dd s}_{E_{\rho}(t)}
\lsm_{R,\rho} 1.
}
\end{lem}

\begin{proof}
Write
\EQn{\label{eq:QPhi0-expansion}
Q(\Phi_0)
=B_0-i\Lam(t)B_1-\Lam(t)^2B_2,
}
where
\EQ{
B_0&:=\Dy\Psi+\nabla\Psi\cdot\nabla\Psi,
\\
B_1&:=\Dy |W|^p+2\nabla\Psi\cdot\nabla |W|^p,
\\
B_2&:=\nabla |W|^p\cdot\nabla |W|^p.
}
By \eqref{eq:eps-small}, Lemma \ref{lem:fourier-algebra}, and the radius gap $\rho_0-\rho_1>0$,
\EQn{\label{eq:N0-real-imag-B012}
\norm{B_0}_{\A^1_{\rho_1}}
+\norm{B_1}_{\A^1_{\rho_1}}
+\norm{B_2}_{\A^1_{\rho_1}}
\lsm_R 1.
}
From \eqref{eq:ode-N0-def} and \eqref{eq:QPhi0-expansion},
\EQn{\label{eq:N0-real-imag-formula}
\re N_0(t)
&=
\frac{\Lam(t)}{2t^2}\re B_1
-\frac{1}{2t^2}\im B_0,
\\
\im N_0(t)
&=
\frac{1}{2t^2}\re B_0
+\frac{\Lam(t)}{2t^2}\im B_1
-\frac{\Lam(t)^2}{2t^2}B_2.
}
Notice that $B_2$ is real-valued, hence it appears only in the imaginary component in \eqref{eq:N0-real-imag-formula}. Since $\Lam(t)\lsm t^{1-\al}$, by \eqref{eq:N0-real-imag-formula} and \eqref{eq:N0-real-imag-B012},
\EQn{\label{eq:N0-component-pointwise}
\norm{\jb{\nabla}\re N_0(t)}_{\A^0_{\rho_1}}
&\lsm_Rt^{-1-\al},
\\
\norm{\jb{\nabla}\im N_0(t)}_{\A^0_{\rho_1}}
&\lsm_Rt^{-2\al}.
}
Then \eqref{eq:N0} follows from \eqref{eq:instant-norm} and \eqref{eq:N0-component-pointwise}. Since $2\al>1$, by \eqref{eq:N0-component-pointwise},
\EQn{\label{eq:N0-component-integrated}
\int_t^\I \norm{\jb{\nabla}\re N_0(s)}_{\A^0_{\rho_1}}\dd s
&\lsm_Rt^{-\al},
\\
\int_t^\I \norm{\jb{\nabla}\im N_0(s)}_{\A^0_{\rho_1}}\dd s
&\lsm_Rt^{1-2\al}.
}
Finally,
\EQ{
\normb{\int_t^\I N_0(s)\dd s}_{E_{\rho_1}(t)}
&\le
t^\al\int_t^\I \norm{\jb{\nabla}\re N_0(s)}_{\A^0_{\rho_1}}\dd s
\\
&\quad+t^{2\al-1}\int_t^\I \norm{\jb{\nabla}\im N_0(s)}_{\A^0_{\rho_1}}\dd s
\\
&\lsm_R 1,
}
which proves \eqref{eq:N0-integrated}. For the Duhamel estimate in \eqref{eq:N0-duhamel-bound}, let
\EQ{
V_0(t):=-\int_t^\I U(t,s)N_0(s)\dd s.
}
Apply Lemma \ref{lem:two-weight-duhamel} with $F=N_0$, upper radius $\rho_1$, and lower radius $\rho$. By \eqref{eq:two-weight-real} and \eqref{eq:N0-component-integrated},
\EQn{\label{eq:N0-real-duhamel-proof}
\norm{\jb{\nabla}\re V_0(t)}_{\A^0_{\rho}}
&\lsm_\nu
\int_t^\I
\norm{\jb{\nabla}\re N_0(s)}_{\A^0_{\rho_1}}
\dd s
\\
&\quad+t^{-\nu}(\rho_1-\rho)^{-2\nu}
\int_t^\I
\norm{\jb{\nabla} N_0(s)}_{\A^0_{\rho_1}}\dd s
\\
&\lsm_{R,\rho}t^{-\al}.
}
Here we used $\nu>1-\al$. Similarly, by \eqref{eq:two-weight-imag} and \eqref{eq:N0-component-integrated},
\EQn{\label{eq:N0-imag-duhamel-proof}
\norm{\jb{\nabla}\im V_0(t)}_{\A^0_{\rho}}
&\lsm_\nu
\int_t^\I
\norm{\jb{\nabla}\im N_0(s)}_{\A^0_{\rho_1}}
\dd s
\\
&\quad+t^{-\nu}(\rho_1-\rho)^{-2\nu}
\int_t^\I
\norm{\jb{\nabla} N_0(s)}_{\A^0_{\rho_1}}\dd s
\\
&\lsm_{R,\rho}t^{1-2\al}.
}
Here we used $\nu>1-\al$ and $\al<1$. Then, \eqref{eq:N0-duhamel-bound} follows by \eqref{eq:N0-real-duhamel-proof} and \eqref{eq:N0-imag-duhamel-proof}.
\end{proof}

\begin{lem}[Nonlinear estimates for $N_{\rm tri}$]\label{lem:ode-Ntri}
Let $N_{\rm tri}$ be defined by \eqref{eq:ode-Ntri-def}. Let $\rho_*<\rho<\rho_1$. If $Z=X+iY$ and $\norm{Z(t)}_{E_\rho(t)}<\I$, then
\EQn{\label{eq:ode-Ntri-bound}
\norm{N_{\rm tri}(t,Z)}_{\A^1_\rho}
\lsm_R\ep t^{-2\al}
\exp\brko{Ct^{-\al}\norm{Z(t)}_{E_\rho(t)}}
\norm{Z(t)}_{E_\rho(t)}.
}
Moreover, if $Z_j=X_j+iY_j$, $j=1,2$, and $\norm{Z_j(t)}_{E_\rho(t)}<\I$, then
\EQn{\label{eq:ode-Ntri-lip}
&\norm{N_{\rm tri}(t,Z_1)-N_{\rm tri}(t,Z_2)}_{\A^1_\rho}
\\
&\quad\lsm_R
\ep t^{-2\al}
\exp\brko{Ct^{-\al}\mbrko{\norm{Z_1(t)}_{E_\rho(t)}+\norm{Z_2(t)}_{E_\rho(t)}}}
\norm{Z_1(t)-Z_2(t)}_{E_\rho(t)}.
}
\end{lem}

\begin{proof}
Recall \eqref{eq:ode-Ntri-def},
\EQ{
N_{\rm tri}(t,Z)=-i t^{-\al}|W|^p(e^{pX}-1).
}
From \eqref{eq:eps-small},
\EQn{\label{eq:ode-Wp-small}
\norm{|W|^p}_{\A^1_\rho}\lsm_R\ep.
}
By \eqref{eq:instant-norm},
\EQ{%\label{eq:ode-X-from-E}
\norm{X(t)}_{\A^1_\rho}
\le t^{-\al}\norm{Z(t)}_{E_\rho(t)}.
}
By the power series of $e^{px}-1$ and \eqref{eq:fourier-algebra},
\EQn{\label{eq:ode-expX-bound}
\norm{e^{pX(t)}-1}_{\A^1_\rho}
\lsm
\exp\brko{C\norm{X(t)}_{\A^1_\rho}}
\norm{X(t)}_{\A^1_\rho}.
}
By \eqref{eq:fourier-algebra}, \eqref{eq:ode-Wp-small}, and \eqref{eq:ode-expX-bound},
\EQn{\label{eq:ode-Ntri-size-proof-bound}
\norm{|W|^p(e^{pX}-1)}_{\A^1_\rho}
&\lsm_R
\ep\norm{e^{pX(t)}-1}_{\A^1_\rho}
\\
&\lsm_R\ep
\exp\brko{Ct^{-\al}\norm{Z(t)}_{E_\rho(t)}}
t^{-\al}\norm{Z(t)}_{E_\rho(t)}.
}
Multiplying \eqref{eq:ode-Ntri-size-proof-bound} by $t^{-\al}$ proves \eqref{eq:ode-Ntri-bound}.

For the difference estimate, the power series of $e^{px}$ yields
\EQn{\label{eq:ode-expX-diff-bound}
\norm{e^{pX_1(t)}-e^{pX_2(t)}}_{\A^1_\rho}
\lsm
\exp\brko{C\mbrko{\norm{X_1(t)}_{\A^1_\rho}+\norm{X_2(t)}_{\A^1_\rho}}}
\norm{(X_1-X_2)(t)}_{\A^1_\rho}.
}
Using \eqref{eq:ode-Wp-small}, \eqref{eq:ode-expX-diff-bound}, and \eqref{eq:instant-norm},
\EQn{\label{eq:ode-Ntri-lip-proof-bound}
&\norm{|W|^p(e^{pX_1}-e^{pX_2})}_{\A^1_\rho}
\\
&\quad\lsm_R
\ep
\exp\brko{Ct^{-\al}\mbrko{\norm{Z_1(t)}_{E_\rho(t)}+\norm{Z_2(t)}_{E_\rho(t)}}}
t^{-\al}\norm{Z_1(t)-Z_2(t)}_{E_\rho(t)}.
}
Multiplying \eqref{eq:ode-Ntri-lip-proof-bound} by $t^{-\al}$ proves \eqref{eq:ode-Ntri-lip}.
\end{proof}
\begin{lem}[Nonlinear estimates for $N_1$ and $N_2$]\label{lem:ode-N1N2-nonlinear}
Let $N_1$ and $N_2$ be defined by \eqref{eq:ode-N1-def} and \eqref{eq:ode-N2-def}, and let $\rho_*<\rho'<\rho<\rho_1$.  If $Z=X+iY$ satisfies $\norm{Z(t)}_{E_\rho(t)}<\I$, then
\EQn{\label{eq:ode-N1-size}
\norm{N_1(t,Z)}_{E_{\rho'}(t)}
\lsm_R
t^{-1-\al}\brko{1+(\rho-\rho')^{-1}}
\norm{Z(t)}_{E_\rho(t)},
}
and
\EQn{\label{eq:ode-N2-size}
\norm{N_2(t,Z)}_{E_{\rho'}(t)}
\lsm
t^{-1-\al}\brko{1+(\rho-\rho')^{-1}}
\norm{Z(t)}_{E_\rho(t)}^2.
}
Moreover, if $Z_j=X_j+iY_j$, $j=1,2$, satisfy $\norm{Z_j(t)}_{E_\rho(t)}<\I$, then
\EQn{\label{eq:ode-N1-lip}
\norm{N_1(t,Z_1)-N_1(t,Z_2)}_{E_{\rho'}(t)}
&\lsm_R
t^{-1-\al}\brko{1+(\rho-\rho')^{-1}}
\\
&\quad\cdot\norm{Z_1(t)-Z_2(t)}_{E_\rho(t)},
}
and
\EQn{\label{eq:ode-N2-lip}
\norm{N_2(t,Z_1)-N_2(t,Z_2)}_{E_{\rho'}(t)}
&\lsm
t^{-1-\al}\brko{1+(\rho-\rho')^{-1}}
\\
&\quad\cdot
\mbrko{\norm{Z_1(t)}_{E_\rho(t)}+\norm{Z_2(t)}_{E_\rho(t)}}
\\
&\quad\cdot
\norm{Z_1(t)-Z_2(t)}_{E_\rho(t)}.
}

\end{lem}

\begin{proof}
From
\EQ{%\label{eq:N1-Phi0-gradient-identity}
\nabla\Phi_0=\nabla\Psi-i\Lam(t)\nabla(|W|^p),
}
the assumptions on $\Psi$ imply
\EQ{%\label{eq:N1-Phi0-gradient-bound}
\norm{\nabla\Phi_0(t)}_{\A^2_{\rho_1}}
+\norm{\nabla^2\Phi_0(t)}_{\A^1_{\rho_1}}
\lsm_Rt^{1-\al}.
}
Since
\EQ{
\re\nabla\Phi_0=\nabla\re\Psi,
\qquad
\im\nabla\Phi_0=\nabla\im\Psi-\Lam(t)\nabla(|W|^p),
}
we also have
\EQn{\label{eq:ode-Phi0-component-bounds}
\norm{\jb{\nabla}\re\nabla\Phi_0(t)}_{\A^0_{\rho_1}}\lsm_R 1,
\qquad
\norm{\jb{\nabla}\im\nabla\Phi_0(t)}_{\A^0_{\rho_1}}\lsm_Rt^{1-\al}.
}
For $N_1$, the real and imaginary parts are
\EQn{\label{eq:ode-N1-component-split}
\re N_1
&=-t^{-2}\re\nabla\Phi_0\cdot\nabla Y
-t^{-2}\im\nabla\Phi_0\cdot\nabla X,
\\
\im N_1
&=t^{-2}\re\nabla\Phi_0\cdot\nabla X
-t^{-2}\im\nabla\Phi_0\cdot\nabla Y .
}
By \eqref{eq:instant-norm}, \eqref{eq:ode-Phi0-component-bounds}, \eqref{eq:radius-gap-estimate}, and \eqref{eq:fourier-algebra},
\EQ{%\label{eq:ode-N1-size-proof}
&\norm{N_1(t,Z)}_{E_{\rho'}(t)} \\
&\lsm_R
t^{-2}\brko{1+(\rho-\rho')^{-1}}t^\al \mbrko{\norm{\re\nabla\Phi_0}_{\A^0_\rho} \norm{\nabla Y}_{\A^0_\rho} + \norm{\im\nabla\Phi_0}_{\A^0_\rho} \norm{\nabla X}_{\A^0_\rho}
} 
\\
&\quad
+t^{-2}\brko{1+(\rho-\rho')^{-1}}t^{2\al-1} \mbrko{\norm{\re\nabla\Phi_0}_{\A^0_\rho} \norm{\nabla X}_{\A^0_\rho} + \norm{\im\nabla\Phi_0}_{\A^0_\rho} \norm{\nabla Y}_{\A^0_\rho}
}  \\
&\lsm_R
t^{-2}\brko{1+(\rho-\rho')^{-1}}
\mbrko{
t^\al\norm{\jb{\nabla}Y(t)}_{\A^0_\rho}
+t^\al\Lam(t)\norm{\jb{\nabla}X(t)}_{\A^0_\rho}
}
\\
&\quad
+t^{-2}\brko{1+(\rho-\rho')^{-1}}
\mbrko{
t^{2\al-1}\norm{\jb{\nabla}X(t)}_{\A^0_\rho}
+t^{2\al-1}\Lam(t)\norm{\jb{\nabla}Y(t)}_{\A^0_\rho}
}
\\
&\lsm_R
t^{-1-\al}\brko{1+(\rho-\rho')^{-1}}
\norm{Z(t)}_{E_\rho(t)}.
}
This proves \eqref{eq:ode-N1-size}.

For $N_2$, direct computation gives
\EQn{\label{eq:ode-N2-component-split}
\re N_2&=-t^{-2}\nabla X\cdot\nabla Y,
\\
\im N_2&=\frac1{2t^2}
\brko{\nabla X\cdot\nabla X-\nabla Y\cdot\nabla Y}.
}
Then,
\EQ{%\label{eq:ode-N2-size-proof}
&\norm{N_2(t,Z)}_{E_{\rho'}(t)} \\ &\lsm
t^{-2}\brko{1+(\rho-\rho')^{-1}}
t^\al
\norm{\jb{\nabla}X(t)}_{\A^0_\rho}
\norm{\jb{\nabla}Y(t)}_{\A^0_\rho}
\\
&\quad
+t^{-2}\brko{1+(\rho-\rho')^{-1}}
\mbrko{
t^{2\al-1}
\norm{\jb{\nabla}X(t)}_{\A^0_\rho}^2
+t^{2\al-1}
\norm{\jb{\nabla}Y(t)}_{\A^0_\rho}^2
}
\\
&\lsm
t^{-1-\al}\brko{1+(\rho-\rho')^{-1}}
\norm{Z(t)}_{E_\rho(t)}^2.
}
This proves \eqref{eq:ode-N2-size}.

For the difference of the linear term, use \eqref{eq:ode-N1-component-split} with $Z$ replaced by $Z_1-Z_2$:
\EQ{%\label{eq:ode-N1-lip-proof}
\norm{N_1(t,Z_1)-N_1(t,Z_2)}_{E_{\rho'}(t)} \lsm_R
t^{-1-\al}\brko{1+(\rho-\rho')^{-1}}
\norm{Z_1(t)-Z_2(t)}_{E_\rho(t)}.
}
This proves \eqref{eq:ode-N1-lip}. For the quadratic term, by \eqref{eq:ode-N2-component-split},
\EQ{%\label{eq:ode-N2-component-diff-split}
\re\brko{N_2(t,Z_1)-N_2(t,Z_2)}
&=-t^{-2}\brko{
\nabla(X_1-X_2)\cdot\nabla Y_1
+\nabla X_2\cdot\nabla(Y_1-Y_2)},
\\
\im\brko{N_2(t,Z_1)-N_2(t,Z_2)}
&=\frac1{2t^2}\brko{
\nabla(X_1-X_2)\cdot\nabla(X_1+X_2)}
\\
&\quad
-\frac1{2t^2}\brko{
\nabla(Y_1-Y_2)\cdot\nabla(Y_1+Y_2)} .
}
Then similarly,
\EQ{%\label{eq:ode-N2-lip-proof}
&\norm{N_2(t,Z_1)-N_2(t,Z_2)}_{E_{\rho'}(t)} \\
&\lsm
t^{-1-\al}\brko{1+(\rho-\rho')^{-1}}
\mbrko{\norm{Z_1(t)}_{E_\rho(t)}+\norm{Z_2(t)}_{E_\rho(t)}}
\norm{Z_1(t)-Z_2(t)}_{E_\rho(t)}.
}
This proves \eqref{eq:ode-N2-lip}.
\end{proof}
\subsection{The asymptotic fixed point}%\label{sec:ode-fixed-point}
Recall the Duhamel map $\TT$ from \eqref{eq:ode-Duhamel-map}.
For a time-dependent function $F$ on $[T,\I)$, define the norm
\EQ{%\label{eq:ode-tail-Q-def}
Q_F(t;\rho):=
\sup_{\tau\ge t}\norm{F(\tau)}_{E_\rho(\tau)}.
}

\begin{lem}[Absolute convergence of the Picard series]\label{lem:ode-picard-convergence}
Let $0\le\rho_-<\rho<\rho_+<\rho_1$. Let $\{c_m\}_{m\ge1}$ be nonnegative numbers such that
\EQ{%\label{eq:ode-picard-am-summability}
\sum_{m\ge1}c_mr^m<\I,
\qquad 0\le r<\I.
}
Let $\eta\ge0$, let $\omega\in L^1([T,\I))$ with $\omega\ge0$, and let $\{Z_n\}_{n\ge0}$ be time-dependent functions satisfying
\EQn{\label{eq:ode-picard-initial-bound}
Q_{Z_0}(T;\rho_+)\le M.
}
Assume also that, for every $n\ge0$ and $\rho_-\le\zeta<\rho_+$,
\EQ{
Z_n\in C([T,\I);\A^1_\zeta).
}
Assume that, for every $n\ge1$ and $\rho_-\le\zeta'<\zeta<\rho_+$,
\EQn{\label{eq:ode-picard-hypotheses}
Q_{Z_n}(t;\zeta')
&\le
\sum_{m\ge1}
\sum_{\substack{n_1,\ldots,n_m\ge0\\ n_1+\cdots+n_m=n-1}}
\eta c_mQ_{Z_{n_1}}(t;\zeta')\cdots Q_{Z_{n_m}}(t;\zeta')
\\
&\quad+
\sum_{m\ge1}
\sum_{\substack{n_1,\ldots,n_m\ge0\\ n_1+\cdots+n_m=n-1}}
c_m(\zeta-\zeta')^{-1}\int_t^\I \omega(s)
Q_{Z_{n_1}}(s;\zeta)\cdots Q_{Z_{n_m}}(s;\zeta)\dd s .
}
Then there exists $z_0=z_0(M,\{c_m\})>0$ such that, if
\EQn{\label{eq:ode-picard-smallness}
\eta+\frac{1}{\rho_+-\rho}\int_T^\I\omega(s)\dd s
\le z_0,
}
the series $\sum_{n\ge0}Z_n$ converges absolutely in $\mathcal X_T(\rho)$. Moreover,
\EQn{\label{eq:ode-picard-absolute-bound}
\sum_{n\ge0}\norm{Z_n}_{\mathcal X_T(\rho)}
\le C_M.
}
\end{lem}

\begin{proof}
Define
\EQ{
\mathcal P(r)=\sum_{m\ge1}c_mr^m,
\qquad
B(t)=\int_t^\I\omega(s)\dd s,
\qquad
z_\zeta(t)=\eta+\frac{B(t)}{\rho_+-\zeta}.
}
We prove the order-by-order estimate
\EQn{\label{eq:ode-picard-order-bound}
Q_{Z_n}(t;\zeta)
\le
b_n z_\zeta(t)^n,
\qquad n\ge0,
\quad \rho_-\le\zeta<\rho_+,
\quad t\ge T.
}
Choose the constants $b_n$ recursively by
\EQn{\label{eq:ode-picard-coefficient-recursion}
b_0=M,
\qquad
b_n=
2e\sum_{m\ge1}
\sum_{\substack{n_1,\ldots,n_m\ge0\\ n_1+\cdots+n_m=n-1}}
c_mb_{n_1}\cdots b_{n_m},
\qquad n\ge1.
}
To show that these estimates are summable, let
\EQ{
A(z)=\sum_{n\ge0}b_nz^n
}
be the formal generating series associated with \eqref{eq:ode-picard-coefficient-recursion}. The recursion is equivalent, as an identity of formal power series, to
\EQn{\label{eq:ode-picard-generating-equation}
A(z)=M+2ez\,\mathcal P(A(z)).
}
Indeed,
\EQ{
z\mathcal P(A(z))
=
\sum_{n\ge1}
\brko{
\sum_{m\ge1}
\sum_{\substack{n_1,\ldots,n_m\ge0\\ n_1+\cdots+n_m=n-1}}
c_mb_{n_1}\cdots b_{n_m}
}z^n.
}
Let $\mathcal F(z,A)=A-M-2ez\mathcal P(A)$. Then
\EQ{
\mathcal F(0,M)=0,
\qquad
\pd_A\mathcal F(0,M)=1.
}
By Lemma \ref{lem:real-analytic-ift}, \eqref{eq:ode-picard-generating-equation} has a real analytic solution near $z=0$ with $A(0)=M$. Uniqueness of the formal solution shows that its Taylor coefficients are exactly the constants $b_n$. Hence there exists $z_0=z_0(M,\{c_m\})>0$ such that
\EQn{\label{eq:ode-picard-coefficient-sum}
\sum_{n\ge0}b_nz^n=A(z)\le C_M,
\qquad 0\le z\le z_0.
}
We now establish \eqref{eq:ode-picard-order-bound} by induction. For $n=0$, it follows from \eqref{eq:ode-picard-initial-bound} and the monotonicity of the analytic norm. Assume that \eqref{eq:ode-picard-order-bound} holds up to order $n-1$. Fix $n_1,\ldots,n_m\ge0$ with $n_1+\cdots+n_m=n-1$. The first summand in \eqref{eq:ode-picard-hypotheses} is bounded by
\EQn{\label{eq:ode-same-radius-induction}
\eta c_mQ_{Z_{n_1}}(t;\zeta)\cdots Q_{Z_{n_m}}(t;\zeta)
\le
\eta c_mb_{n_1}\cdots b_{n_m}
z_\zeta(t)^{n_1+\cdots+n_m}
\le
c_mb_{n_1}\cdots b_{n_m}z_\zeta(t)^n,
}
because $\eta\le z_\zeta(t)$. For the second summand, let
\EQ{
\zeta_n=\zeta+\frac{\rho_+-\zeta}{n+1}.
}
Then $\zeta<\zeta_n<\rho_+$ and
\EQ{
(\zeta_n-\zeta)^{-1}=\frac{n+1}{\rho_+-\zeta},
\qquad
\rho_+-\zeta_n=\frac{n}{n+1}(\rho_+-\zeta).
}
Using \eqref{eq:ode-picard-hypotheses} at the radius $\zeta_n$ and applying the induction hypothesis at $\zeta_n$, we obtain
\EQn{\label{eq:ode-radius-gap-induction}
&c_m(\zeta_n-\zeta)^{-1}\int_t^\I\omega(s)
Q_{Z_{n_1}}(s;\zeta_n)\cdots Q_{Z_{n_m}}(s;\zeta_n)\dd s
\\
&\le
c_mb_{n_1}\cdots b_{n_m}
(\zeta_n-\zeta)^{-1}\int_t^\I\omega(s)z_{\zeta_n}(s)^{n-1}\dd s
\\
&\le
e\,c_mb_{n_1}\cdots b_{n_m}z_\zeta(t)^n.
}
The last inequality follows from
\EQ{
-\frac{\dd}{\dd s}z_{\zeta_n}(s)^n
=
\frac{n}{\rho_+-\zeta_n}\omega(s)z_{\zeta_n}(s)^{n-1}
}
and
\EQ{
(\zeta_n-\zeta)^{-1}\int_t^\I\omega(s)z_{\zeta_n}(s)^{n-1}\dd s
&=
\frac{\rho_+-\zeta_n}{n(\zeta_n-\zeta)}
\brko{z_{\zeta_n}(t)^n-\eta^n}
\\
&=
z_{\zeta_n}(t)^n-\eta^n
\\
&\le z_{\zeta_n}(t)^n
\le
\brko{1+\frac1n}^nz_\zeta(t)^n
\le ez_\zeta(t)^n.
}
Therefore \eqref{eq:ode-picard-hypotheses}, \eqref{eq:ode-same-radius-induction}, and \eqref{eq:ode-radius-gap-induction} imply
\EQ{%\label{eq:ode-summed-induction}
Q_{Z_n}(t;\zeta)
&\le
(1+e)
\sum_{m\ge1}
\sum_{\substack{n_1,\ldots,n_m\ge0\\ n_1+\cdots+n_m=n-1}}
c_mb_{n_1}\cdots b_{n_m}z_\zeta(t)^n
\\
&\le b_nz_\zeta(t)^n,
}
where the last inequality follows from \eqref{eq:ode-picard-coefficient-recursion}. This completes the induction at order $n$. By \eqref{eq:ode-picard-smallness}, $z_\rho(T)\le z_0$. Since $z_\rho(t)\le z_\rho(T)$ for $t\ge T$, \eqref{eq:ode-picard-order-bound} and \eqref{eq:ode-picard-coefficient-sum} give
\EQ{
\sum_{n\ge0}Q_{Z_n}(T;\rho)
\le
\sum_{n\ge0}b_nz_{\rho}(T)^n
\le C_M.
}
At the radius $\rho$, the last estimate and the Weierstrass test give uniform convergence in $\A^1_\rho$ on $[T,\I)$ of
\EQ{
\sum_{n\ge0}t^\al\re Z_n(t)
\qquad\hbox{and}\qquad
\sum_{n\ge0}t^{2\al-1}\im Z_n(t).
}
Since every term is continuous, the sums are continuous. Hence $\sum_n Z_n$ converges absolutely in $\mathcal X_T(\rho)$, and \eqref{eq:ode-picard-absolute-bound} follows.
\end{proof}

\begin{prop}[Fixed point for the phase correction]\label{prop:ode-fixed-point}
Under the assumptions of Theorem \ref{thm:ode}, set
\EQ{
\rho_-:=\frac{\rho_*+\rho_1}{2},
\qquad
\rho:=\frac{2\rho_-+\rho_1}{3},
\qquad
\rho_+:=\frac{\rho+\rho_1}{2}.
}
There exist $\ep_0>0$, $T\ge2$, and $R_0>1$ such that, if $0<\ep\le\ep_0$, there is a final-state correction $Z\in\mathcal X_T(\rho)$ satisfying
\EQ{
Z=\TT Z\quad\hbox{in }\mathcal X_T(\rho_-),
\qquad
\norm{Z}_{\mathcal X_T(\rho)}\le R_0.
}
Moreover, any two final-state corrections for the same final profile $W$ and the corresponding map $\TT$ agree on $[T_*,\I)$ for some $T_*\ge2$.
\end{prop}

\begin{proof}
For $0\le\zeta\le\rho_+$, the monotonicity of the analytic norm gives
\EQ{
\norm{f}_{\A^\sigma_\zeta}\le \norm{f}_{\A^\sigma_{\rho_+}},
\qquad
Q_F(t;\zeta)\le Q_F(t;\rho_+),
\qquad 0\le\zeta\le\rho_+.
}
We write the Duhamel map in the form
\EQn{\label{eq:ode-duhamel-decomposition}
\TT Z=D_0+\mathcal A_0(Z)+\mathcal A_1(Z),
}
where
\EQ{%\label{eq:ode-duhamel-parts}
D_0(t)&=-\int_t^\I U(t,s)N_0(s)\dd s,
\\
\mathcal A_0(Z)(t)&=-\int_t^\I N_{\rm tri}(s,Z(s))\dd s,
\\
\mathcal A_1(Z)(t)&=-\int_t^\I U(t,s)\brko{N_1(s,Z(s))+N_2(s,Z(s))}\dd s
\\
&\quad
-\int_t^\I\brko{U(t,s)-I}N_{\rm tri}(s,Z(s))\dd s.
}
We construct the fixed point as a formal Picard series
\EQ{
Z=\sum_{n\ge0}Z_n.
}
By \eqref{eq:ode-duhamel-decomposition} and $Z=\TT Z$, this series satisfies the formal identity
\EQn{\label{eq:ode-picard-formal-equation}
\sum_{n\ge0}Z_n
=
D_0+
\mathcal A_0\brko{\sum_{n\ge0}Z_n}
+
\mathcal A_1\brko{\sum_{n\ge0}Z_n}.
}
The exponential part in $\mathcal A_0$ is grouped by the index $1+n_1+\cdots+n_m$:
\EQ{%\label{eq:ode-exponential-layer-expansion}
e^{p\re\brko{\sum_{\ell\ge0}Z_\ell(t)}}-1
=
\sum_{n\ge1}
\sum_{m\ge1}\frac{p^m}{m!}
\sum_{\substack{n_1,\ldots,n_m\ge0\\ n_1+\cdots+n_m=n-1}}
\re Z_{n_1}(t)\cdots\re Z_{n_m}(t).
}
For $n\ge1$ define
\EQn{\label{eq:ode-picard-forcing}
N_{{\rm tri},n}(t)
&=-i t^{-\al}|W|^p
\sum_{m\ge1}\frac{p^m}{m!}
\sum_{\substack{n_1,\ldots,n_m\ge0\\ n_1+\cdots+n_m=n-1}} \re Z_{n_1}(t)\cdots\re Z_{n_m}(t),
\\
N_{1,n}(t)
&=\frac{i}{t^2}\nabla\Phi_0(t)\cdot\nabla Z_{n-1}(t),
\\
N_{2,n}(t)
&=\sum_{n_1+n_2=n-1}\frac{i}{2t^2}
\nabla Z_{n_1}(t)\cdot\nabla Z_{n_2}(t).
}
With this notation, the right hand side of \eqref{eq:ode-picard-formal-equation} expands as
\EQn{\label{eq:ode-picard-rhs-layer}
&D_0(t)
+\mathcal A_0\brko{\sum_{\ell\ge0}Z_\ell}(t)
+\mathcal A_1\brko{\sum_{\ell\ge0}Z_\ell}(t)
\\
&=D_0(t)
+i\sum_{m\ge1}\frac{p^m}{m!}
\sum_{n_1,\ldots,n_m\ge0} \int_t^\I s^{-\al}|W|^p
\re Z_{n_1}(s)\cdots\re Z_{n_m}(s)\dd s
\\
&\quad
-\sum_{\ell\ge0}
\int_t^\I U(t,s)
\frac{i}{s^2}\nabla\Phi_0(s)\cdot\nabla Z_\ell(s)\dd s
\\
&\quad
-\sum_{\ell_1,\ell_2\ge0}
\int_t^\I U(t,s)
\frac{i}{2s^2}
\nabla Z_{\ell_1}(s)\cdot\nabla Z_{\ell_2}(s)\dd s
\\
&\quad
+i\sum_{m\ge1}\frac{p^m}{m!}
\sum_{n_1,\ldots,n_m\ge0} \int_t^\I\brko{U(t,s)-I}
\brko{s^{-\al}|W|^p\re Z_{n_1}(s)\cdots\re Z_{n_m}(s)}\dd s
\\
&=D_0(t)+
\sum_{n\ge1}
\brko{
-\int_t^\I N_{{\rm tri},n}(s)\dd s 
-\int_t^\I U(t,s)\brko{N_{1,n}(s)+N_{2,n}(s)}\dd s
\\
&\qquad\qquad\quad
-\int_t^\I\brko{U(t,s)-I}N_{{\rm tri},n}(s)\dd s
}.
}
By \eqref{eq:ode-picard-formal-equation} and \eqref{eq:ode-picard-rhs-layer},
\EQ{
Z_0=D_0,
}
and, for $n\ge1$,
\EQn{\label{eq:ode-picard-series}
Z_n(t)
&=-\int_t^\I N_{{\rm tri},n}(s)\dd s
 -\int_t^\I U(t,s)\brko{N_{1,n}(s)+N_{2,n}(s)}\dd s
\\
&\quad-\int_t^\I\brko{U(t,s)-I}N_{{\rm tri},n}(s)\dd s .
}
Write the three terms in \eqref{eq:ode-picard-series} as
\EQn{\label{eq:ode-picard-three-parts}
J_{{\rm tri},n}(t)&=-\int_t^\I N_{{\rm tri},n}(s)\dd s,
\\
J_{12,n}(t)&=-\int_t^\I U(t,s)\brko{N_{1,n}(s)+N_{2,n}(s)}\dd s,
\\
J_{U,n}(t)&=-\int_t^\I\brko{U(t,s)-I}N_{{\rm tri},n}(s)\dd s .
}
Thus $Z_n=J_{{\rm tri},n}+J_{12,n}+J_{U,n}$. Fix $\rho_-\le\zeta'<\zeta<\rho_+$. We first estimate $J_{{\rm tri},n}$. Fix $\tau\ge t$. For $s\ge\tau$, by \eqref{eq:ode-picard-forcing}, \eqref{eq:eps-small}, \eqref{eq:fourier-algebra}, and \eqref{eq:instant-norm},
\EQn{\label{eq:ode-picard-tri-integral-bound}
\norm{\jb{\nabla}\im N_{{\rm tri},n}(s)}_{\A^0_{\zeta'}}
&\lsm_R\ep
\sum_{m\ge1}
\sum_{\substack{n_1,\ldots,n_m\ge0\\ n_1+\cdots+n_m=n-1}}
\frac{p^m}{m!}s^{-\al}
\norm{\re Z_{n_1}(s)\cdots\re Z_{n_m}(s)}_{\A^1_{\zeta'}}
\\
&\lsm_R\ep
\sum_{m\ge1}
\sum_{\substack{n_1,\ldots,n_m\ge0\\ n_1+\cdots+n_m=n-1}}
\frac{(C p)^m}{m!}s^{-\al}
\\
&\quad\cdot
\norm{\jb{\nabla}\re Z_{n_1}(s)}_{\A^0_{\zeta'}}\cdots\norm{\jb{\nabla}\re Z_{n_m}(s)}_{\A^0_{\zeta'}}
\\
&\lsm_R\ep
\sum_{m\ge1}
\sum_{\substack{n_1,\ldots,n_m\ge0\\ n_1+\cdots+n_m=n-1}}
\frac{(C p)^m}{m!}s^{-(m+1)\al} 
Q_{Z_{n_1}}(\tau;\zeta')\cdots Q_{Z_{n_m}}(\tau;\zeta').
}
Since $N_{{\rm tri},n}$ and $J_{{\rm tri},n}$ are purely imaginary, \eqref{eq:ode-picard-three-parts}, \eqref{eq:instant-norm}, and \eqref{eq:ode-picard-tri-integral-bound} imply, for every $\tau\ge t$,
\EQ{
\norm{J_{{\rm tri},n}(\tau)}_{E_{\zeta'}(\tau)}
&=
\tau^{2\al-1}
\norm{\jb{\nabla}\im J_{{\rm tri},n}(\tau)}_{\A^0_{\zeta'}}
\\
&\le
\tau^{2\al-1}
\int_\tau^\I
\norm{\jb{\nabla}\im N_{{\rm tri},n}(s)}_{\A^0_{\zeta'}}\dd s
\\
&\lsm_R\ep
\sum_{m\ge1}
\sum_{\substack{n_1,\ldots,n_m\ge0\\ n_1+\cdots+n_m=n-1}}
\frac{(C p)^m}{m!}
\tau^{2\al-1}\int_\tau^\I s^{-(m+1)\al}\dd s
\\
&\quad\cdot
Q_{Z_{n_1}}(\tau;\zeta')\cdots Q_{Z_{n_m}}(\tau;\zeta')
\\
&\lsm_R\ep
\sum_{m\ge1}
\sum_{\substack{n_1,\ldots,n_m\ge0\\ n_1+\cdots+n_m=n-1}}
\frac{(C p)^m}{m!}
\tau^{-(m-1)\al} 
Q_{Z_{n_1}}(\tau;\zeta')\cdots Q_{Z_{n_m}}(\tau;\zeta')
\\
&\lsm_R\ep
\sum_{m\ge1}
\sum_{\substack{n_1,\ldots,n_m\ge0\\ n_1+\cdots+n_m=n-1}}
\frac{(C p)^m}{m!} 
Q_{Z_{n_1}}(\tau;\zeta')\cdots Q_{Z_{n_m}}(\tau;\zeta').
}
Taking the supremum over $\tau\ge t$,
\EQn{\label{eq:ode-picard-tri-Q-bound}
Q_{J_{{\rm tri},n}}(t;\zeta')
\lsm_R\ep
\sum_{m\ge1}
\sum_{\substack{n_1,\ldots,n_m\ge0\\ n_1+\cdots+n_m=n-1}}
\frac{(C p)^m}{m!}
Q_{Z_{n_1}}(t;\zeta')\cdots Q_{Z_{n_m}}(t;\zeta').
}
For the term $J_{12,n}$, fix $\tau\ge t$. Since $\al>2\al-1$, by \eqref{eq:instant-norm} and \eqref{eq:U-isometry}, for $s\ge\tau$,
\EQn{\label{eq:ode-U-E-weight-conversion}
\norm{U(\tau,s)F}_{E_{\zeta'}(\tau)}
&\lsm\tau^{\al}\norm{\jb{\nabla}U(\tau,s)F}_{\A^0_{\zeta'}}
=\tau^{\al}\norm{\jb{\nabla}F}_{\A^0_{\zeta'}}
\\
&\le\tau^{\al}s^{1-2\al}\norm{F}_{E_{\zeta'}(s)}
\le s^{1-\al}\norm{F}_{E_{\zeta'}(s)}.
}
Applying \eqref{eq:ode-U-E-weight-conversion} with $F=N_{1,n}(s)+N_{2,n}(s)$ and using \eqref{eq:ode-N1-size} and \eqref{eq:ode-N2-size}, the factor $s^{-1-\al}$ in these two estimates becomes $s^{-2\al}$. Therefore
\EQ{%\label{eq:ode-picard-N12-tau-bound}
\norm{J_{12,n}(\tau)}_{E_{\zeta'}(\tau)}
&\lsm_R(\zeta-\zeta')^{-1}
\int_\tau^\I s^{-2\al}Q_{Z_{n-1}}(s;\zeta)\dd s
\\
&\quad+(\zeta-\zeta')^{-1}
\sum_{n_1+n_2=n-1}
\int_\tau^\I s^{-2\al}
Q_{Z_{n_1}}(s;\zeta)Q_{Z_{n_2}}(s;\zeta)\dd s .
}
Taking the supremum over $\tau\ge t$,
\EQn{\label{eq:ode-picard-N12-Q-bound}
Q_{J_{12,n}}(t;\zeta')
&\lsm_R(\zeta-\zeta')^{-1}
\int_t^\I s^{-2\al}Q_{Z_{n-1}}(s;\zeta)\dd s
\\
&\quad+(\zeta-\zeta')^{-1}
\sum_{n_1+n_2=n-1}
\int_t^\I s^{-2\al}
Q_{Z_{n_1}}(s;\zeta)Q_{Z_{n_2}}(s;\zeta)\dd s .
}
For $J_{U,n}$, fix $\tau\ge t$. For $s\ge\tau$ and $m\ge1$, by \eqref{eq:nu-choice},
\EQ{
\tau^\al\tau^{-\nu}s^{-(m+1)\al}
\le s^{-m\al-\nu}
\le s^{-\al-\nu}.
}
Thus, by \eqref{eq:UminusI}, $2\nu<1$, \eqref{eq:fourier-algebra}, and \eqref{eq:ode-picard-forcing},
\EQ{%\label{eq:ode-picard-Utri-tau-bound}
\norm{J_{U,n}(\tau)}_{E_{\zeta'}(\tau)}
&\lsm_R\ep
\sum_{m\ge1}
\sum_{\substack{n_1,\ldots,n_m\ge0\\ n_1+\cdots+n_m=n-1}}
\frac{p^m}{m!}(\zeta-\zeta')^{-1}
\tau^\al\tau^{-\nu}
\int_\tau^\I s^{-\al}
\\
&\quad\cdot
\norm{\re Z_{n_1}(s)\cdots\re Z_{n_m}(s)}_{\A^1_{\zeta}}\dd s
\\
&\lsm_R\ep
\sum_{m\ge1}
\sum_{\substack{n_1,\ldots,n_m\ge0\\ n_1+\cdots+n_m=n-1}}
\frac{(C p)^m}{m!}(\zeta-\zeta')^{-1}
\int_\tau^\I s^{-\al-\nu}
\\
&\quad\cdot
Q_{Z_{n_1}}(s;\zeta)\cdots Q_{Z_{n_m}}(s;\zeta)\dd s .
}
Taking the supremum over $\tau\ge t$,
\EQn{\label{eq:ode-picard-Utri-Q-bound}
Q_{J_{U,n}}(t;\zeta') &
\lsm_R\ep
\sum_{m\ge1}
\sum_{\substack{n_1,\ldots,n_m\ge0\\ n_1+\cdots+n_m=n-1}} \frac{(C p)^m}{m!}(\zeta-\zeta')^{-1}
\\
&\quad\cdot \int_t^\I s^{-\al-\nu}
Q_{Z_{n_1}}(s;\zeta)\cdots Q_{Z_{n_m}}(s;\zeta)\dd s.
}
Since $\nu<1/2<\al$ and $s\ge1$,
\EQ{
s^{-2\al}\le s^{-\al-\nu}.
}
Choose $C_p\ge1$ so large that, with
\EQ{
c_m=\frac{C_p^m}{m!},
}
one has
\EQn{\label{eq:ode-Cp-choice}
\frac{(C p)^m}{m!}\le c_m,
\quad m\ge1,
\qquad
1\le c_1,c_2.
}
Then $\sum_{m\ge1}c_mr^m<\I$ for every fixed $0\le r<\I$. Fix $C_R\ge1$, depending only on $R$, so that the implicit constants in \eqref{eq:ode-picard-tri-Q-bound}, \eqref{eq:ode-picard-N12-Q-bound}, \eqref{eq:ode-picard-Utri-Q-bound}, and \eqref{eq:N0-duhamel-bound} are bounded by $C_R$, and set
\EQ{
\omega(s)=C_Rs^{-\al-\nu}.
}
Combining these estimates with \eqref{eq:ode-Cp-choice} and using $\ep\le1$, we obtain, for $n\ge1$ and $\rho_-\le\zeta'<\zeta<\rho_+$,
\EQn{\label{eq:ode-picard-coeff-bound}
Q_{Z_n}(t;\zeta')
&\le
C_R\ep
\sum_{m\ge1}
\sum_{\substack{n_1,\ldots,n_m\ge0\\ n_1+\cdots+n_m=n-1}}
c_mQ_{Z_{n_1}}(t;\zeta')\cdots Q_{Z_{n_m}}(t;\zeta')
\\
&\quad+
\sum_{m\ge1}
\sum_{\substack{n_1,\ldots,n_m\ge0\\ n_1+\cdots+n_m=n-1}}
c_m(\zeta-\zeta')^{-1}\int_t^\I\omega(s)
Q_{Z_{n_1}}(s;\zeta)\cdots Q_{Z_{n_m}}(s;\zeta)\dd s .
}
Since $2\al>1$ and $\nu>1-\al$,
\EQn{\label{eq:ode-omega-tail}
B(T):=\int_T^\I\omega(s)\dd s
=\frac{C_R}{\al+\nu-1}T^{1-\al-\nu}.
}
By \eqref{eq:N0-duhamel-bound},
\EQn{\label{eq:ode-D0-bound-outer}
Q_{Z_0}(T;\rho_+)\le C_R.
}
Let $z_0$ and $C_M$ be the constants in Lemma \ref{lem:ode-picard-convergence} with
\EQ{
\eta=C_R\ep,
\qquad
M=C_R .
}
Fix $R_0>1$ such that $C_M\le R_0$. Choose $\ep_0>0$ and then $T\ge2$ so large that
\EQn{\label{eq:ode-fixed-point-smallness}
0<\ep_0\le1,
\qquad
C_R\ep_0+
\frac{C_R}{(\al+\nu-1)(\rho_+-\rho)}T^{1-\al-\nu}
\le z_0,
}
The Duhamel formulas defining the Picard terms give
\EQ{
Z_n\in C([T,\I);\A^1_\zeta),
\qquad n\ge0,
\qquad \rho_-\le\zeta<\rho_+.
}
The bounds \eqref{eq:ode-picard-coeff-bound}, \eqref{eq:ode-omega-tail}, \eqref{eq:ode-D0-bound-outer}, and \eqref{eq:ode-fixed-point-smallness} verify the hypotheses of Lemma \ref{lem:ode-picard-convergence}. Hence the Picard series $\sum_{n\ge0}Z_n$, whose terms are defined by \eqref{eq:ode-picard-series}, converges absolutely in $\mathcal X_T(\rho)$. Since $\rho_-<\rho$, the monotonicity of the analytic norm gives convergence in $\mathcal X_T(\rho_-)$. Let
\EQ{
Z=\sum_{n\ge0}Z_n.
}
The absolute convergence at $\rho$ and the summable estimate \eqref{eq:ode-picard-coeff-bound} at the radius $\rho_-$ justify summation of \eqref{eq:ode-picard-series}. Thus
\EQ{
Z=D_0+\mathcal A_0(Z)+\mathcal A_1(Z)=\TT Z .
}
This identity holds in $\mathcal X_T(\rho_-)$ and hence in $\mathcal X_T(\rho')$ for every $\rho_*<\rho'<\rho_-$. Thus $Z$ is a final-state correction in the sense of Definition \ref{def:ode-asymptotic}.
By \eqref{eq:ode-picard-absolute-bound},
\EQ{
\norm{Z}_{\mathcal X_T(\rho)}\le R_0.
}

It remains to prove the uniqueness assertion. Let $Z_j$, $j=1,2$, be two final-state corrections as in the statement. By Definition \ref{def:ode-asymptotic}, we may choose $T_0\ge2$ and
\EQ{
\rho_*<\zeta_-<\zeta_+<\rho_1
}
such that
\EQ{
Z_j\in\mathcal X_{T_0}(\zeta_+),
\qquad
Z_j=\TT Z_j
\quad\hbox{in }\mathcal X_{T_0}(\zeta')
\quad\hbox{for every }\rho_*<\zeta'<\zeta_+.
}
Set
\EQ{%\label{eq:ode-uniqueness-tail-def}
M&=1+\norm{Z_1}_{\mathcal X_{T_0}(\zeta_+)}
+\norm{Z_2}_{\mathcal X_{T_0}(\zeta_+)},
\\
\omega_M(s)&=Ms^{-2\al}+\ep s^{-\al-\nu}.
}
Decrease $\ep_0$ so that
\EQn{\label{eq:ode-uniqueness-epsilon-final}
2C_R\ep_0\le\frac12.
}
For every $T_1\ge T_0$,
\EQ{
\norm{Z_1}_{\mathcal X_{T_1}(\zeta_+)}
+\norm{Z_2}_{\mathcal X_{T_1}(\zeta_+)}
\le M.
}
Since $2\al>1$ and $\al+\nu>1$, we may choose $T_1\ge T_0$ so that
\EQn{\label{eq:ode-uniqueness-small-final}
e^{CT_1^{-\al}M}\le2,
\qquad
\theta_{T_1}:=
\frac{C_{R,M}e}{\zeta_+-\zeta_-}
\int_{T_1}^\I\omega_M(s)\dd s
<1.
}
By \eqref{eq:ode-Ntri-lip}, \eqref{eq:ode-uniqueness-epsilon-final}, and \eqref{eq:ode-uniqueness-small-final}, for $\zeta_-\le\zeta'<\zeta_+$,
\EQn{\label{eq:ode-A0-diff-final}
Q_{\mathcal A_0(Z_1)-\mathcal A_0(Z_2)}(t;\zeta')
\le C_R\ep e^{CT_1^{-\al}M}Q_{Z_1-Z_2}(t;\zeta')
\le\frac12Q_{Z_1-Z_2}(t;\zeta').
}
By \eqref{eq:ode-N1-lip}, \eqref{eq:ode-N2-lip}, and \eqref{eq:UminusI} applied to $N_{\rm tri}(Z_1)-N_{\rm tri}(Z_2)$, for $\zeta_-\le\zeta'<\zeta\le\zeta_+$,
\EQn{\label{eq:ode-A1-diff-final}
Q_{\mathcal A_1(Z_1)-\mathcal A_1(Z_2)}(t;\zeta')
\le C_{R,M}(\zeta-\zeta')^{-1}
\int_t^\I\omega_M(s)Q_{Z_1-Z_2}(s;\zeta)\dd s.
}
Since $Z_j=\TT Z_j$, combining \eqref{eq:ode-A0-diff-final} and \eqref{eq:ode-A1-diff-final} and absorbing the first term gives
\EQn{\label{eq:ode-uniqueness-step-final}
Q_{Z_1-Z_2}(t;\zeta')
\le
C_{R,M}(\zeta-\zeta')^{-1}
\int_t^\I\omega_M(s)Q_{Z_1-Z_2}(s;\zeta)\dd s.
}
For $n\ge1$, let
\EQ{
\zeta_j=\zeta_-+\frac{j}{n}(\zeta_+-\zeta_-),
\qquad j=0,\ldots,n.
}
Iterating \eqref{eq:ode-uniqueness-step-final} on the pairs $(\zeta_{j-1},\zeta_j)$ gives
\EQn{\label{eq:ode-uniqueness-iteration}
\norm{Z_1-Z_2}_{\mathcal X_{T_1}(\zeta_-)}
\le
\frac{1}{n!}
\brko{\frac{C_{R,M}n}{\zeta_+-\zeta_-}
\int_{T_1}^\I\omega_M(s)\dd s}^{n}
\norm{Z_1-Z_2}_{\mathcal X_{T_1}(\zeta_+)}.
}
Using $n!\ge(n/e)^n$ and \eqref{eq:ode-uniqueness-small-final} in \eqref{eq:ode-uniqueness-iteration}, we obtain
\EQ{
\norm{Z_1-Z_2}_{\mathcal X_{T_1}(\zeta_-)}
\le
\theta_{T_1}^n
\norm{Z_1-Z_2}_{\mathcal X_{T_1}(\zeta_+)}.
}
Letting $n\ra\I$ gives $Z_1=Z_2$ on $[T_1,\I)$.
\end{proof}

\subsection{Construction, convergence, and uniqueness}%\label{sec:ode-construction}

Let $Z$ be the fixed point from Proposition \ref{prop:ode-fixed-point}, set $\Phi=\Phi_0+Z$, define
\EQ{%\label{eq:a-construction}
a(t,y)=e^{\Phi(t,y)},
}
and define $u$ from $a$ by the ray transform \eqref{eq:ray-transform}. On every compact interval $I\subset[T,\I)$, Lemma \ref{lem:fourier-algebra}, \eqref{eq:wiener-L2-product}, and the $L^2$ isometry of the ray transform give
\EQ{
a=W\,e^{-i\Lam|W|^p}e^Z
&\in C(I;L^2\cap\A^1_{\rho_*}).
}
Moreover, $u\in C(I;L^2)$ solves \eqref{eq:nls} on $[T,\I)$, satisfying
\EQ{
\norm{u(t)}_{L^2_x}=\norm{a(t)}_{L^2_y}
&\le e^{\norm{X(t)}_{L^\I_y}}\norm{W}_{L^2_y}.
}
Proposition \ref{prop:ode-fixed-point} shows that $Z$ is a final-state correction, so the corresponding ray profile belongs to the ODE-type solution class of Definition \ref{def:ode-asymptotic}.

Since $W\in L^2\cap\A^6_{\rho_0}$, the Wiener embedding and interpolation give $W\in L^q$ for $2\le q\le\I$.
The fixed-point bound in Proposition \ref{prop:ode-fixed-point}, the definition of $\mathcal X_T$, and \eqref{eq:wiener-embedding} give
\EQn{\label{eq:ode-XY-Linfty-decay}
\norm{X(t)}_{L^\I_y}
\lsm_Rt^{-\al}, \quad
\norm{Y(t)}_{L^\I_y}
\lsm_Rt^{-(2\al-1)}, \quad
\norm{Z(t)}_{L^\I_y}
\lsm_Rt^{-(2\al-1)}.
}
Here $1/2<\al<1$, so $0<2\al-1<\al$. Since $|e^z-1|\le e^{|\re z|}|z|$, \eqref{eq:ray-Lq-scaling} and \eqref{eq:ode-XY-Linfty-decay} give
\EQ{
\norm{u(t)-u_{\rm ODE}(t)}_{L^q_x}
&\le t^{-d\brko{\frac12-\frac1q}}\norm{W}_{L^q_y}
e^{\norm{X(t)}_{L^\I_y}}\norm{Z(t)}_{L^\I_y}
\\
&\lsm_R\norm{W}_{L^q_y}t^{-d\brko{\frac12-\frac1q}-(2\al-1)}.
}
Moreover,
\EQ{
|u(t,x)|-|u_{\rm ODE}(t,x)|
=t^{-\frac d2}|W(x/t)|\brko{e^{X(t,x/t)}-1},
}
and hence
\EQ{
\norm{|u(t)|-|u_{\rm ODE}(t)|}_{L^q_x}
\lsm_R\norm{W}_{L^q_y}t^{-d\brko{\frac12-\frac1q}-\al}.
}
Finally, \eqref{eq:ray-Lq-scaling} gives
\EQ{
\norm{u_{\rm ODE}(t)}_{L^q_x}=t^{-d\brko{\frac12-\frac1q}}\norm{W}_{L^q_y},
}
and therefore
\EQ{
\absb{t^{d\brko{\frac12-\frac1q}}\norm{u(t)}_{L^q_x}-\norm{W}_{L^q_y}}
&\le t^{d\brko{\frac12-\frac1q}}\norm{|u(t)|-|u_{\rm ODE}(t)|}_{L^q_x}
\\
&\lsm_R\norm{W}_{L^q_y}t^{-\al}.
}
This proves \eqref{eq:ode-Lq-remainder}--\eqref{eq:ode-Lq-leading}, and hence the existence and convergence assertions of Theorem \ref{thm:ode}.

If $u_1,u_2$ belong to the ODE-type solution class with the same logarithm $\Psi$, Definition \ref{def:ode-asymptotic} gives corrections $Z_1,Z_2$ that are fixed points of the same Duhamel map at every smaller analytic radius. Proposition \ref{prop:ode-fixed-point} therefore gives $Z_1=Z_2$, and hence $u_1=u_2$, for all sufficiently large times.

Every solution in the ODE-type solution class satisfies the $L^2$ bound above. Since $p<2/d<4/d$, Tsutsumi's $L^2$ Cauchy theory \cite{Tsutsumi1987} extends every such solution uniquely to $[0,\I)$. Uniqueness of the $L^2$ flow propagates the large-time equality backward and shows that constructions from different choices of the starting time $T$ have the same value at $t=0$. Thus the wave operator \eqref{eq:wave-operator-def} is well defined, and Theorem \ref{thm:ode} follows.

\section{Fuchsian profiles for \texorpdfstring{$0<p\le1/d$}{0 < p <= 1/d}}\label{sec:eikonal-proof}

Throughout this section, assume $0<p\le1/d$ and set
\EQn{\label{eq:gamma-def}
\gamma=\frac{1-\al}{\al}
=\frac{2-dp}{dp}.
}
For $t>0$, let
\EQ{
\tau=t^{-\al}.
}

\begin{defn}[Truncated Fuchsian profiles]\label{def:eikonal-profile}
For an integer $\mu>2\gamma$, set
\EQn{\label{eq:eikonal-exponent-set}
E_{\gamma,\mu}
=\fbrk{\ell_1+\ell_2\gamma:
\ell_1,\ell_2\in\N,\ \ell_1+\ell_2\ge1}\cap(0,\mu).
}
For a complex-valued function $L$, set
\EQ{
S_*(L):=\frac1{1-\al}e^{p\re L}.
}
Given nonnegative integers $J_q^\Omega,J_q^S$ and coefficients $\Omega_{q,j}(y),\Sigma_{q,j}(y)$ for $q\in E_{\gamma,\mu}$, set
\EQ{
\Omega_q(r,y)=\sum_{j=0}^{J_q^\Omega}\Omega_{q,j}(y)r^j,
\qquad
\Sigma_q(r,y)=\sum_{j=0}^{J_q^S}\Sigma_{q,j}(y)r^j.
}
Let
\EQ{
\Omega^{(\mu)}(\tau,y)
&=\sum_{q\in E_{\gamma,\mu}}\tau^q\Omega_q(\log\tau,y),
\\
\Sigma^{(\mu)}(\tau,y)
&=\sum_{q\in E_{\gamma,\mu}}\tau^q\Sigma_q(\log\tau,y).
}
For a real-valued function $C_+$, the corresponding truncated Fuchsian profile is the pair
\EQn{\label{eq:eikonal-LS-expansion}
L^{(\mu)}(\tau,y)&=\Psi(y)+\Omega^{(\mu)}(\tau,y),
\\
S^{(\mu)}(\tau,y)&=S_*(\Psi)(y)+C_+(y)\tau^\gamma
+\Sigma^{(\mu)}(\tau,y).
}
\end{defn}

\begin{remark}[Final datum and the homogeneous mode]
Since $t^{1-\al}\tau^\gamma=1$, the profile in Definition \ref{def:eikonal-profile} satisfies
\EQ{
\exp\mbrko{L^{(\mu)}(\tau)-it^{1-\al}S^{(\mu)}(\tau)}
=e^{\Psi-iC_+}\exp\mbrko{\Omega^{(\mu)}(\tau)-it^{1-\al}\brko{S_*(\Psi)+\Sigma^{(\mu)}(\tau)}}.
}
Thus the prescribed bounded nonlogarithmic term in the full logarithmic phase and its corresponding factor are
\EQ{
\Theta_+:=\Psi-iC_+,
\qquad
W_{\rm full}:=e^{\Theta_+}=e^\Psi e^{-iC_+}.
}
The pair $(\Psi,C_+)$ is an auxiliary parametrization of $\Theta_+$. Then a natural question appears: if $\Theta_+$ can be written in two different forms $\Theta_+=\Psi_1-iC_{+,1}=\Psi_2-iC_{+,2}$, then the coefficients in Fuchsian profile constructed from $(\Psi_1,C_{+,1})$ and $(\Psi_2,C_{+,2})$ are different, thus the solution may also differ. However, Proposition \ref{prop:eikonal-final-datum-invariance} below shows that the solutions constructed from $(\Psi_1,C_{+,1})$ and $(\Psi_2,C_{+,2})$ are actually identical. Thus $C_+$ is not an additional datum once $\Theta_+$ is prescribed. Varying $C_+$ with $\Psi$ fixed changes $\Theta_+$. Theorem \ref{thm:main} uses the normalization $C_+=0$, for which $\Theta_+=\Psi$ and $W_{\rm full}=W$.
\end{remark}

The coefficients in Definition \ref{def:eikonal-profile} are constructed in Proposition \ref{prop:eikonal-prep}. They are determined by $(\Psi,C_+)$. Explicit lower-order coefficients are listed in the Appendix.

\begin{defn}[Fuchsian solution class]\label{def:eikonal-asymptotic}
Fix $0<\rho_*<\rho_1$, an integer $\mu>2\gamma$, and an exponent $\kappa$ satisfying
\EQn{\label{eq:kappa-choice}
1-\al<\kappa<\al(\mu+1)-1.
}
For $T\ge2$ and $0\le\rho<\rho_1$, define
\EQ{
\mathcal Z_T^\kappa(\rho)
:=\fbrk{Z\in C([T,\I);\A^4_\rho):
\norm{Z}_{\mathcal Z_T^\kappa(\rho)}<\I},
}
where
\EQ{%\label{eq:eikonal-original-solution-norm}
\norm{Z}_{\mathcal Z_T^\kappa(\rho)}
:=\sup_{t\ge T}\mbrko{
t^\kappa\norm{\re Z(t)}_{\A^4_\rho}
+t^{\kappa+\al-1}\norm{\im Z(t)}_{\A^4_\rho}}.
}
The Duhamel map $\TT_\mu$ is defined in \eqref{eq:eikonal-Duhamel-map} below. We call $Z$ a final-state correction associated with a truncated Fuchsian profile $(L^{(\mu)},S^{(\mu)})$ if there exist $T\ge2$ and $\rho_*<\rho<\rho_1$ such that
\EQn{\label{eq:eikonal-final-state-correction}
Z\in\mathcal Z_T^\kappa(\rho),
\qquad
Z=\TT_\mu Z
\quad\hbox{in }\mathcal Z_T^\kappa(\rho')
\quad\hbox{for every }\rho_*<\rho'<\rho.
}
The Fuchsian solution class consists of the ray profiles
\EQ{%\label{eq:eikonal-phase-class-form}
a(t,y)=\exp\mbrko{L^{(\mu)}(t^{-\al},y)-it^{1-\al}S^{(\mu)}(t^{-\al},y)+Z(t,y)}.
}
\end{defn}

\begin{thm}[Modified wave operator for a truncated Fuchsian profile]\label{thm:eikonal}
Let $d\ge1$, $0<p\le1/d$, and let $\al$ and $\gamma$ be defined by \eqref{eq:alpha-def} and \eqref{eq:gamma-def}. Fix $0<\rho_*<\rho_1<\rho_0$, $R>0$, an integer $\mu>2\gamma$, and an exponent $\kappa$ satisfying $1-\al<\kappa<\al(\mu+1)-1$. There exists $0<\ep_0=\ep_0(d,p,\rho_*,\rho_1,\rho_0,R,\mu,\kappa)\le1$ such that the following holds whenever $0<\ep\le\ep_0$, $\Psi$ is the logarithm of an admissible final profile $W=e^\Psi$ of size $\ep$ in the sense of Definition \ref{def:final-data}, and $C_+\in\A^6_{\rho_0}$ is real-valued with
\EQn{\label{eq:Cplus-small}
\norm{C_+}_{\A^6_{\rho_0}}\le\ep.
}
Then, for each $q\in E_{\gamma,\mu}$, there exist nonnegative integers $J_q^\Omega,J_q^S$ and coefficients
\EQ{
\Omega_{q,j}&\in\A^6_{\rho_1},
\qquad 0\le j\le J_q^\Omega,
\\
\Sigma_{q,j}&\in\A^6_{\rho_1},
\qquad 0\le j\le J_q^S.
}
The coefficients depend on $(\Psi,C_+)$. Definition \ref{def:eikonal-profile} then determines $\Omega_q$, $\Sigma_q$, and the truncated profile $(L^{(\mu)},S^{(\mu)})$. The following conclusions hold.
\begin{enumerate}
\item There exist $T\ge2$ and a final-state correction $Z$ in the sense of Definition \ref{def:eikonal-asymptotic}. The function
\EQ{%\label{eq:eikonal-solution-form}
u(t,x)=\frac{1}{(it)^{\frac d2}}e^{\frac{i|x|^2}{2t}}
\exp\mbrko{
L^{(\mu)}\brko{t^{-\al},\frac{x}{t}}
-it^{1-\al}S^{(\mu)}\brko{t^{-\al},\frac{x}{t}}
+Z\brko{t,\frac{x}{t}}
}
}
solves \eqref{eq:nls} on $[T,\I)$, and its ray profile belongs to the Fuchsian solution class of Definition \ref{def:eikonal-asymptotic}.
\item For every $2\le q\le\I$, the solution satisfies
\EQn{\label{eq:eikonal-Lq-remainder}
\normb{u(t)-u_{\rm eik}^{(\mu)}(t)}_{L^q_x}
\lsm_{\mu,R}\norm{W}_{L^q_y}t^{-d\brko{\frac12-\frac1q}-(\kappa+\al-1)},
\qquad t\ge T,
}
where
\EQ{%\label{eq:eikonal-modifier-intro}
u_{\rm eik}^{(\mu)}(t,x)
=\frac{1}{(it)^{\frac d2}}e^{\frac{i|x|^2}{2t}}
\exp\mbrko{
L^{(\mu)}\brko{t^{-\al},\frac{x}{t}}
-it^{1-\al}S^{(\mu)}\brko{t^{-\al},\frac{x}{t}}
}.
}
Moreover,
\EQn{\label{eq:eikonal-Lq-modulus}
\normb{|u(t)|-|u_{\rm eik}^{(\mu)}(t)|}_{L^q_x}
\lsm_{\mu,R}\norm{W}_{L^q_y}t^{-d\brko{\frac12-\frac1q}-\kappa},
}
and, with $M_\mu$ as in Proposition \ref{prop:eikonal-prep},
\EQn{\label{eq:eikonal-Lq-leading}
\absb{t^{d\brko{\frac12-\frac1q}}\norm{u(t)}_{L^q_x}-\norm{W}_{L^q_y}}
\lsm_{\mu,R}\norm{W}_{L^q_y}t^{-\al}(1+\log t)^{M_\mu}.
}
\item For fixed $(\mu,\kappa)$, the solution is unique within the Fuchsian solution class of Definition \ref{def:eikonal-asymptotic} associated with the same truncated Fuchsian profile.
\item Let $\mu_j>2\gamma$ be integers and let $1-\al<\kappa_j<\al(\mu_j+1)-1$, $j=1,2$. If the same Fuchsian profile parameters $(\Psi,C_+)$ satisfy the smallness assumptions for both constructions, then the corresponding time-zero values coincide. Thus the constructed wave operator is independent of both the truncation order $\mu$ and the auxiliary exponent $\kappa$.
\end{enumerate}
\end{thm}

\begin{remark}[Truncation order]
In Theorem \ref{thm:eikonal}, the integer $\mu>2\gamma$ and the exponent $\kappa$ are fixed before the fixed point argument. For the same Fuchsian profile parameters $(\Psi,C_+)$, the finite Fuchsian profile $(L^{(\mu)},S^{(\mu)})$ depends on $\mu$: increasing $\mu$ adds further terms of the form $\tau^q(\log\tau)^j$ and cancels more coefficients in the two eikonal residuals. Thus a larger $\mu$ gives a higher-order Fuchsian profile; the estimate \eqref{eq:eikonal-formal-residual} records this by the factor $\tau^\mu$. The exponent $\kappa$ determines the decay class used in the fixed point argument and does not enter the finite Fuchsian profile.
\end{remark}

\begin{cor}[Arbitrary finite-order decay of the Fuchsian remainder]
Fix the parameters and data in Theorem \ref{thm:eikonal}, except for $(\mu,\kappa)$, and let $N>0$. Choose an integer $\mu_N$ and define $\kappa_N$ by
\EQ{
\mu_N>\max\fbrko{2\gamma,\frac{N+2}{\al}-2},
\qquad
\kappa_N:=N+1-\al.
}
Take $\mu=\mu_N$ and $\kappa=\kappa_N$ in Theorem \ref{thm:eikonal}. The corresponding solution satisfies
\EQ{
\norm{u(t)-u_{\rm eik}^{(\mu_N)}(t)}_{L^q_x}
\lsm_{\mu_N,R}\norm{W}_{L^q_y}t^{-d\brko{\frac12-\frac1q}-N},
\qquad 2\le q\le\I,\quad t\ge T.
}
The profile $u_{\rm eik}^{(\mu_N)}$ depends on $N$.
\end{cor}

\begin{proof}
The choices above give
\EQ{
1-\al<\kappa_N<\al(\mu_N+1)-1,
\qquad
\kappa_N+\al-1=N.
}
The conclusion follows from \eqref{eq:eikonal-Lq-remainder}.
\end{proof}

\subsection{Truncated Fuchsian expansion}%\label{sec:eikonal-preparation}

To construct the Fuchsian profile, set
\EQn{\label{eq:eikonal-phase-formal}
\Phi(t,y)=L(\tau,y)-it^{1-\al}S(\tau,y),
\qquad S \text{ real-valued}.
}
Here $\tau\to0$ as $t\to\I$.

We first compute the residual of \eqref{eq:phase-equation} for the general form \eqref{eq:eikonal-phase-formal}. Since $S$ is real,
\EQ{
e^{p\re\Phi}=e^{p\re L},
\qquad
\pd_t\tau=-\al t^{-1}\tau .
}
Hence
\EQ{
i\pd_t\Phi
=
-t^{-\al}\brko{\al\tau\pd_\tau S-(1-\al)S}
-it^{-1}\al\tau\pd_\tau L,
}
and
\EQ{
\frac1{2t^2}Q(\Phi)
&=
-t^{-\al}\frac{\tau}{2}\nabla S\cdot\nabla S
\\
&\quad
-it^{-1}
\mbrko{
\tau\nabla S\cdot\nabla L+\frac{\tau}{2}\Dy S
+\frac{i}{2}\tau^{\gamma+1}
\brko{\Dy L+\nabla L\cdot\nabla L}
}.
}
Therefore
\EQn{\label{eq:eikonal-phase-residual-expanded}
&i\pd_t\Phi+\frac1{2t^2}Q(\Phi)-t^{-\al}e^{p\re\Phi}
\\
&\quad =
-t^{-\al}\mbrko{\al\tau\pd_\tau S-(1-\al)S+\frac{\tau}{2}\nabla S\cdot\nabla S+e^{p\re L}}
\\
&\qquad
-it^{-1}\mbrko{\al\tau\pd_\tau L+\tau\nabla S\cdot\nabla L+\frac{\tau}{2}\Dy S+\frac{i}{2}\tau^{\gamma+1}\brko{\Dy L+\nabla L\cdot\nabla L}}.
}
A pair $(L,S)$ for which the two brackets in \eqref{eq:eikonal-phase-residual-expanded} vanish would solve the phase equation exactly. This gives the exact eikonal system
\EQn{\label{eq:eikonal-system-S}
\al\tau\pd_\tau S
=(1-\al)S-\frac{\tau}{2}\nabla S\cdot\nabla S-e^{p\re L},
}
and
\EQn{\label{eq:eikonal-system-L}
\al\pd_\tau L
=-\nabla S\cdot\nabla L-\frac12\Dy S
-\frac{i}{2}\tau^\gamma\brko{\Dy L+\nabla L\cdot\nabla L}.
}
We do not need an exact solution of \eqref{eq:eikonal-system-S}--\eqref{eq:eikonal-system-L}. Instead we solve the system formally at the singular endpoint $\tau=0$ up to a finite order. The two brackets in \eqref{eq:eikonal-phase-residual-expanded} are therefore kept as residuals:
\EQn{\label{eq:eikonal-ES}
\mathcal R_S(L,S)
:=\al\tau\pd_\tau S-(1-\al)S
+\frac{\tau}{2}\nabla S\cdot\nabla S+e^{p\re L},
}
and
\EQn{\label{eq:eikonal-EL}
\mathcal R_L(L,S)
:=\al\tau\pd_\tau L+\tau\nabla S\cdot\nabla L
+\frac{\tau}{2}\Dy S
+\frac{i}{2}\tau^{\gamma+1}\brko{\Dy L+\nabla L\cdot\nabla L}.
}
Thus, for $\Phi=L(t^{-\al})-it^{1-\al}S(t^{-\al})$, the residual of \eqref{eq:phase-equation} is
\EQn{\label{eq:eikonal-residual-identity}
&i\pd_t\Phi+\frac1{2t^2}Q(\Phi)-t^{-\al}e^{p\re\Phi} \\
&=-t^{-\al}\mathcal R_S(L,S)(t^{-\al})
-it^{-1}\mathcal R_L(L,S)(t^{-\al}).
}
The construction of the truncated eikonal phase is the recursive cancellation of the coefficients of $\mathcal R_S$ and $\mathcal R_L$ at $\tau=0$ up to a prescribed order.

We first determine the leading terms for the eikonal equation. If the terms with an explicit factor $\tau$ are omitted, \eqref{eq:eikonal-system-S} becomes the algebraic equilibrium
\EQ{%\label{eq:Sstar-def}
S=S_*(L).
}
Thus the order-zero Fuchsian profile is $(\Psi,S_*(\Psi))$. Since $-(1-\al)S_*(\Psi)+e^{p\re\Psi}=0$, the Fuchsian part in the $\Sigma$ equation is governed by
\EQ{%\label{eq:Lp-def}
	\mathcal L_\al\sigma
	:=\al\tau\pd_\tau\sigma-(1-\al)\sigma.
}
Next, the homogeneous equation associated with the Fuchsian part of $\mathcal R_S$
\EQ{
\mathcal L_\al S = \al\tau\pd_\tau S-(1-\al)S = 0
}
is
\EQ{
C_+(y)\tau^\gamma,
}
for any $C_+(y)$. The homogeneous exponent of this Fuchsian operator is $q=\gamma=(1-\al)/\al$. Hence the coefficient of $\tau^\gamma$ cannot be fixed using the equation. The coefficient $C_+$ is retained to display the homogeneous mode in the auxiliary recursion. The main theorem uses only the restricted case $C_+=0$.

We write the correction pair by
\EQ{
L=\Psi+\Omega,
\qquad
S=S_*(\Psi)+C_+\tau^\gamma+\Sigma.
}
Let $S^\circ=S-C_+\tau^\gamma$. Since $\al\gamma=1-\al$, the Fuchsian part of \eqref{eq:eikonal-ES} vanishes on $C_+\tau^\gamma$, and \eqref{eq:eikonal-ES}--\eqref{eq:eikonal-EL} give
\EQn{\label{eq:eikonal-Cplus-first-forcing}
\mathcal R_S(L,S)-\mathcal R_S(L,S^\circ)
&=\frac{\tau^{1+\gamma}}2\mbrko{2\nabla S^\circ\cdot\nabla C_++\tau^\gamma|\nabla C_+|^2},
\\
\mathcal R_L(L,S)-\mathcal R_L(L,S^\circ)
&=\tau^{1+\gamma}\mbrko{\nabla C_+\cdot\nabla L+\frac12\Dy C_+}.
}
Thus the prescribed homogeneous coefficient $C_+$ does not enter the forcing at exponent $q=\gamma$; its first forced occurrence is at exponent $1+\gamma$. Here $\Omega$ and $\Sigma$ denote higher-order terms in $\tau$. Then, the main idea is to determine the coefficient of $\Omega$ and $\Sigma$ such that the residuals $\mathcal R_S$ and $\mathcal R_L$ possess desired time decay.

Next, we should make clear the form of the expansion of $\tau$. We have the following observations, which are based on the Frobenius method for the Fuchsian equations:
\begin{itemize}
\item The initial terms in $L$ and $S$, that is $\Psi$ and $S_*(\Psi)+C_+\tau^\gamma$, contain $\tau^0$ and $\tau^\gamma$. In the formulas for $\mathcal R_S$ and $\mathcal R_L$, terms of order $\tau^1$ also appear.
\item The algebraic construction of $\mathcal R_S$ and $\mathcal R_L$ yields that the higher order terms for $\tau$ should include the form of
\EQ{
\tau^{m+n\gamma},\qquad m,n\in\N,\quad m+n\ge1,
}
which is generated by $\tau^0$, $\tau^1$, and $\tau^\gamma$.
\item At the resonant exponent $q=\gamma$, one has
\EQ{
\mathcal L_\al(\tau^\gamma)=0.
}
Hence the coefficient of $\tau^\gamma$ cannot be obtained by inverting the nonresonant operator. Instead,
\EQ{
\mathcal L_\al(\tau^\gamma\log\tau)=\al\tau^\gamma.
}
Therefore a resonant forcing $F(y)\tau^\gamma$ is solved by the term $\al^{-1}F(y)\tau^\gamma\log\tau$.
\item The algebraic construction of $\mathcal R_S$ and $\mathcal R_L$ yields again that the higher order terms for $\tau$ should include the form of
\EQ{
\tau^{m+n\gamma}(\log\tau)^r,\qquad m,n,r\in\N,\quad m+n\ge1
}
This explains why we denote $E_{\ga,\mu}$ and the coefficients in Definition \ref{def:eikonal-profile}
\end{itemize}

These observations explain the exponent set \eqref{eq:eikonal-exponent-set} and the profile form \eqref{eq:eikonal-LS-expansion} in Definition \ref{def:eikonal-profile}.

\begin{prop}[Truncated Fuchsian profile]\label{prop:eikonal-prep}
Let $\mu>2\gamma$ be an integer, and let $E_{\gamma,\mu}$ be defined by \eqref{eq:eikonal-exponent-set}. Under the assumptions of Theorem \ref{thm:eikonal}, there exist coefficients $\Omega_q$ and $\Sigma_q$ depending on $(\Psi,C_+)$ such that the finite Fuchsian profile $(L^{(\mu)},S^{(\mu)})$ in Definition \ref{def:eikonal-profile} satisfies the following estimates: There exist $M_\mu,N_\mu\in\N$ such that, for $0<\tau\le1$,
\EQn{\label{eq:eikonal-positive-order}
\norm{\Omega^{(\mu)}(\tau)}_{\A^6_{\rho_1}}
+\norm{\Sigma^{(\mu)}(\tau)}_{\A^6_{\rho_1}}
\lsm_{\mu,R}\tau(1+|\log\tau|)^{M_\mu},
}
and
\EQn{\label{eq:eikonal-formal-residual}
\norm{\mathcal R_S(L^{(\mu)},S^{(\mu)})(\tau)}_{\A^8_{\rho_1}}
+\norm{\mathcal R_L(L^{(\mu)},S^{(\mu)})(\tau)}_{\A^8_{\rho_1}}
\lsm_{\mu,R}\tau^\mu(1+|\log\tau|)^{N_\mu}.
}
Moreover, $S^{(\mu)}$ is real-valued.
\end{prop}

Before proving this proposition, we need the following technical lemma for solving the coefficients for suitable polynomial equations.

\begin{lem}[Polynomial equations]%\label{lem:eikonal-poly-inversions}
Let $F$ be a polynomial
	\EQ{
		F(r,y)=\sum_{j=0}^{J}f_j(y)r^j,
	}
where $f_j$ are given. Then, the coefficients of a polynomial solution $P$ are uniquely determined by $f_j$ in the following three polynomial equations:
	\EQ{
		&\al qP+\al P'=F\quad (q>0),\\
		&\brko{\al(q+1)-1}P+\al P'=F\quad (q\ne\gamma),\\
		&\al P'=F,\quad P(0,y)=0.
	}
\end{lem}
\begin{proof}
Notice that we have for a polynomial $P=P(r)$, $r=\log\tau$,
	\EQ{
		\mathcal L_\al\brko{\tau^qP(\log\tau)}
		=\tau^q\mbrko{\brko{\al(q+1)-1}P+\al P'}.
	}
For the transport equation, write $P(r,y)=\sum_{j=0}^{J}p_j(y)r^j$ and set $p_{J+1}:=0$. Comparing the coefficients of $r^j$ gives
	\EQn{\label{eq:eikonal-poly-inversion-transport}
		&\al qP+\al P'=F,
		\qquad
		P(r,y)=\sum_{j=0}^{J}p_j(y)r^j,
		\qquad p_{J+1}:=0,
		\\
		&p_j=\frac{f_j-\al(j+1)p_{j+1}}{\al q},
		\qquad 0\le j\le J.
	}
Thus the coefficients are obtained successively for $j=J,J-1,\ldots,0$. For the nonresonant Fuchsian equation, again write $P(r,y)=\sum_{j=0}^{J}p_j(y)r^j$ and set $p_{J+1}:=0$. Since $q\ne\gamma=(1-\al)/\al$, comparing coefficients gives
	\EQn{\label{eq:eikonal-poly-inversion-fuchsian}
		&\brko{\al(q+1)-1}P+\al P'=F,
		\qquad
		P(r,y)=\sum_{j=0}^{J}p_j(y)r^j,
		\qquad p_{J+1}:=0,
		\\
		&p_j=\frac{f_j-\al(j+1)p_{j+1}}{\al(q+1)-1},
		\qquad 0\le j\le J.
	}
The denominator is nonzero, so this also determines $p_J,p_{J-1},\ldots,p_0$. Finally, for the resonant Fuchsian equation, write $P(r,y)=\sum_{j=0}^{J+1}p_j(y)r^j$. The condition $P(0,y)=0$ gives $p_0=0$, and equating coefficients gives
	\EQn{\label{eq:eikonal-poly-inversion-resonant}
		&\al P'=F,
		\qquad
		P(0,y)=0,
		\qquad
		P(r,y)=\sum_{j=0}^{J+1}p_j(y)r^j,
		\\
		&p_0=0,
		\qquad
		p_{j+1}=\frac{f_j}{\al(j+1)},
		\qquad 0\le j\le J.
	}
This finishes the proof of this lemma.
\end{proof}
\begin{proof}[Proof of Proposition \ref{prop:eikonal-prep}]

\emph{Step 1: basic settings.} Start from
\EQ{
L=\Psi+\Omega,\qquad
S=S_*(\Psi)+C_+\tau^\gamma+\Sigma .
}
By $L=\Psi+\Omega$, $S=S_*(\Psi)+C_+\tau^\gamma+\Sigma$, and $-(1-\al)S_*(\Psi)+e^{p\re\Psi}=0$, the two residuals are
\EQn{\label{eq:eikonal-residual-expanded}
\mathcal R_S(L,S)&=\mathcal L_\al\brko{C_+\tau^\gamma+\Sigma}+\frac{\tau}{2}\nabla S\cdot\nabla S+\brko{e^{p\re L}-e^{p\re\Psi}},
\\
\mathcal R_L(L,S)&=\al\tau\pd_\tau\Omega+\tau\nabla S\cdot\nabla L+\frac{\tau}{2}\Dy S+\frac{i}{2}\tau^{\gamma+1}\brko{\Dy L+\nabla L\cdot\nabla L} .
}
The finite Fuchsian profile is obtained by choosing polynomial coefficients in $\Omega$ and $\Sigma$ so that, in \eqref{eq:eikonal-residual-expanded}, every coefficient of $\tau^q(\log\tau)^j$ with $0<q<\mu$ is zero, and then the related polynomials are denoted as $\Omega^{(\mu)}$ and $\Sigma^{(\mu)}$. We next fix the notation for this finite coefficient recursion.

Set $r=\log\tau$. If a finite expansion is written as
\EQ{
F(\tau,y)=\sum_{q\in\mathcal Q}\tau^qF_q(r,y),
}
with each $F_q$ polynomial in $r$, then we define the coefficient
\EQ{
\mbrk{F}_q:=F_q(r,y),
\qquad
\mbrk{F}_q:=0\quad(q\notin\mathcal Q).
}
The recursive cancellation condition is
\EQn{\label{eq:eikonal-coefficient-equations}
\mbrk{\mathcal R_L(L^{(\mu)},S^{(\mu)})}_q=0,
\qquad
\mbrk{\mathcal R_S(L^{(\mu)},S^{(\mu)})}_q=0,
\qquad q\in E_{\gamma,\mu}.
}
Equivalently, every coefficient of $r^j$ in the two polynomials in \eqref{eq:eikonal-coefficient-equations} is set equal to zero.

\emph{Step 2: recursive determination of the correction coefficients.} The recursion starts from the order-zero terms
\EQ{%\label{eq:eikonal-step4-zero-order}
L_0=\Psi,
\qquad
S_0=S_*(\Psi)=\frac1{1-\al}e^{p\re\Psi}.
}
Fix $q\in E_{\gamma,\mu}$, and assume that all coefficient pairs with exponent $q'\in E_{\gamma,\mu}$ and $q'<q$ have already been chosen. The lower correction pair is
\EQ{%\label{eq:eikonal-lower-correction-q}
\Omega_{<q}
=\sum_{\substack{q'\in E_{\gamma,\mu}\\ q'<q}}\tau^{q'}\Omega_{q'}(r,y);\quad
\Sigma_{<q}
=\sum_{\substack{q'\in E_{\gamma,\mu}\\ q'<q}}\tau^{q'}\Sigma_{q'}(r,y).
}
Then we already have
\EQn{\label{eq:eikonal-lower-profile-q}
L_{<q}&=\Psi+\Omega_{<q},
\\
S_{<q}&=S_*(\Psi)+\cha_{\{\gamma<q\}}C_+\tau^\gamma+\Sigma_{<q},
}
where the lower coefficients have already been constructed, and the coefficients in $\Sigma_{<q}$ are real-valued. Let
\EQn{\label{eq:eikonal-step4-insertion}
L=L_{<q}+\tau^q\Omega_q,
\qquad
S=S_{<q}+\tau^q\Sigma_q,
}
where order-$q$ unknowns are $\Omega_q$ and $\Sigma_q$. Next, we show how we can solve $\Omega_q$ and $\Sigma_q$ by the coefficients $\Omega_{<q}$ and $\Sigma_{<q}$. By \eqref{eq:eikonal-step4-insertion}, \eqref{eq:eikonal-residual-expanded}, and $r=\log\tau$, the coefficient of $\tau^q$ in $\mathcal R_L$ is
\EQ{%\label{eq:eikonal-RL-q-dependence}
&\mbrk{\mathcal R_L\brko{L_{<q}+\tau^q\Omega_q,\,S_{<q}+\tau^q\Sigma_q}}_q
\\
= &\al\brko{q\Omega_q+\pd_r\Omega_q}+\mbrk{\tau\nabla S_{<q}\cdot\nabla L_{<q}+\frac{\tau}{2}\Dy S_{<q}}_q
\\
&+\mbrk{\frac{i}{2}\tau^{\gamma+1}\brko{\Dy L_{<q}+\nabla L_{<q}\cdot\nabla L_{<q}}}_q .
}
Hence the equation $\mbrk{\mathcal R_L}_q=0$ is
\EQ{%\label{eq:eikonal-Omegaq-equation}
\al q\Omega_q+\al\pd_r\Omega_q=F^\Omega_q,
}
where the known polynomial $F^\Omega_q$ is
\EQn{\label{eq:eikonal-Omegaq-forcing}
F^\Omega_q
&=-\mbrk{
\tau\nabla S_{<q}\cdot\nabla L_{<q}
+\frac{\tau}{2}\Dy S_{<q}
}_q
\\
&\quad
-\mbrk{
\frac{i}{2}\tau^{\gamma+1}
\brko{\Dy L_{<q}+\nabla L_{<q}\cdot\nabla L_{<q}}
}_q .
}
Write $F^\Omega_q(r,y)=\sum_{j=0}^{K_q^\Omega}f^\Omega_{q,j}(y)r^j$. Then $J_q^\Omega=K_q^\Omega$, and \eqref{eq:eikonal-poly-inversion-transport} determines $\Omega_q$. Thus $\Omega_q$ depends only on $(\Psi,C_+)$ and on the previously constructed coefficients, and can be written explicitly on those coefficients.

After $\Omega_q$ has been fixed, let
\EQn{\label{eq:eikonal-partial-profile-q}
L_{\le q}&=L_{<q}+\tau^q\Omega_q(r,y),
\\
S_{\le q}&=S_{<q}+\cha_{\{q=\gamma\}}C_+\tau^\gamma+\tau^q\Sigma_q(r,y).
}
For the phase component we use $r=\log\tau$ and hence $\tau\pd_\tau r=1$. Therefore
\EQn{\label{eq:eikonal-Sigmaq-fuchsian-action}
\mathcal L_\al\brko{\tau^q\Sigma_q(r,y)}
=\tau^q\mbrko{\brko{\al(q+1)-1}\Sigma_q+\al\pd_r\Sigma_q}.
}
By \eqref{eq:eikonal-partial-profile-q}, \eqref{eq:eikonal-Sigmaq-fuchsian-action}, the explicit factor $\tau$ in the quadratic term of \eqref{eq:eikonal-residual-expanded}, and \eqref{eq:eikonal-Cplus-first-forcing},
\EQn{\label{eq:eikonal-RS-q-dependence}
0
&=\mbrk{\mathcal R_S\brko{L^{(\mu)},S^{(\mu)}}}_q
=\mbrk{\mathcal R_S\brko{L_{\le q},S_{\le q}}}_q
\\
&=\brko{\al(q+1)-1}\Sigma_q+\al\pd_r\Sigma_q
+\mbrk{
\frac{\tau}{2}\nabla S_{<q}\cdot\nabla S_{<q}
+\brko{e^{p\re L_{\le q}}-e^{p\re\Psi}}
}_q.
}
Here the first equality is the second cancellation condition in \eqref{eq:eikonal-coefficient-equations}. Define the known polynomial
\EQn{\label{eq:eikonal-Sigmaq-forcing}
F^S_q
=-\mbrk{
\frac{\tau}{2}\nabla S_{<q}\cdot\nabla S_{<q}
+\brko{e^{p\re L_{\le q}}-e^{p\re\Psi}}
}_q .
}
Then \eqref{eq:eikonal-RS-q-dependence} is equivalent to
\EQn{\label{eq:eikonal-Sigmaq-equation}
\brko{\al(q+1)-1}\Sigma_q+\al\pd_r\Sigma_q=F^S_q.
}
By the induction hypothesis, the lower coefficients of $\Sigma$ are real-valued. Since $S_*(\Psi)$ and $C_+$ are real-valued, $S_{<q}$ is real-valued. Hence $\nabla S_{<q}\cdot\nabla S_{<q}$ is real-valued. Moreover $e^{p\re L_{\le q}}$ and $e^{p\re\Psi}$ are real-valued. Therefore $F^S_q$ is a real-valued polynomial. Write
\EQ{%\label{eq:eikonal-FSq-polynomial}
F^S_q(r,y)=\sum_{j=0}^{K_q^S}f^S_{q,j}(y)r^j,
\qquad f^S_{q,j}\ \text{is real-valued}.
}
The coefficient $\Sigma_q$ in \eqref{eq:eikonal-Sigmaq-equation} is determined as follows:
\begin{itemize}
\item If $q\ne\gamma$, then \eqref{eq:eikonal-poly-inversion-fuchsian} gives
\EQ{%\label{eq:eikonal-Sigma-case-nonresonant}
J_q^S=K_q^S,
\qquad
\Sigma_q=\sum_{j=0}^{J_q^S}\Sigma_{q,j}(y)r^j .
}
\item If $q=\gamma$, then \eqref{eq:eikonal-poly-inversion-resonant} gives
\EQ{%\label{eq:eikonal-Sigma-case-resonant}
J_\gamma^S=K_\gamma^S+1,
\qquad
\Sigma_{\gamma,0}=0,
\qquad
\Sigma_\gamma=\sum_{j=0}^{J_\gamma^S}\Sigma_{\gamma,j}(y)r^j .
}
\end{itemize}
In both cases, the coefficients in \eqref{eq:eikonal-poly-inversion-fuchsian} and \eqref{eq:eikonal-poly-inversion-resonant} are real. Hence each $\Sigma_{q,j}$ is real-valued. Consequently the correction pair $(\Omega_{\le q},\Sigma_{<q}+\tau^q\Sigma_q)$ is determined by $(\Psi,C_+)$ and by the lower exponents. This completes the recursion for all $q\in E_{\gamma,\mu}$.

\emph{Step 3: coefficient and residual estimates.} Let
\EQ{%\label{eq:eikonal-ordered-exponents}
E_{\gamma,\mu}=\{q_1<q_2<\cdots<q_K\}.
}
Choose a finite decreasing sequence of analytic radii
\EQ{%\label{eq:eikonal-coefficient-radius-sequence}
\rho_0=\zeta_0>\zeta_1>\cdots>\zeta_K>\rho_1 .
}
At the $k$-th step the already constructed lower coefficients are estimated at the radius $\zeta_{k-1}$, while the new coefficients are estimated at $\zeta_k$.  By \eqref{eq:radius-gap-estimate},
\EQn{\label{eq:eikonal-coefficient-radius-gap}
\norm{f}_{\A^8_{\zeta_k}}
\lsm\mbrko{1+(\zeta_{k-1}-\zeta_k)^{-2}}
\norm{f}_{\A^6_{\zeta_{k-1}}}.
}
For $1\le k\le K$, write
\EQn{\label{eq:eikonal-coefficient-induction-size}
\mathcal B_{k-1}:=
\sum_{m<k}
\brko{
\sum_{j=0}^{J_{q_m}^\Omega}
\norm{\Omega_{q_m,j}}_{\A^6_{\zeta_{k-1}}}
+
\sum_{j=0}^{J_{q_m}^S}
\norm{\Sigma_{q_m,j}}_{\A^6_{\zeta_{k-1}}}
}.
}
We prove by induction that $\mathcal B_{k-1}\lsm_{k,\mu,R}1$.  This is clear for $k=1$.  Assume it has been proved for every exponent strictly below $q_k$. For a finite expansion in $\tau$ and $r$, let $\mbrk{F}_{a,r^j}$ denote the coefficient of $\tau^ar^j$. By \eqref{eq:eikonal-lower-profile-q}, \eqref{eq:eikonal-coefficient-radius-gap}, the assumptions on $(\Psi,C_+)$, and \eqref{eq:eikonal-coefficient-induction-size},
\EQ{%\label{eq:eikonal-lower-coefficient-derivative-bound}
&\sum_{\substack{a<q_k\\j\ge0}}
\brko{
\norm{\nabla\mbrk{L_{<q_k}}_{a,r^j}}_{\A^7_{\zeta_k}}
+\norm{\nabla\mbrk{S_{<q_k}}_{a,r^j}}_{\A^7_{\zeta_k}}
}
\lsm_{k,\mu,R}1.
}
Then, by \eqref{eq:eikonal-poly-inversion-transport} and \eqref{eq:eikonal-Omegaq-forcing},
\EQn{\label{eq:eikonal-Omegaq-bound-proof}
\sum_{j=0}^{J_{q_k}^\Omega}
\norm{\Omega_{q_k,j}}_{\A^6_{\zeta_k}}
\lsm_{k,\mu}
\sum_{j=0}^{K_{q_k}^\Omega}
\norm{f^\Omega_{q_k,j}}_{\A^6_{\zeta_k}}
\lsm_{k,\mu,R}1.
}
After $\Omega_{q_k}$ has been fixed, define $g_{q_k,j}$ by
\EQn{\label{eq:eikonal-exp-q-coefficient}
\sum_j g_{q_k,j}r^j
=
\mbrk{e^{p\re L_{\le q_k}}-e^{p\re\Psi}}_{q_k}
=
e^{p\re\Psi}
\mbrk{\sum_{\ell\ge1}
\frac{\brko{p\re\Omega_{\le q_k}}^\ell}{\ell!}}_{q_k} .
}
Only finitely many $\ell$ contribute to the coefficient of $\tau^{q_k}$. Thus \eqref{eq:eikonal-exp-q-coefficient}, the algebra property, \eqref{eq:eikonal-coefficient-induction-size}, and \eqref{eq:eikonal-Omegaq-bound-proof} imply
\EQn{\label{eq:eikonal-exp-q-bound}
\sum_j
\norm{g_{q_k,j}}_{\A^6_{\zeta_k}}
\lsm_{k,\mu,R}1.
}
The remaining part of \eqref{eq:eikonal-Sigmaq-forcing} contains $\tau\nabla S_{<q_k}\cdot\nabla S_{<q_k}$.  Applying \eqref{eq:eikonal-coefficient-radius-gap} once to each differentiated factor gives
\EQn{\label{eq:eikonal-S-gradient-q-bound}
\sum_j
\norm{
\mbrk{\frac{\tau}{2}\nabla S_{<q_k}\cdot\nabla S_{<q_k}}_{q_k,r^j}
}_{\A^6_{\zeta_k}}
\lsm_{k,\mu,R}1,
}
Combining \eqref{eq:eikonal-exp-q-bound} and \eqref{eq:eikonal-S-gradient-q-bound} in \eqref{eq:eikonal-Sigmaq-forcing} gives
\EQn{\label{eq:eikonal-FSq-bound}
\sum_{j=0}^{K_{q_k}^S}
\norm{f^S_{q_k,j}}_{\A^6_{\zeta_k}}
\lsm_{k,\mu,R}1.
}
For $q_k\ne\gamma$, \eqref{eq:eikonal-poly-inversion-fuchsian} and \eqref{eq:eikonal-FSq-bound} give
\EQ{%\label{eq:eikonal-Sigmaq-bound-proof}
\sum_{j=0}^{J_{q_k}^S}
\norm{\Sigma_{q_k,j}}_{\A^6_{\zeta_k}}
\lsm_{k,\mu,R}1.
}
For $q_k=\gamma$, the same bound follows from the resonant formula \eqref{eq:eikonal-poly-inversion-resonant} and \eqref{eq:eikonal-FSq-bound}.  Therefore
\EQ{%\label{eq:eikonal-coefficient-step-bound}
\sum_{j=0}^{J_{q_k}^\Omega}
\norm{\Omega_{q_k,j}}_{\A^6_{\zeta_k}}
+
\sum_{j=0}^{J_{q_k}^S}
\norm{\Sigma_{q_k,j}}_{\A^6_{\zeta_k}}
\lsm_{k,\mu,R}1.
}
This closes the induction. In particular, the last radius in the construction satisfies
\EQn{\label{eq:eikonal-coefficient-bound-construction-radius}
\sum_{q\in E_{\gamma,\mu}}
\brko{
\sum_{j=0}^{J_q^\Omega}
\norm{\Omega_{q,j}}_{\A^6_{\zeta_K}}
+
\sum_{j=0}^{J_q^S}
\norm{\Sigma_{q,j}}_{\A^6_{\zeta_K}}
}
\lsm_{\mu,R}1.
}
Since $\zeta_K>\rho_1$, monotonicity of the analytic norm gives
\EQn{\label{eq:eikonal-coefficient-bound}
\sum_{q\in E_{\gamma,\mu}}
\brko{
\sum_{j=0}^{J_q^\Omega}
\norm{\Omega_{q,j}}_{\A^6_{\rho_1}}
+
\sum_{j=0}^{J_q^S}
\norm{\Sigma_{q,j}}_{\A^6_{\rho_1}}
}
\lsm_{\mu,R}1.
}
Let
\EQ{
M_\mu:=\max_{q\in E_{\gamma,\mu}}\max\{J_q^\Omega,J_q^S\}.
}
For $0<\tau\le1$, $q\ge1$, and $0\le j\le M_\mu$,
\EQ{
\tau^q\abs{\log\tau}^j\le \tau(1+\abs{\log\tau})^{M_\mu}.
}
Combining this with \eqref{eq:eikonal-coefficient-bound} proves \eqref{eq:eikonal-positive-order}. The gap $\zeta_K-\rho_1$ is kept for the residual estimates. Since $\mathcal R_L$ contains Laplacian terms, the residual estimate in $\A^8_{\rho_1}$ requires two additional spatial derivatives of the coefficients. By \eqref{eq:radius-gap-estimate} and \eqref{eq:eikonal-coefficient-bound-construction-radius},
\EQn{\label{eq:eikonal-final-gap-derivatives}
\sum_{q\in E_{\gamma,\mu}}
\brko{
\sum_{j=0}^{J_q^\Omega}
\norm{\Omega_{q,j}}_{\A^{10}_{\rho_1}}
+
\sum_{j=0}^{J_q^S}
\norm{\Sigma_{q,j}}_{\A^{10}_{\rho_1}}
}
\lsm_{\mu,R}1.
}

It remains to prove \eqref{eq:eikonal-formal-residual}. We estimate the two residuals in \eqref{eq:eikonal-residual-expanded} with
\EQ{
L=L^{(\mu)},\qquad S=S^{(\mu)}.
}
We first separate the powers below $\mu$ from the remaining terms.  For $x\in\R$, we use the Taylor expansion with the remainder in the integral form:
\EQ{%\label{eq:eikonal-exp-scalar-taylor}
e^{px}-1
=\sum_{m=1}^{\mu-1}\frac{(px)^m}{m!}
+x^\mu G_\mu(x),
\qquad
G_\mu(x)=\frac{p^\mu}{(\mu-1)!}\int_0^1(1-\theta)^{\mu-1}e^{p\theta x}\dd \theta .
}
The same formula with $x$ replaced by $z\in\C$ shows that $G_\mu$ is entire. Then,
\EQn{\label{eq:eikonal-exp-taylor-remainder}
e^{p\re\Omega^{(\mu)}}-1
&=\sum_{m=1}^{\mu-1}\frac{(p\re\Omega^{(\mu)})^m}{m!}
+\brko{\re\Omega^{(\mu)}}^\mu G_\mu(\re\Omega^{(\mu)}).
}
By \eqref{eq:eikonal-positive-order}, \eqref{eq:eikonal-final-gap-derivatives}, and the algebra and composition estimates in Lemma \ref{lem:fourier-algebra},
\EQn{\label{eq:eikonal-exp-composition-bound}
|G_\mu(0)|+
\norm{G_\mu(\re\Omega^{(\mu)}(\tau))-G_\mu(0)}_{\A^8_{\rho_1}}
\lsm_{\mu,R}1.
}
Moreover,
\EQn{\label{eq:eikonal-exp-remainder-split}
\brko{\re\Omega^{(\mu)}}^\mu G_\mu(\re\Omega^{(\mu)})
=G_\mu(0)\brko{\re\Omega^{(\mu)}}^\mu
+\brko{\re\Omega^{(\mu)}}^\mu
\brko{G_\mu(\re\Omega^{(\mu)})-G_\mu(0)}.
}
Combining \eqref{eq:eikonal-positive-order}, \eqref{eq:eikonal-exp-composition-bound}, and \eqref{eq:eikonal-exp-remainder-split} gives
\EQn{\label{eq:eikonal-exp-remainder-bound}
\norm{\brko{\re\Omega^{(\mu)}}^\mu G_\mu(\re\Omega^{(\mu)})}_{\A^8_{\rho_1}}
\lsm_{\mu,R}\tau^\mu(1+|\log\tau|)^{N_\mu}.
}
By \eqref{eq:eikonal-residual-expanded} and \eqref{eq:eikonal-exp-taylor-remainder}, there exist finite sets
\EQ{
\mathcal Q_L,\mathcal Q_S
\subset
\fbrk{\ell_1+\ell_2\gamma:
\ell_1,\ell_2\in\N,\ \ell_1+\ell_2\ge1}
}
and polynomials $P^L_q(r,y)$ and $P^S_q(r,y)$ such that
\EQn{\label{eq:eikonal-residual-finite-expansion}
\mathcal R_L(L^{(\mu)},S^{(\mu)})(\tau)
&=\sum_{q\in\mathcal Q_L}\tau^qP^L_q(r,y),
\\
\mathcal R_S(L^{(\mu)},S^{(\mu)})(\tau)
&=\sum_{q\in\mathcal Q_S}\tau^qP^S_q(r,y)
+e^{p\re\Psi}
\brko{\re\Omega^{(\mu)}}^\mu
G_\mu\brko{\re\Omega^{(\mu)}}.
}
Here the first sum contains all terms in $\mathcal R_L$, while the second sum contains the Fuchsian term, the quadratic term in $S^{(\mu)}$, and the finite Taylor polynomial in \eqref{eq:eikonal-exp-taylor-remainder}. We set $P^L_q=0$ for $q\notin\mathcal Q_L$ and $P^S_q=0$ for $q\notin\mathcal Q_S$. Every exponent in $\mathcal Q_L\cup\mathcal Q_S$ below $\mu$ belongs to $E_{\gamma,\mu}$. Hence \eqref{eq:eikonal-residual-finite-expansion} can be written as
\EQn{\label{eq:eikonal-residual-low-high-split}
\mathcal R_L(L^{(\mu)},S^{(\mu)})(\tau)
&=\sum_{q\in E_{\gamma,\mu}}\tau^qP^L_q(\log\tau,y)
+\tau^\mu E_L(\tau,\log\tau,y),
\\
\mathcal R_S(L^{(\mu)},S^{(\mu)})(\tau)
&=\sum_{q\in E_{\gamma,\mu}}\tau^qP^S_q(\log\tau,y)
+\tau^\mu E_S(\tau,\log\tau,y).
}
More precisely,
\EQn{\label{eq:eikonal-EL-ES-structure}
E_L(\tau,r,y)
&=\sum_{\substack{q\in\mathcal Q_L\\q\ge\mu}}
\tau^{q-\mu}P^L_q(r,y),
\\
E_S(\tau,r,y)
&=\sum_{\substack{q\in\mathcal Q_S\\q\ge\mu}}
\tau^{q-\mu}P^S_q(r,y)
+e^{p\re\Psi}\tau^{-\mu}
\brko{\re\Omega^{(\mu)}}^\mu
G_\mu\brko{\re\Omega^{(\mu)}}.
}
The two sums in \eqref{eq:eikonal-EL-ES-structure} are finite. For every $q>\mu$ and $j\in\N$,
\EQn{\label{eq:eikonal-high-log-uniform}
\sup_{0<\tau\le1}\tau^{q-\mu}(1+|\log\tau|)^j<\I.
}
Thus the terms with $q>\mu$ in \eqref{eq:eikonal-EL-ES-structure} are uniformly bounded, while the terms with $q=\mu$ are polynomials in $r=\log\tau$. Moreover, \eqref{eq:eikonal-exp-remainder-bound}, \eqref{eq:radius-gap-estimate}, the assumptions on $e^{p\re\Psi}$, and Lemma \ref{lem:fourier-algebra} give
\EQn{\label{eq:eikonal-ES-Taylor-tail-bound}
\norm{e^{p\re\Psi}\tau^{-\mu}
\brko{\re\Omega^{(\mu)}}^\mu
G_\mu\brko{\re\Omega^{(\mu)}}
}_{\A^8_{\rho_1}}
\lsm_{\mu,R}(1+|\log\tau|)^{N_\mu}.
}

For $q\in E_{\gamma,\mu}$,
\EQ{%\label{eq:eikonal-low-polynomial-identification}
	P^L_q=\mbrk{\mathcal R_L(L^{(\mu)},S^{(\mu)})}_q,
	\qquad
	P^S_q=\mbrk{\mathcal R_S(L^{(\mu)},S^{(\mu)})}_q.
}
By \eqref{eq:eikonal-coefficient-equations},
\EQ{%\label{eq:eikonal-low-polynomial-vanishing}
	P^L_q=0,
	\qquad
	P^S_q=0,
	\qquad q\in E_{\gamma,\mu}.
}
Therefore \eqref{eq:eikonal-residual-low-high-split} reduces to
\EQn{\label{eq:eikonal-residual-after-cancellation}
	\mathcal R_L(L^{(\mu)},S^{(\mu)})(\tau)
	&=\tau^\mu E_L(\tau,\log\tau,y),
	\\
	\mathcal R_S(L^{(\mu)},S^{(\mu)})(\tau)
	&=\tau^\mu E_S(\tau,\log\tau,y).
}
The coefficient bounds \eqref{eq:eikonal-coefficient-bound} and \eqref{eq:eikonal-final-gap-derivatives}, the structure \eqref{eq:eikonal-EL-ES-structure}, \eqref{eq:eikonal-high-log-uniform}, and \eqref{eq:eikonal-ES-Taylor-tail-bound} give, for some integer $N_\mu$,
\EQn{\label{eq:eikonal-residual-coefficient-bound}
\norm{E_L(\tau,\log\tau)}_{\A^8_{\rho_1}}
+\norm{E_S(\tau,\log\tau)}_{\A^8_{\rho_1}}
\lsm_{\mu,R}(1+\abs{\log\tau})^{N_\mu}.
}
\eqref{eq:eikonal-residual-after-cancellation} and \eqref{eq:eikonal-residual-coefficient-bound} prove \eqref{eq:eikonal-formal-residual}. This completes the proof of Proposition \ref{prop:eikonal-prep}.
\end{proof}

For the Fuchsian profile constructed in Proposition \ref{prop:eikonal-prep}, define
\EQ{
\Phi^{(\mu)}(t,y)=L^{(\mu)}(t^{-\al},y)-it^{1-\al}S^{(\mu)}(t^{-\al},y)
}
and
\EQn{\label{eq:Rmu-def}
R^{(\mu)}(t)
:=i\pd_t\Phi^{(\mu)}+\frac1{2t^2}Q(\Phi^{(\mu)})
-t^{-\al}e^{p\re\Phi^{(\mu)}}.
}

\begin{cor}[Estimates for truncated Fuchsian profiles]\label{cor:eikonal-auxiliary-bounds}
After increasing $T$ if necessary, the following statements hold for $t\ge T$.
\begin{enumerate}[label=\textup{(\roman*)},leftmargin=*]
\item Let $L^{(\mu)}$ and $S^{(\mu)}$ be the Fuchsian profile constructed in Proposition \ref{prop:eikonal-prep}. Then
\EQn{\label{eq:eikonal-profile-bounds}
\norm{\nabla L^{(\mu)}(t^{-\al})}_{\A^4_{\rho_1}}
+\norm{\nabla S^{(\mu)}(t^{-\al})}_{\A^4_{\rho_1}}
&\lsm_{\mu,R}1,
}
\EQn{\label{eq:eikonal-amplitude-comparable}
\norm{e^{\re L^{(\mu)}(t^{-\al})-\re\Psi}}_{L^\I_y}
\lsm_{\mu,R}1,
}
and
\EQn{\label{eq:eikonal-amplitude-small}
\norm{e^{p\re L^{(\mu)}(t^{-\al})}}_{\A^4_{\rho_1}}
\lsm_{\mu,R}\ep.
}
\item Let $\mu_1,\mu_2>2\gamma$ be integers, let $\bar\mu=\min\fbrko{\mu_1,\mu_2}$, and set $\sigma_0=\al\bar\mu$. For $j=1,2$, let $(L^{(\mu_j)},S^{(\mu_j)})$ be the truncated Fuchsian profile of order $\mu_j$ constructed from the same Fuchsian profile parameters $(\Psi,C_+)$. Then there is $N\in\N$ such that
\EQn{\label{eq:eikonal-separated-profile-diff}
&\norm{L^{(\mu_2)}(t^{-\al})-L^{(\mu_1)}(t^{-\al})}_{\A^5_{\rho_1}}
\\
&\quad+\norm{S^{(\mu_2)}(t^{-\al})-S^{(\mu_1)}(t^{-\al})}_{\A^5_{\rho_1}}
\lsm_{\mu_1,\mu_2,R}t^{-\sigma_0}(1+\log t)^N.
}
\item The residual $R^{(\mu)}$ defined by \eqref{eq:Rmu-def} satisfies
\EQn{\label{eq:Rmu-bound}
\norm{R^{(\mu)}(t)}_{\A^8_{\rho_1}}
\lsm_{\mu,R}t^{-1-\kappa},
\qquad t\ge T.
}
\end{enumerate}
\end{cor}

\begin{proof}
For \textup{(i)}, \eqref{eq:eikonal-LS-expansion} and \eqref{eq:eikonal-positive-order} give
\EQ{
\norm{\nabla\Omega^{(\mu)}(t^{-\al})}_{\A^4_{\rho_1}}
+\norm{\nabla\Sigma^{(\mu)}(t^{-\al})}_{\A^4_{\rho_1}}
\lsm_{\mu,R}1.
}
Together with \eqref{eq:R-bound}, \eqref{eq:eps-small}, and \eqref{eq:Cplus-small}, this proves \eqref{eq:eikonal-profile-bounds}. Moreover,
\EQn{\label{eq:eikonal-exponential-factorization}
e^{p\re L^{(\mu)}}
=e^{p\re\Psi}e^{p\re\Omega^{(\mu)}},
\qquad
e^{\re L^{(\mu)}-\re\Psi}=e^{\re\Omega^{(\mu)}}.
}
Lemma \ref{lem:fourier-algebra} and \eqref{eq:eikonal-positive-order} give
\EQn{\label{eq:eikonal-expOmega-bound}
\normb{e^{\re\Omega^{(\mu)}(t^{-\al})}}_{L^\I_y}
+\norm{e^{p\re\Omega^{(\mu)}(t^{-\al})}-1}_{\A^4_{\rho_1}}
\lsm_{\mu,R}1.
}
The first identity in \eqref{eq:eikonal-exponential-factorization} is written as
\EQn{\label{eq:eikonal-expL-product}
e^{p\re L^{(\mu)}}
=e^{p\re\Psi}
+e^{p\re\Psi}\brko{e^{p\re\Omega^{(\mu)}}-1}.
}
Using \eqref{eq:eps-small}, \eqref{eq:eikonal-expOmega-bound}, and \eqref{eq:fourier-algebra} in \eqref{eq:eikonal-expL-product}, we obtain
\EQ{
\normb{e^{p\re L^{(\mu)}(t^{-\al})}}_{\A^4_{\rho_1}}
\lsm_{\mu,R}\ep.
}
This proves \eqref{eq:eikonal-amplitude-small}. The second identity in \eqref{eq:eikonal-exponential-factorization}, together with \eqref{eq:eikonal-expOmega-bound}, proves \eqref{eq:eikonal-amplitude-comparable}.

For \textup{(ii)}, the two finite Fuchsian profiles have identical coefficients for every exponent $q<\bar\mu$. Hence their difference is a finite sum of terms $t^{-\al q}(\log t)^j$ with $q\ge\bar\mu$. The coefficient bounds in Proposition \ref{prop:eikonal-prep}, with the fixed gap between the construction radius and $\rho_1$, give the $\A^5_{\rho_1}$ bounds in \eqref{eq:eikonal-separated-profile-diff}.

For \textup{(iii)}, \eqref{eq:eikonal-residual-identity} gives
\EQ{
R^{(\mu)}(t)
=-t^{-\al}
\mathcal R_S(L^{(\mu)},S^{(\mu)})(t^{-\al})
-it^{-1}
\mathcal R_L(L^{(\mu)},S^{(\mu)})(t^{-\al}).
}
Together with \eqref{eq:eikonal-formal-residual}, this identity gives
\EQn{\label{eq:Rmu-log-bound}
\norm{R^{(\mu)}(t)}_{\A^8_{\rho_1}}
\lsm_{\mu,R}
t^{-\al(\mu+1)}(1+\log t)^{N_\mu}.
}
Recall the choice of $\kappa$ in \eqref{eq:kappa-choice}, and let
\EQ{
\delta_\mu:=\al(\mu+1)-1-\kappa>0.
}
For every $N\in\N$ and $\delta>0$, there exists $T_{N,\delta}\ge1$ such that
\EQ{
(1+\log t)^N\le t^\delta,
\qquad t\ge T_{N,\delta}.
}
Taking $N=N_\mu$ and $\delta=\delta_\mu$, and increasing $T$ if necessary, gives
\EQn{\label{eq:eikonal-log-absorption}
(1+\log t)^{N_\mu}\le t^{\delta_\mu},
\qquad t\ge T.
}
\eqref{eq:Rmu-log-bound} and \eqref{eq:eikonal-log-absorption} prove \eqref{eq:Rmu-bound}.
\end{proof}
\subsection{Equation for the phase correction around the Fuchsian profile}%\label{sec:eikonal-error}

Let
\EQ{%\label{eq:eikonal-error-ansatz}
a(t,y)=\exp\mbrko{\Phi^{(\mu)}(t,y)+Z(t,y)},
\qquad Z=X+iY.
}
We now impose the phase equation \eqref{eq:phase-equation} on $\Phi^{(\mu)}+Z$. The identity
\EQ{
Q(\Phi^{(\mu)}+Z)-Q(\Phi^{(\mu)})
=\Dy Z+2\nabla\Phi^{(\mu)}\cdot\nabla Z+\nabla Z\cdot\nabla Z
}
and the definition \eqref{eq:Rmu-def} give
\EQ{
0
&=
R^{(\mu)}+i\pd_tZ+\frac1{2t^2}\Dy Z
+\frac1{t^2}\nabla\Phi^{(\mu)}\cdot\nabla Z
+\frac1{2t^2}\nabla Z\cdot\nabla Z
\\
&\quad
-t^{-\al}e^{p\re\Phi^{(\mu)}}\brko{e^{pX}-1}.
}
Multiplying this identity by $-i$ gives the equation for the phase correction,
\EQn{\label{eq:eikonal-Z-equation}
\pd_tZ=\frac{i}{2t^2}\Dy Z+N^{(\mu)}(t,Z),
}
where
\EQn{\label{eq:Nmu-def}
N^{(\mu)}(t,Z)
=N_{\rm res}(t)+N_{\rm tri}(t,Z)+N_{\rm fir}(t,Z)+N_{\rm qua}(t,Z),
}
with
\EQn{\label{eq:eikonal-terms-decomp}
N_{\rm res}(t)=iR^{(\mu)}(t),
\qquad
N_{\rm tri}(t,Z)=-i t^{-\al}e^{p\re\Phi^{(\mu)}}(e^{pX}-1),
}
the first-order term
\EQn{\label{eq:eikonal-fir-term}
N_{\rm fir}(t,Z)
&=\frac{i}{t^2}\nabla\Phi^{(\mu)}\cdot\nabla Z
\\
&=t^{-1-\al}\nabla S^{(\mu)}(t^{-\al})\cdot\nabla Z
+\frac{i}{t^2}\nabla L^{(\mu)}(t^{-\al})\cdot\nabla Z,
}
and the quadratic term
\EQ{%\label{eq:eikonal-qua-term}
N_{\rm qua}(t,Z)=\frac{i}{2t^2}\nabla Z\cdot\nabla Z.
}

The Duhamel map associated with \eqref{eq:eikonal-Z-equation} is
\EQn{\label{eq:eikonal-Duhamel-map}
(\TT_\mu Z)(t):=-\int_t^\I U(t,s)N^{(\mu)}(s,Z(s))\dd s.
}

\subsection{Normal-form reduction}%\label{sec:eikonal-normal-form}

For $T\ge2$, $0\le\rho<\rho_1$, and $\sigma>0$, define the auxiliary space
\EQ{%\label{eq:eikonal-fast-norm}
\mathcal F_T^\sigma(\rho)
:=\fbrk{H\in C([T,\I);\A^4_\rho):
\norm{H}_{\mathcal F_T^\sigma(\rho)}<\I},
}
where
\EQ{
\norm{H}_{\mathcal F_T^\sigma(\rho)}
:=\sup_{t\ge T}t^\sigma\norm{H(t)}_{\A^4_\rho}.
}
Define
\EQ{%\label{eq:eikonal-PF-def}
P_\mu(t)H&:=t^{-1-\al}\nabla S^{(\mu)}(t^{-\al})\cdot\nabla H,
\\
F_\mu(t,X)&:=t^{-\al}e^{p\re\Phi^{(\mu)}(t)}(e^{pX}-1).
}
Then
\EQ{%\label{eq:eikonal-tri-normal-form}
N_{\rm tri}(t,Z)=-iF_\mu(t,X),\qquad X=\re Z.
}
We split the first-order term in \eqref{eq:eikonal-fir-term} as
\EQ{%\label{eq:eikonal-fir-normal-split}
N_{\rm fir}(t,Z)=P_\mu(t)Z+N_{\rm fir}^{\rm sh}(t,Z),\qquad
N_{\rm fir}^{\rm sh}(t,Z):=\frac{i}{t^2}\nabla L^{(\mu)}(t^{-\al})\cdot\nabla Z.
}
For a real-valued $X$, we look for a real-valued function $K$ satisfying the final-state equation
\EQn{\label{eq:eikonal-K-integral}
K(t)=\int_t^\I F_\mu(s,X(s))\dd s-\int_t^\I P_\mu(s)K(s)\dd s.
}
When it exists, this function is denoted by $K_\mu(X)$. Differentiating \eqref{eq:eikonal-K-integral} gives
\EQ{%\label{eq:eikonal-K-equation}
\pd_tK=P_\mu(t)K-F_\mu(t,X),\qquad K(\I)=0.
}
All construction, stability, and compatibility properties of $K_\mu$ are collected in Lemma \ref{lem:eikonal-K-tail}. Let
\EQn{\label{eq:eikonal-normal-decomposition}
Z=\widetilde Z+iK,\qquad K=K_\mu(\re\widetilde Z).
}
Here $\wt Z$ represents the new unknown after applying normal form transform. Since $K$ is real-valued, $\re Z=\re\widetilde Z$. By \eqref{eq:eikonal-normal-decomposition} and \eqref{eq:eikonal-Z-equation},
\EQn{\label{eq:eikonal-normal-reduction-identity}
\pd_t\widetilde Z+i\pd_tK
&=\frac{i}{2t^2}\Dy\widetilde Z-\frac1{2t^2}\Dy K+N_{\rm res}-iF_\mu(t,\re\widetilde Z)
\\
&\quad +P_\mu(t)\widetilde Z+iP_\mu(t)K+N_{\rm fir}^{\rm sh}(t,\widetilde Z+iK)+N_{\rm qua}(t,\widetilde Z+iK).
}
By \eqref{eq:eikonal-normal-reduction-identity} and $i\pd_tK=iP_\mu(t)K-iF_\mu(t,\re\widetilde Z)$,
\EQ{%\label{eq:eikonal-normal-equation}
\pd_t\widetilde Z
&=\frac{i}{2t^2}\Dy\widetilde Z+P_\mu(t)\widetilde Z+N_{\rm res}(t)
\\
&\quad+N_{\rm fir}^{\rm sh}(t,\widetilde Z+iK)+N_{\rm qua}(t,\widetilde Z+iK)-\frac1{2t^2}\Dy K.
}
The corresponding final-state equation for the new unknown $\wt Z$ is
\EQn{\label{eq:eikonal-normal-duhamel}
\widetilde Z(t)
&=-\int_t^\I U(t,s)\mbrko{P_\mu(s)\widetilde Z(s)+N_{\rm res}(s)}\dd s
\\
&\quad-\int_t^\I U(t,s)\mbrko{N_{\rm fir}^{\rm sh}(s,\widetilde Z(s)+iK(s))+N_{\rm qua}(s,\widetilde Z(s)+iK(s))}\dd s
\\
&\quad+\frac12\int_t^\I U(t,s)s^{-2}\Dy K(s)\dd s.
}
Here $K(s)=K_\mu(\re\widetilde Z)(s)$ is given in \eqref{eq:eikonal-K-integral}, and the terms in \eqref{eq:eikonal-normal-duhamel} are
\EQ{
P_\mu(s)\widetilde Z(s)
&=s^{-1-\al}\nabla S^{(\mu)}(s^{-\al})\cdot\nabla\widetilde Z(s),
\\
N_{\rm res}(s)
&=i\mbrko{
i\pd_s\Phi^{(\mu)}(s)+\frac1{2s^2}Q(\Phi^{(\mu)}(s))
- s^{-\al}e^{p\re\Phi^{(\mu)}(s)}
},
\\
N_{\rm fir}^{\rm sh}(s,\widetilde Z(s)+iK(s))
&=\frac{i}{s^2}\nabla L^{(\mu)}(s^{-\al})\cdot
\nabla\brko{\widetilde Z(s)+iK(s)},
\\
N_{\rm qua}(s,\widetilde Z(s)+iK(s))
&=\frac{i}{2s^2}
\nabla\brko{\widetilde Z(s)+iK(s)}
\cdot\nabla\brko{\widetilde Z(s)+iK(s)}.
}

\subsection{Estimates for normal form transform}%\label{sec:eikonal-normal-estimates}

For the fixed truncation order $\mu$, define
\EQ{%\label{eq:eikonal-JP-def}
(J_PH)(t):=\int_t^\I P_\mu(s)H(s)\dd s
=\int_t^\I s^{-1-\al}\nabla S^{(\mu)}(s^{-\al})\cdot\nabla H(s)\dd s.
}
Expanding $F_\mu$ in powers of its second argument gives
\EQn{\label{eq:eikonal-F-expansion}
t^{-\al}e^{p\re\Phi^{(\mu)}(t)}\brko{e^{pX}-1}
&=\sum_{m\ge1}F_{\mu,m}(t;X,\ldots,X),
}
where
\EQ{
F_{\mu,m}(t;H_1,\ldots,H_m)
&:=t^{-\al}e^{p\re\Phi^{(\mu)}(t)}\frac{p^m}{m!}
\re H_1\cdots\re H_m.
}
Set
\EQn{\label{eq:eikonal-am-definition}
a_m:=\frac{(C_4p)^m}{m!},
\qquad m\ge1,
}
and define the real $m$-linear transport coefficients recursively by
\EQn{\label{eq:eikonal-Kml-def}
K_{m,0}(H_1,\ldots,H_m)(t)&:=\int_t^\I
F_{\mu,m}(s;H_1(s),\ldots,H_m(s))\dd s,
\\
K_{m,\ell+1}(H_1,\ldots,H_m)&:=-J_PK_{m,\ell}(H_1,\ldots,H_m),
\qquad \ell\ge0.
}

\begin{lem}\label{lem:eikonal-K-tail}
The following statements hold.
\begin{enumerate}[label=\textup{(\roman*)},leftmargin=*]
\item \emph{(Multilinear transport coefficients)}.
Let $m\ge1$, $\ell\ge0$, and let $\zeta_h$, $0\le h\le\ell$, be analytic radii satisfying
\EQ{
0\le\zeta_0,
\qquad
\zeta_{h-1}<\zeta_h\quad(1\le h\le\ell),
\qquad
\zeta_\ell<\rho_1.
}
Then
\EQn{\label{eq:eikonal-K-flexible-estimate}
&\norm{K_{m,\ell}(G_1,\ldots,G_m)(t)}_{\A^4_{\zeta_0}} \\
&\le C_{\mu,R}^{\ell+1}\ep a_m
\mbrkbb{\prod_{h=1}^\ell(\zeta_h-\zeta_{h-1})}^{-1}
\\
&\quad\times\int_{t\le s_1\le\cdots\le s_\ell\le\theta<\I}
s_1^{-1-\al}\cdots s_\ell^{-1-\al}\theta^{-\al}
\norm{G_1(\theta)}_{\A^4_{\zeta_\ell}}\cdots
\norm{G_m(\theta)}_{\A^4_{\zeta_\ell}}
\dd\theta\dd s_\ell\cdots\dd s_1.
}
For $\ell=0$, the gap product and the intermediate time variables are
absent, and the integral is taken over $t\le\theta<\I$.

\item \emph{(Construction and size estimate)}.
Let $0\le\rho_-<\rho_+<\rho_1$ and assume
\EQn{\label{eq:eikonal-K-smallness-radius}
C_{\mu,R}T^{-\al}(\rho_+-\rho_-)^{-1}\le\frac14.
}
If $X$ is real-valued and
$\norm{X}_{\mathcal F_T^\kappa(\rho_+)}\le M$, then the series
\EQn{\label{eq:eikonal-K-double-series}
K_\mu(X)(t)=\sum_{m\ge1}\sum_{\ell\ge0}
K_{m,\ell}(X,\ldots,X)(t)
}
converges absolutely in
$\mathcal F_T^{\kappa+\al-1}(\rho_-)$ and satisfies
\EQn{\label{eq:eikonal-K-tail-bound}
\norm{K_\mu(X)}_{\mathcal F_T^{\kappa+\al-1}(\rho_-)}
\lsm_{\mu,R,M}\ep.
}
For every $0\le\rho'<\rho_-$, its sum solves
\eqref{eq:eikonal-K-integral} in
$C([T,\I);\A^4_{\rho'})$.

\item \emph{(Stability with respect to the input)}.
Let $0\le\rho_-<\rho_+<\rho_1$ and assume
\eqref{eq:eikonal-K-smallness-radius}. Let $X_j$, $j=1,2$, be
real-valued and satisfy
\EQ{
\norm{X_j}_{\mathcal F_T^\kappa(\rho_+)}\le M,
\qquad j=1,2.
}
If
$\norm{X_1-X_2}_{\mathcal F_T^{\sigma}(\rho_+)}<\I$
for some $\sigma>1-\al$, then
\EQn{\label{eq:eikonal-K-tail-lip}
\norm{K_\mu(X_1)-K_\mu(X_2)}
_{\mathcal F_T^{\sigma+\al-1}(\rho_-)}
\lsm_{\mu,R,M,\sigma}\ep
\norm{X_1-X_2}_{\mathcal F_T^{\sigma}(\rho_+)}.
}

\item \emph{(Dependence on the truncation order)}.
Let $0\le\rho_-<\rho_+<\rho_1$. Let
$\mu_1,\mu_2>2\gamma$ be integers, put
$\bar\mu=\min\fbrko{\mu_1,\mu_2}$, and set
$\sigma_0=\al\bar\mu$. For $j=1,2$, let
$(L^{(\mu_j)},S^{(\mu_j)})$ be the truncated Fuchsian profile of order
$\mu_j$ constructed from the same Fuchsian profile parameters $(\Psi,C_+)$, and let
$K_{\mu_j}$ be defined by \eqref{eq:eikonal-K-integral} using this
Fuchsian profile. Increase $T$, if necessary, so that
\EQn{\label{eq:eikonal-K-two-order-smallness}
C_{\mu_1,\mu_2,R}T^{-\al}(\rho_+-\rho_-)^{-1}
\le\frac1{12}.
}
There exists $N=N(\mu_1,\mu_2)\in\N$ such that, if $X$ is
real-valued and
$\norm{X}_{\mathcal F_T^\kappa(\rho_+)}\le M$, then
\EQn{\label{eq:eikonal-K-truncation-diff}
\norm{K_{\mu_2}(X)(t)-K_{\mu_1}(X)(t)}
_{\A^4_{\rho_-}}
\lsm_{\mu_1,\mu_2,R,M}
t^{1-\al-\sigma_0-\kappa}(1+\log t)^N.
}

\item \emph{(Compatibility of the construction parameters)}.
For $j=1,2$, let $T_j\ge2$ and
$0\le\zeta_{j,-}<\zeta_{j,+}<\rho_1$. Suppose that the
same real-valued function $X$ belongs to
$\mathcal F_{T_j}^\kappa(\zeta_{j,+})$ and that, for each $j$, the
double series in \eqref{eq:eikonal-K-double-series} converges absolutely
in $\mathcal F_{T_j}^{\kappa+\al-1}(\zeta_{j,-})$. Denote its sum
on $[T_j,\I)$ by $K^{(j)}_\mu(X)$. Then
\EQ{%\label{eq:eikonal-K-compatibility}
K^{(1)}_\mu(X)=K^{(2)}_\mu(X)
\quad\hbox{in }C([\max\{T_1,T_2\},\I);\A^4_\zeta)
}
for every
$0\le\zeta\le\min\{\zeta_{1,-},\zeta_{2,-}\}$.
\end{enumerate}
\end{lem}

\begin{proof}
For \textup{(i)}, the radius-gap estimate and
\eqref{eq:eikonal-profile-bounds} give, whenever
$0\le\zeta_-<\zeta_+<\rho_1$,
\EQn{\label{eq:eikonal-P-estimate}
\norm{P_\mu(t)H}_{\A^4_{\zeta_-}}
\le C_{\mu,R}t^{-1-\al}(\zeta_+-\zeta_-)^{-1}
\norm{H}_{\A^4_{\zeta_+}}.
}
By \eqref{eq:eikonal-F-expansion} and
\eqref{eq:eikonal-amplitude-small}, for every $0\le\zeta<\rho_1$,
\EQn{\label{eq:eikonal-Fm-estimate}
\norm{F_{\mu,m}(t;H_1,\ldots,H_m)}_{\A^4_\zeta}
&\le t^{-\al}\frac{p^m}{m!}C_4^m
\norm{e^{p\re\Phi^{(\mu)}(t)}}_{\A^4_\zeta}
\prod_{j=1}^m\norm{H_j(t)}_{\A^4_\zeta}
\\
&\le C_{\mu,R}\ep a_m t^{-\al}
\prod_{j=1}^m\norm{H_j(t)}_{\A^4_\zeta}.
}
Induction in \eqref{eq:eikonal-Kml-def} gives
\EQ{%\label{eq:eikonal-Kml-iteration}
K_{m,\ell}=(-J_P)^\ell K_{m,0},
\qquad \ell\ge0,
}
and hence
\EQn{\label{eq:eikonal-K-ordered-expansion}
&K_{m,\ell}(G_1,\ldots,G_m)(t)
\\
&=(-1)^\ell
\int_{t\le s_1\le\cdots\le s_\ell\le\theta<\I}
P_\mu(s_1)\cdots P_\mu(s_\ell)
F_{\mu,m}(\theta;G_1(\theta),\ldots,G_m(\theta))
\dd\theta\dd s_\ell\cdots\dd s_1.
}
For $t\le s_1\le\cdots\le s_\ell\le\theta$, successive applications
of \eqref{eq:eikonal-P-estimate} along the prescribed radius chain give
\EQn{\label{eq:eikonal-K-P-iteration}
&\norm{P_\mu(s_1)\cdots P_\mu(s_\ell)
F_{\mu,m}(\theta;G_1(\theta),\ldots,G_m(\theta))}
_{\A^4_{\zeta_0}}
\\
&\quad\le C_{\mu,R}^\ell
(s_1\cdots s_\ell)^{-1-\al}
\mbrk{\prod_{h=1}^\ell(\zeta_h-\zeta_{h-1})}^{-1} 
\norm{F_{\mu,m}(\theta;G_1(\theta),\ldots,G_m(\theta))}
_{\A^4_{\zeta_\ell}}.
}
Moreover, \eqref{eq:eikonal-Fm-estimate} yields
\EQn{\label{eq:eikonal-K-terminal-estimate}
\norm{F_{\mu,m}(\theta;G_1(\theta),\ldots,G_m(\theta))}
_{\A^4_{\zeta_\ell}} \le C_{\mu,R}\ep a_m\theta^{-\al}
\prod_{j=1}^m\norm{G_j(\theta)}_{\A^4_{\zeta_\ell}}.
}
Combining \eqref{eq:eikonal-K-ordered-expansion},
\eqref{eq:eikonal-K-P-iteration}, and
\eqref{eq:eikonal-K-terminal-estimate} proves
\eqref{eq:eikonal-K-flexible-estimate}.

For \textup{(ii)}, rewrite \eqref{eq:eikonal-K-integral} as
\EQn{\label{eq:eikonal-K-integral-expanded}
K_\mu(X)(t)
&=\int_t^\I s^{-\al}e^{p\re\Phi^{(\mu)}(s)}\brko{e^{pX(s)}-1}\dd s
\\
&\quad-\int_t^\I s^{-1-\al}\nabla S^{(\mu)}(s^{-\al})\cdot\nabla K_\mu(X)(s)\dd s.
}
Then \eqref{eq:eikonal-K-integral-expanded} becomes
\EQn{\label{eq:eikonal-K-operator-form}
\brko{I+J_P}K_\mu(X)(t)
=\int_t^\I s^{-\al}e^{p\re\Phi^{(\mu)}(s)}\brko{e^{pX(s)}-1}\dd s.
}
For integers $N\ge1$ and $J\ge0$, the recursion in
\eqref{eq:eikonal-Kml-def} gives
\EQn{\label{eq:eikonal-K-finite-telescoping}
&\brko{I+J_P}\sum_{m=1}^N\sum_{\ell=0}^JK_{m,\ell}(X,\ldots,X)
\\
&\quad=\sum_{m=1}^N\sum_{\ell=0}^J
\mbrko{K_{m,\ell}(X,\ldots,X)+J_PK_{m,\ell}(X,\ldots,X)}
\\
&\quad=\sum_{m=1}^N\sum_{\ell=0}^J
\mbrko{K_{m,\ell}(X,\ldots,X)-K_{m,\ell+1}(X,\ldots,X)}
\\
&\quad=\sum_{m=1}^NK_{m,0}(X,\ldots,X)
-\sum_{m=1}^NK_{m,J+1}(X,\ldots,X).
}
The passage $J\ra\I$ requires
\EQn{\label{eq:eikonal-K-remainder-condition}
\lim_{J\ra\I}\normB{\sum_{m=1}^NK_{m,J+1}(X,\ldots,X)}
_{\mathcal F_T^{\kappa+\al-1}(\rho_-)}=0,
\qquad N\ge1.
}
We prove \eqref{eq:eikonal-K-remainder-condition} quantitatively below. By \eqref{eq:eikonal-F-expansion} and \eqref{eq:eikonal-Kml-def}, the right-hand side of \eqref{eq:eikonal-K-operator-form} has the expansion
\EQn{\label{eq:eikonal-K-forcing-series}
&\int_t^\I s^{-\al}e^{p\re\Phi^{(\mu)}(s)}\brko{e^{pX(s)}-1}\dd s
\\
&=\int_t^\I\sum_{m\ge1}F_{\mu,m}(s;X(s),\ldots,X(s))\dd s
=\sum_{m\ge1}K_{m,0}(X,\ldots,X)(t).
}
Assuming \eqref{eq:eikonal-K-remainder-condition} for the moment, we
first let $J\ra\I$ and then $N\ra\I$ in
\eqref{eq:eikonal-K-finite-telescoping}. By
\eqref{eq:eikonal-K-forcing-series}, we obtain formally
\EQ{%\label{eq:eikonal-K-formal-solution}
&\brko{I+J_P}\sum_{m\ge1}\sum_{\ell\ge0}K_{m,\ell}(X,\ldots,X)(t)
\\
&\quad=\sum_{m\ge1}K_{m,0}(X,\ldots,X)(t)
=\int_t^\I s^{-\al}e^{p\re\Phi^{(\mu)}(s)}\brko{e^{pX(s)}-1}\dd s.
}
Thus \eqref{eq:eikonal-K-double-series} is the formal solution of
\eqref{eq:eikonal-K-operator-form}. The estimates below justify the
passage $J\ra\I$, followed by $N\ra\I$, in
\eqref{eq:eikonal-K-finite-telescoping} at every strictly smaller
analytic radius.

The coefficients $a_m$ satisfy
\EQn{\label{eq:eikonal-am-sum}
\sum_{m\ge1}a_mR_1^m=e^{C_4pR_1}-1<\I,
\qquad 0\le R_1<\I.
}

To prove \eqref{eq:eikonal-K-tail-bound}, let $\ell\ge1$ and choose
\EQ{%\label{eq:eikonal-K-equal-radii}
\zeta_h:=\rho_-+\frac{h}{\ell+1}(\rho_+-\rho_-),
\qquad h=0,1,\ldots,\ell+1.
}
Thus $\zeta_0=\rho_-$, $\zeta_{\ell+1}=\rho_+$, and
\EQn{\label{eq:eikonal-K-equal-gap}
\zeta_h-\zeta_{h-1}
=\frac{\rho_+-\rho_-}{\ell+1},
\qquad h=1,\ldots,\ell+1.
}
Applying \eqref{eq:eikonal-K-flexible-estimate} along this chain,
using \eqref{eq:eikonal-K-equal-gap} and monotonicity from
$\rho_+$ to $\zeta_\ell$, gives
\EQn{\label{eq:eikonal-K-ell-estimate}
&\norm{K_{m,\ell}(H_1,\ldots,H_m)(t)}
_{\A^4_{\rho_-}}
\\
&\le C_{\mu,R}^{\ell+1}\ep a_m(\ell+1)^\ell
(\rho_+-\rho_-)^{-\ell}
\\
&\quad\times\int_{t\le s_1\le\cdots\le s_\ell\le\theta<\I}
s_1^{-1-\al}\cdots s_\ell^{-1-\al}\theta^{-\al}
\prod_{j=1}^m\norm{H_j(\theta)}_{\A^4_{\rho_+}}
\dd\theta\dd s_\ell\cdots\dd s_1.
}
The case $\ell=0$ follows directly from
\eqref{eq:eikonal-K-flexible-estimate}; hence
\eqref{eq:eikonal-K-ell-estimate} holds for every $\ell\ge0$, with
the usual empty-product convention. Now set $H_j=X$. Then
\EQ{
\norm{X(\theta)}_{\A^4_{\rho_+}}^m
\le M^m\theta^{-m\kappa}.
}
Successive integration, together with $m\ge1$ and $\kappa>1-\al$, gives
\EQn{\label{eq:eikonal-K-time-simplex}
\int_{t\le s_1\le\cdots\le s_\ell\le\theta<\I}
s_1^{-1-\al}\cdots s_\ell^{-1-\al}\theta^{-\al-m\kappa}
\dd\theta\dd s_\ell\cdots\dd s_1
\le C_{\al,\kappa}^{\ell+1}(\ell!)^{-1}
 t^{1-\al-m\kappa-\ell\al}.
}
By \eqref{eq:eikonal-K-ell-estimate} and \eqref{eq:eikonal-K-time-simplex}, together with
\EQ{
\frac{(\ell+1)^\ell}{\ell!}\le C^\ell,
\qquad
t^{-\ell\al}\le T^{-\ell\al},
\qquad
t^{-(m-1)\kappa}\le1,
}
we obtain
\EQ{%\label{eq:eikonal-Kml-X-bound}
\norm{K_{m,\ell}(X,\ldots,X)(t)}_{\A^4_{\rho_-}}
\lsm_{\mu,R}\ep a_mM^m
\brko{\frac{C_{\mu,R}T^{-\al}}{\rho_+-\rho_-}}^\ell
 t^{1-\al-m\kappa},
}
and hence
\EQn{\label{eq:eikonal-Kml-F-bound}
\norm{K_{m,\ell}(X,\ldots,X)}
_{\mathcal F_T^{\kappa+\al-1}(\rho_-)}
\lsm_{\mu,R}\ep a_mM^m
\brko{\frac{C_{\mu,R}T^{-\al}}{\rho_+-\rho_-}}^\ell.
}
By \eqref{eq:eikonal-Kml-F-bound} and \eqref{eq:eikonal-K-smallness-radius},
\EQn{\label{eq:eikonal-K-remainder-decay}
&\sup_{N\ge1}\norm{\sum_{m=1}^NK_{m,J+1}(X,\ldots,X)}
_{\mathcal F_T^{\kappa+\al-1}(\rho_-)}
\\
&\quad\le\sum_{m\ge1}\norm{K_{m,J+1}(X,\ldots,X)}
_{\mathcal F_T^{\kappa+\al-1}(\rho_-)}
\\
&\quad\lsm_{\mu,R}\ep\brko{\sum_{m\ge1}a_mM^m}4^{-(J+1)}\ra0,
\qquad J\ra\I.
}
Thus \eqref{eq:eikonal-K-remainder-condition} holds. Moreover, \eqref{eq:eikonal-K-smallness-radius} gives
\EQn{\label{eq:eikonal-K-absolute-sum}
\sum_{m\ge1}\sum_{\ell\ge0}
\norm{K_{m,\ell}(X,\ldots,X)}
_{\mathcal F_T^{\kappa+\al-1}(\rho_-)}
\lsm_{\mu,R}\ep
\brko{\sum_{m\ge1}a_mM^m}
\brko{\sum_{\ell\ge0}4^{-\ell}}<\I.
}
Thus the double series in \eqref{eq:eikonal-K-double-series} converges
absolutely in $\mathcal F_T^{\kappa+\al-1}(\rho_-)$, and
\eqref{eq:eikonal-K-tail-bound} holds. Fix
$0\le\rho'<\rho_-$. By \eqref{eq:eikonal-P-estimate},
\eqref{eq:eikonal-Fm-estimate}, and
\eqref{eq:eikonal-K-absolute-sum},
\EQn{\label{eq:eikonal-K-limit-justification}
&\sum_{m\ge1}\sum_{\ell\ge0}
\norm{J_PK_{m,\ell}(X,\ldots,X)}_{\mathcal F_T^{\kappa+\al-1}(\rho')}
\\
&\quad\lsm_{\mu,R}T^{-\al}(\rho_--\rho')^{-1}
\sum_{m\ge1}\sum_{\ell\ge0}
\norm{K_{m,\ell}(X,\ldots,X)}
_{\mathcal F_T^{\kappa+\al-1}(\rho_-)}<\I,
\\
&\sum_{m\ge1}\int_t^\I
\norm{F_{\mu,m}(s;X(s),\ldots,X(s))}
_{\A^4_{\rho_+}}\dd s
\lsm_{\mu,R}\ep\sum_{m\ge1}a_mM^m
\int_t^\I s^{-\al-m\kappa}\dd s<\I.
}
By \eqref{eq:eikonal-K-absolute-sum},
\eqref{eq:eikonal-K-remainder-decay}, and
\eqref{eq:eikonal-K-limit-justification}, we may let $J\ra\I$ and
then $N\ra\I$ in \eqref{eq:eikonal-K-finite-telescoping}. By
\eqref{eq:eikonal-K-forcing-series}, at the radius $\rho'$,
\EQ{
\brko{I+J_P}K_\mu(X)(t)
=\sum_{m\ge1}K_{m,0}(X,\ldots,X)(t)
=\int_t^\I s^{-\al}e^{p\re\Phi^{(\mu)}(s)}\brko{e^{pX(s)}-1}\dd s.
}
Thus $K_\mu(X)$ solves \eqref{eq:eikonal-K-integral} in
$C([T,\I);\A^4_{\rho'})$ for every
$0\le\rho'<\rho_-$.

For \textup{(iii)}, since $P_\mu$ is independent of $X$, subtracting
\eqref{eq:eikonal-K-integral} for $X_1$ and $X_2$ amounts to applying
the same linear iteration to $F_\mu(X_1)-F_\mu(X_2)$. The identity
\EQ{
e^{pX_1}-e^{pX_2}
=p(X_1-X_2)\int_0^1
e^{p(\vartheta X_1+(1-\vartheta)X_2)}\dd\vartheta
}
and Lemma \ref{lem:fourier-algebra} give
\EQn{\label{eq:eikonal-F-lip-proof}
\norm{F_\mu(t,X_1(t))-F_\mu(t,X_2(t))}
_{\A^4_{\rho_+}}
\lsm_{\mu,R,M,\sigma}\ep t^{-\al-\sigma}
\norm{X_1-X_2}_{\mathcal F_T^{\sigma}(\rho_+)}.
}
Since $\sigma>1-\al$, the corresponding time integral is finite. More precisely, for every $\ell\ge0$,
\EQn{\label{eq:eikonal-K-lip-time-simplex}
\int_{t\le s_1\le\cdots\le s_\ell\le\theta<\I}
s_1^{-1-\al}\cdots s_\ell^{-1-\al}\theta^{-\al-\sigma}
\dd\theta\dd s_\ell\cdots\dd s_1
\le C_{\al,\sigma}^{\ell+1}(\ell!)^{-1}
 t^{1-\al-\sigma-\ell\al}.
}
Dividing $[\rho_-,\rho_+]$ into $\ell+1$ equal parts and applying
\eqref{eq:eikonal-P-estimate}, \eqref{eq:eikonal-F-lip-proof}, and
\eqref{eq:eikonal-K-lip-time-simplex}, we obtain
\EQn{\label{eq:eikonal-K-lip-ell-bound}
&\normB{(-J_P)^\ell\brko{\int_{t}^\I
\brko{F_\mu(s,X_1(s))-F_\mu(s,X_2(s))}\dd s}}
_{\mathcal F_T^{\sigma+\al-1}(\rho_-)}
\\
&\quad\lsm_{\mu,R,M,\sigma}\ep
\norm{X_1-X_2}_{\mathcal F_T^{\sigma}(\rho_+)}
\brko{\frac{C_{\mu,R}T^{-\al}}{\rho_+-\rho_-}}^\ell.
}
The same iteration gives the difference series
\EQ{%\label{eq:eikonal-K-lip-series}
K_\mu(X_1)-K_\mu(X_2)
=\sum_{\ell\ge0}(-J_P)^\ell
\brko{\int_{\,\cdot}^\I
\brko{F_\mu(s,X_1(s))-F_\mu(s,X_2(s))}\dd s}
}
with absolute convergence in
$\mathcal F_T^{\sigma+\al-1}(\rho_-)$. Summing
\eqref{eq:eikonal-K-lip-ell-bound} over $\ell$ and using
\eqref{eq:eikonal-K-smallness-radius} proves
\eqref{eq:eikonal-K-tail-lip}.

For \textup{(iv)}, Corollary
\ref{cor:eikonal-auxiliary-bounds}\textup{(ii)} gives
\eqref{eq:eikonal-separated-profile-diff}. Let
\EQ{%\label{eq:eikonal-K-truncation-radii}
\zeta_1=\rho_-+\frac{\rho_+-\rho_-}3,
\qquad
\zeta_2=\rho_-+\frac{2(\rho_+-\rho_-)}3.
}
Choose the constant in \eqref{eq:eikonal-K-two-order-smallness} to dominate the constants for both truncation orders. Since
\EQ{
\zeta_1-\rho_-=\zeta_2-\zeta_1=\rho_+-\zeta_2
=\frac{\rho_+-\rho_-}3,
}
condition \eqref{eq:eikonal-K-two-order-smallness} gives
\EQ{%\label{eq:eikonal-K-truncation-subgap-smallness}
C_{\mu_1,R}T^{-\al}(\zeta_2-\zeta_1)^{-1}\le\frac14,
\qquad
C_{\mu_2,R}T^{-\al}(\rho_+-\zeta_2)^{-1}\le\frac14.
}
Part \textup{(ii)} applies to $K_{\mu_1}(X)$ from $\zeta_2$ to $\zeta_1$ and to $K_{\mu_2}(X)$ from $\rho_+$ to $\zeta_2$.
In particular,
\EQn{\label{eq:eikonal-Kmu2-intermediate-radius}
\norm{K_{\mu_2}(X)(t)}_{\A^4_{\zeta_2}}
\lsm_{\mu_2,R,M}\ep t^{1-\al-\kappa}.
}
For every $0\le\rho'<\zeta_1$, both identities hold in $C([T,\I);\A^4_{\rho'})$ and take the form
\EQ{
K_{\mu_j}(X)(t)=\int_t^\I F_{\mu_j}(s,X(s))\dd s-
\int_t^\I P_{\mu_j}(s)K_{\mu_j}(X)(s)\dd s,
\qquad j=1,2.
}
Subtracting these identities gives
\EQn{\label{eq:eikonal-K-truncation-equation}
&\brko{K_{\mu_2}(X)-K_{\mu_1}(X)}(t)
\\
&=\int_t^\I \brko{F_{\mu_2}-F_{\mu_1}}(s,X(s))\dd s
-\int_t^\I P_{\mu_1}(s)\brko{K_{\mu_2}(X)-K_{\mu_1}(X)}(s)\dd s
\\
&\quad-\int_t^\I \brko{P_{\mu_2}-P_{\mu_1}}(s)K_{\mu_2}(X)(s)\dd s.
}
The separated profile estimate \eqref{eq:eikonal-separated-profile-diff} gives
\EQn{\label{eq:eikonal-F-truncation-diff}
\norm{\brko{F_{\mu_2}-F_{\mu_1}}(t,X(t))}
_{\A^4_{\zeta_1}}
\lsm_{\mu_1,\mu_2,R,M}t^{-\al-\sigma_0-\kappa}(1+\log t)^N.
}
Moreover, \eqref{eq:radius-gap-estimate}, \eqref{eq:eikonal-separated-profile-diff}, and \eqref{eq:eikonal-Kmu2-intermediate-radius} yield
\EQn{\label{eq:eikonal-P-truncation-diff}
&\norm{\brko{P_{\mu_2}-P_{\mu_1}}(t)K_{\mu_2}(X)(t)}
_{\A^4_{\zeta_1}}
\\
&\quad\lsm_{\mu_1,\mu_2,R,M}(\zeta_2-\zeta_1)^{-1}
 t^{-1-\al-\sigma_0}t^{1-\al-\kappa}(1+\log t)^N
\\
&\quad\lsm_{\mu_1,\mu_2,R,M}t^{-\al-\sigma_0-\kappa}(1+\log t)^N,
}
where the last line uses \(t^{-2\al}\le t^{-\al}\) for \(t\ge1\). Since $\al+\sigma_0+\kappa>1$, integration of \eqref{eq:eikonal-F-truncation-diff} and \eqref{eq:eikonal-P-truncation-diff} gives
\EQn{\label{eq:eikonal-K-truncation-forcing-integrals}
&\normB{\int_t^\I \brko{F_{\mu_2}-F_{\mu_1}}(s,X(s))\dd s
-\int_t^\I \brko{P_{\mu_2}-P_{\mu_1}}(s)K_{\mu_2}(X)(s)\dd s}
_{\A^4_{\zeta_1}}
\\
&\quad \lsm_{\mu_1,\mu_2,R,M}
 t^{1-\al-\sigma_0-\kappa}(1+\log t)^N.
}
Let
\EQ{
G(t)&:=\int_t^\I \brko{F_{\mu_2}-F_{\mu_1}}(s,X(s))\dd s
-\int_t^\I \brko{P_{\mu_2}-P_{\mu_1}}(s)K_{\mu_2}(X)(s)\dd s.
\\
H(t)&:=K_{\mu_2}(X)(t)-K_{\mu_1}(X)(t).
}
Then \eqref{eq:eikonal-K-truncation-equation} is
\EQ{
H=G-J_{P_{\mu_1}}H,
}
Here $J_{P_{\mu_1}}H(t)=\int_t^\I P_{\mu_1}(s)H(s)\dd s$. For $n\ge1$, iteration along radii between $\rho_-$ and $\zeta_1$ gives
\EQ{
H=\sum_{\ell=0}^{n-1}(-J_{P_{\mu_1}})^\ell G+(-J_{P_{\mu_1}})^nH
}
in $\mathcal F_T^{\kappa+\al-1}(\rho_-)$, and the estimates in \textup{(ii)} give
\EQ{
\norm{(-J_{P_{\mu_1}})^nH}_{\mathcal F_T^{\kappa+\al-1}(\rho_-)}
\lsm 4^{-n}\norm{H}_{\mathcal F_T^{\kappa+\al-1}(\zeta_1)}
\ra0.
}
Letting $n\ra\I$ gives
\EQn{\label{eq:eikonal-K-truncation-series}
K_{\mu_2}(X)-K_{\mu_1}(X)
=\sum_{\ell\ge0}(-J_{P_{\mu_1}})^\ell G.
}
For $\ell\ge1$, the ordered time integral arising from $\ell$ applications of $J_{P_{\mu_1}}$ satisfies
\EQn{\label{eq:eikonal-K-truncation-time-simplex}
&\int_{t\le s_1\le\cdots\le s_\ell<\I}
s_1^{-1-\al}\cdots s_\ell^{-1-\al}
 s_\ell^{1-\al-\sigma_0-\kappa}(1+\log s_\ell)^N
\dd s_\ell\cdots\dd s_1
\\
&\quad\le \frac{C^\ell}{\ell!}
 t^{1-\al-\sigma_0-\kappa-\ell\al}(1+\log t)^N.
}
For $\ell=0$, use
\eqref{eq:eikonal-K-truncation-forcing-integrals} directly. Dividing
$[\rho_-,\zeta_1]$ into $\ell+1$ equal parts and applying
\eqref{eq:eikonal-P-estimate},
\eqref{eq:eikonal-K-truncation-forcing-integrals}, and
\eqref{eq:eikonal-K-truncation-time-simplex} yields
\EQn{\label{eq:eikonal-K-truncation-ell-bound}
\norm{(-J_{P_{\mu_1}})^\ell G(t)}_{\A^4_{\rho_-}}
&\lsm_{\mu_1,\mu_2,R,M}
\brko{\frac{C_{\mu_1,R}T^{-\al}}{\zeta_1-\rho_-}}^\ell
 t^{1-\al-\sigma_0-\kappa}(1+\log t)^N
\\
&\lsm_{\mu_1,\mu_2,R,M}4^{-\ell}
 t^{1-\al-\sigma_0-\kappa}(1+\log t)^N.
}
The series in \eqref{eq:eikonal-K-truncation-series} is therefore
absolutely convergent at the radius $\rho_-$. Summing
\eqref{eq:eikonal-K-truncation-ell-bound} over $\ell$ proves
\eqref{eq:eikonal-K-truncation-diff}.

For \textup{(v)}, let $T_0:=\max\{T_1,T_2\}$ and fix
$0\le\zeta\le\min\{\zeta_{1,-},\zeta_{2,-}\}$. For
$j=1,2$, the monotonicity of the analytic norms and $T_0\ge T_j$ give
\EQ{
&\sum_{m\ge1}\sum_{\ell\ge0}
\sup_{t\ge T_0}t^{\kappa+\al-1}
\norm{K_{m,\ell}(X,\ldots,X)(t)}_{\A^4_\zeta}
\\
&\quad\le
\sum_{m\ge1}\sum_{\ell\ge0}
\norm{K_{m,\ell}(X,\ldots,X)}
_{\mathcal F_{T_j}^{\kappa+\al-1}(\zeta_{j,-})}<\I.
}
The terms defined in \eqref{eq:eikonal-Kml-def} depend on neither $T_j$ nor the radii. Consequently,
\EQ{
K^{(1)}_\mu(X)
=\lim_{N,J\ra\I}\sum_{m=1}^N\sum_{\ell=0}^J
K_{m,\ell}(X,\ldots,X)
=K^{(2)}_\mu(X)
}
in $\mathcal F_{T_0}^{\kappa+\al-1}(\zeta)$, and hence in
$C([T_0,\I);\A^4_\zeta)$.
\end{proof}

\subsection{Estimates for iteration}%\label{sec:eikonal-iteration-estimates}

We first derive the expansion used in the Picard iteration. Define
\EQn{\label{eq:eikonal-normal-map}
(\widetilde{\TT}H)(t)=D_0(t)+\mathcal A^{(1)}(H)(t)+\mathcal A^{(2)}(H)(t),
}
where
\EQn{\label{eq:eikonal-normal-map-parts}
D_0(t)&:=-\int_t^\I U(t,s)N_{\rm res}(s)\dd s,
\\
\mathcal A^{(1)}(H)(t)&:=-\int_t^\I U(t,s)\mbrko{P_\mu(s)H(s)+N_{\rm fir}^{\rm sh}(s,H(s)+iK_\mu(\re H)(s))}\dd s
\\
&\quad-\int_t^\I U(t,s)N_{\rm qua}(s,H(s)+iK_\mu(\re H)(s))\dd s,
\\
\mathcal A^{(2)}(H)(t)&:=\frac12\int_t^\I U(t,s)s^{-2}\Dy K_\mu(\re H)(s)\dd s.
}

\begin{lem}[Duhamel form of the normal-form transformation]\label{lem:eikonal-Duhamel-conjugacy}
Let $T\ge2$, $\sigma>0$, and fix
\EQ{
0\le\rho'<\rho''<\rho<\rho_1.
}
Suppose that $H\in\mathcal F_T^\sigma(\rho)$, $\re H\in\mathcal F_T^\kappa(\rho)$, and $K=K_\mu(\re H)\in\mathcal F_T^{\kappa+\al-1}(\rho'')$ satisfies \eqref{eq:eikonal-K-integral} at radius $\rho''$. Then the integrals defining $\widetilde{\TT}H$ and $\TT_\mu(H+iK)$ converge in $C([T,\I);\A^4_{\rho'})$, and
\EQ{
\TT_\mu(H+iK)=\widetilde{\TT}H+iK
\quad\hbox{in }C([T,\I);\A^4_{\rho'}).
}
Consequently,
\EQn{\label{eq:eikonal-Duhamel-conjugacy}
H=\widetilde{\TT}H
\quad\Longleftrightarrow\quad
H+iK=\TT_\mu(H+iK)
\quad\hbox{in }C([T,\I);\A^4_{\rho'}).
}
\end{lem}

\begin{proof}
Set
\EQ{
B_\mu(s):=F_\mu(s,\re H(s))-P_\mu(s)K(s).
}
The integral equation for $K$ reads
\EQ{
K(t)=\int_t^\I B_\mu(s)\dd s.
}
Choose $\rho'<\zeta<\rho''$. The analytic estimates above give
\EQ{
\norm{B_\mu(s)}_{\A^4_\zeta}
\lsm s^{-\al-\kappa},
\qquad
B_\mu\in L^1([T,\I);\A^4_\zeta).
}
Since
\EQ{
i\brko{U(t,s)-I}
=\frac12\int_t^s r^{-2}U(t,r)\Dy\dd r,
}
the radius-gap estimate gives
\EQ{
&\int_t^\I\int_t^s r^{-2}
\norm{\Dy B_\mu(s)}_{\A^4_{\rho'}}\dd r\dd s
\\
&\quad\lsm(\zeta-\rho')^{-2}t^{-1}
\int_t^\I\norm{B_\mu(s)}_{\A^4_\zeta}\dd s<\I.
}
Fubini's theorem therefore gives
\EQn{\label{eq:eikonal-K-propagated-identity}
iK(t)
=i\int_t^\I U(t,s)B_\mu(s)\dd s
-\frac12\int_t^\I U(t,s)s^{-2}\Dy K(s)\dd s.
}
Moreover,
\EQ{
\int_t^\I s^{-2}\norm{\Dy K(s)}_{\A^4_{\rho'}}\dd s
\lsm(\rho''-\rho')^{-2}\int_t^\I s^{-1-\kappa-\al}\dd s<\I.
}
The decomposition \eqref{eq:Nmu-def} gives, with $Z=H+iK$,
\EQ{
N^{(\mu)}(s,Z)
&=P_\mu(s)H(s)+N_{\rm res}(s)
+N_{\rm fir}^{\rm sh}(s,Z(s))+N_{\rm qua}(s,Z(s))
-iB_\mu(s).
}
The assumptions give $Z=H+iK\in\mathcal F_T^{\min\{\sigma,\kappa+\al-1\}}(\rho'')$. The radius-gap estimate, \eqref{eq:Rmu-bound}, \eqref{eq:eikonal-profile-bounds}, and the direct product estimates give
\EQ{
&\norm{N_{\rm res}(s)}_{\A^4_{\rho'}}
+\norm{P_\mu(s)H(s)}_{\A^4_{\rho'}}
+\norm{N_{\rm fir}^{\rm sh}(s,Z(s))}_{\A^4_{\rho'}}
\\
&\quad+\norm{N_{\rm qua}(s,Z(s))}_{\A^4_{\rho'}}
+s^{-2}\norm{\Dy K(s)}_{\A^4_{\rho'}}
\\
&\lsm s^{-1-\min\{\kappa,\al+\sigma\}}.
}
All terms are integrable. Substitution of \eqref{eq:eikonal-K-propagated-identity} into the two Duhamel maps proves the stated map identity and hence \eqref{eq:eikonal-Duhamel-conjugacy}.
\end{proof}

For the construction of the coefficients in this map, write
\EQ{%\label{eq:eikonal-qua-bilinear}
N_{\rm qua}(t;Z_1,Z_2):=\frac{i}{2t^2}\nabla Z_1\cdot\nabla Z_2
}
and
\EQ{
K_{m,\ell}(H):=K_{m,\ell}(H,\ldots,H).
}
The series expansion \eqref{eq:eikonal-K-double-series} reads
\EQn{\label{eq:eikonal-K-expanded-for-map}
K_\mu(\re H)=\sum_{\ell\ge0}\sum_{m\ge1}K_{m,\ell}(H).
}
By \eqref{eq:eikonal-K-expanded-for-map}, \eqref{eq:eikonal-normal-map-parts}, and the symmetry of $N_{\rm qua}$,
\EQnnsub{
\widetilde{\TT}H-D_0
&=-\int_t^\I U(t,s)P_\mu(s)H(s)\dd s,
\label{eq:eikonal-map-expanded-P}
\\
&\quad-\int_t^\I U(t,s)N_{\rm fir}^{\rm sh}(s,H(s))\dd s,
%\label{eq:eikonal-map-expanded-firH}
\\
&\quad-\int_t^\I U(t,s)N_{\rm qua}(s;H(s),H(s))\dd s,
%\label{eq:eikonal-map-expanded-quaHH}
\\
&\quad-\sum_{\ell\ge0}\sum_{m\ge1}\int_t^\I U(t,s)N_{\rm fir}^{\rm sh}\brko{s,iK_{m,\ell}(H)(s)}\dd s,
%\label{eq:eikonal-map-expanded-firK}
\\
&\quad-2\sum_{\ell\ge0}\sum_{m\ge1}\int_t^\I U(t,s)N_{\rm qua}\brko{s;H(s),iK_{m,\ell}(H)(s)}\dd s,
%\label{eq:eikonal-map-expanded-quaHK}
\\
&\quad-\sum_{\ell_1,\ell_2\ge0}\sum_{m_1,m_2\ge1}\int_t^\I U(t,s)N_{\rm qua}\brko{s,iK_{m_1,\ell_1}(H)(s),iK_{m_2,\ell_2}(H)(s)}\dd s,
%\label{eq:eikonal-map-expanded-quaKK}
\\
&\quad+\frac12\sum_{\ell\ge0}\sum_{m\ge1}\int_t^\I U(t,s)s^{-2}\Dy K_{m,\ell}(H)(s)\dd s.
\label{eq:eikonal-map-expanded-DeltaK}
}
We now group \eqref{eq:eikonal-map-expanded-P}--\eqref{eq:eikonal-map-expanded-DeltaK} according to the number of applications of \eqref{eq:radius-gap-estimate}. Define the first layer by
\EQ{%\label{eq:eikonal-P1-definition}
\mathcal P_1(H)(t)
&:=-\int_t^\I U(t,s)P_\mu(s)H(s)\dd s
-\int_t^\I U(t,s)N_{\rm fir}^{\rm sh}(s,H(s))\dd s
\\
&\quad-\int_t^\I U(t,s)N_{\rm qua}(s;H(s),H(s))\dd s
\\
&\quad-\sum_{m\ge1}\int_t^\I U(t,s)N_{\rm fir}^{\rm sh}\brko{s,iK_{m,0}(H)(s)}\dd s
\\
&\quad-2\sum_{m\ge1}\int_t^\I U(t,s)N_{\rm qua}\brko{s;H(s),iK_{m,0}(H)(s)}\dd s
\\
&\quad-\sum_{m_1,m_2\ge1}\int_t^\I U(t,s)
N_{\rm qua}\brko{s,iK_{m_1,0}(H)(s),iK_{m_2,0}(H)(s)}\dd s.
}
For $q\ge2$, define
\EQ{%\label{eq:eikonal-Pq-definition}
\mathcal P_q(H)(t)
&:=-\sum_{m\ge1}\int_t^\I U(t,s)N_{\rm fir}^{\rm sh}\brko{s,iK_{m,q-1}(H)(s)}\dd s
\\
&\quad-2\sum_{m\ge1}\int_t^\I U(t,s)N_{\rm qua}\brko{s;H(s),iK_{m,q-1}(H)(s)}\dd s
\\
&\quad-\sum_{\substack{\ell_1,\ell_2\ge0\\ \ell_1+\ell_2+1=q}}\sum_{m_1,m_2\ge1}\int_t^\I U(t,s)
N_{\rm qua}\brko{s,iK_{m_1,\ell_1}(H)(s),iK_{m_2,\ell_2}(H)(s)}\dd s
\\
&\quad+\frac12\sum_{m\ge1}\int_t^\I U(t,s)s^{-2}\Dy K_{m,q-2}(H)(s)\dd s.
}
Thus
\EQn{\label{eq:eikonal-map-q-layers}
\widetilde{\TT}H=D_0+\sum_{q\ge1}\mathcal P_q(H).
}
We construct a fixed point as a formal Picard series
\EQ{
H=\sum_{n\ge0}H_n.
}
Since $K_{m,\ell}$ is real \(m\)-linear by \eqref{eq:eikonal-Kml-def}, the terms of degree $m$ in $\mathcal P_q(H)$ can be written directly, retaining the order of their arguments. For $m=q=1$, define
\EQn{\label{eq:eikonal-P11-explicit}
\mathcal P_{1,1}(G_1)(t)
&:=-\int_t^\I U(t,s)P_\mu(s)G_1(s)\dd s
-\int_t^\I U(t,s)N_{\rm fir}^{\rm sh}(s,G_1(s))\dd s
\\
&\quad-\int_t^\I U(t,s)N_{\rm fir}^{\rm sh}\brko{s,iK_{1,0}(G_1)(s)}\dd s.
}
For $m\ge2$, define
\EQn{\label{eq:eikonal-Pm1-explicit}
\mathcal P_{m,1}(G_1,\ldots,G_m)(t)
&:=-\cha_{\{m=2\}}\int_t^\I U(t,s)N_{\rm qua}\brko{s;G_1(s),G_2(s)}\dd s
\\
&\quad-\int_t^\I U(t,s)N_{\rm fir}^{\rm sh}\brko{s,iK_{m,0}(G_1,\ldots,G_m)(s)}\dd s
\\
&\quad-2\int_t^\I U(t,s)
\\
&\qquad N_{\rm qua}\brko{s;G_1(s),iK_{m-1,0}(G_2,\ldots,G_m)(s)}\dd s
\\
&\quad+\frac{i}{2}\sum_{\substack{m_1,m_2\ge1\\m_1+m_2=m}}\int_t^\I U(t,s)s^{-2}
\nabla K_{m_1,0}(G_1,\ldots,G_{m_1})(s)
\\
&\qquad\cdot\nabla K_{m_2,0}(G_{m_1+1},\ldots,G_m)(s)\dd s.
}
For $q\ge2$, define
\EQn{\label{eq:eikonal-P1q-explicit}
\mathcal P_{1,q}(G_1)(t)
&:=-\int_t^\I U(t,s)N_{\rm fir}^{\rm sh}\brko{s,iK_{1,q-1}(G_1)(s)}\dd s
\\
&\quad+\frac12\int_t^\I U(t,s)s^{-2}\Dy K_{1,q-2}(G_1)(s)\dd s.
}
For $m,q\ge2$, define
\EQn{\label{eq:eikonal-Pmq-explicit}
\mathcal P_{m,q}(G_1,\ldots,G_m)(t)
&:=-\int_t^\I U(t,s)N_{\rm fir}^{\rm sh}\brko{s,iK_{m,q-1}(G_1,\ldots,G_m)(s)}\dd s
\\
&\;+\int_t^\I U(t,s)s^{-2}\nabla G_1(s)
\cdot\nabla K_{m-1,q-1}(G_2,\ldots,G_m)(s)\dd s
\\
&\;+\frac{i}{2}\sum_{\substack{\ell_1,\ell_2\ge0\\\ell_1+\ell_2+1=q}}
\sum_{\substack{m_1,m_2\ge1\\m_1+m_2=m}}\int_t^\I U(t,s)s^{-2}
\\
&\quad\nabla K_{m_1,\ell_1}(G_1,\ldots,G_{m_1})(s)
\cdot\nabla K_{m_2,\ell_2}(G_{m_1+1},\ldots,G_m)(s)\dd s
\\
&\;+\frac12\int_t^\I U(t,s)s^{-2}\Dy K_{m,q-2}(G_1,\ldots,G_m)(s)\dd s.
}
By \eqref{eq:eikonal-P11-explicit}--\eqref{eq:eikonal-Pmq-explicit},
\EQn{\label{eq:eikonal-Pmq-decomposition}
\mathcal P_q(H)
=\sum_{m\ge1}\mathcal P_{m,q}(H,\ldots,H).
}
For $N\ge0$, the real multilinearity of $\mathcal P_{m,q}$ gives
\EQn{\label{eq:eikonal-Pmq-sum}
\mathcal P_{m,q}\brko{\sum_{n=0}^NH_n,\ldots,\sum_{n=0}^NH_n}
=\sum_{n_1=0}^N\cdots\sum_{n_m=0}^N
\mathcal P_{m,q}(H_{n_1},\ldots,H_{n_m}).
}
By \eqref{eq:eikonal-map-q-layers} and \eqref{eq:eikonal-Pmq-decomposition},
\EQn{\label{eq:eikonal-q-index-expansion}
\widetilde{\TT}H=D_0+\sum_{q\ge1}\sum_{m\ge1}\mathcal P_{m,q}(H,\ldots,H).
}
Whenever \eqref{eq:eikonal-K-double-series} and the double series in \eqref{eq:eikonal-q-index-expansion} converge absolutely at a smaller analytic radius, the series may be rearranged, and \eqref{eq:eikonal-q-index-expansion} is an identity in that space.

By \eqref{eq:eikonal-q-index-expansion} and $H=\widetilde{\TT}H$,
\EQn{\label{eq:eikonal-picard-formal-equation}
\sum_{n\ge0}H_n
=D_0+\sum_{q\ge1}\sum_{m\ge1}\mathcal P_{m,q}\brko{\sum_{n\ge0}H_n,\ldots,\sum_{n\ge0}H_n}.
}
By \eqref{eq:eikonal-Pmq-sum}, the right-hand side of \eqref{eq:eikonal-picard-formal-equation} is
\EQn{\label{eq:eikonal-picard-rhs-layer}
&D_0+\sum_{q\ge1}\sum_{m\ge1}\mathcal P_{m,q}\brko{\sum_{n\ge0}H_n,\ldots,\sum_{n\ge0}H_n}
\\
&=D_0+\sum_{q\ge1}\sum_{m\ge1}\sum_{n_1,\ldots,n_m\ge0}\mathcal P_{m,q}(H_{n_1},\ldots,H_{n_m})
\\
&=D_0+\sum_{n\ge1}\sum_{q=1}^{n}\sum_{m\ge1}
\sum_{\substack{n_1,\ldots,n_m\ge0\\ n_1+\cdots+n_m=n-q}}
\mathcal P_{m,q}(H_{n_1},\ldots,H_{n_m}).
}
Comparing \eqref{eq:eikonal-picard-formal-equation} and \eqref{eq:eikonal-picard-rhs-layer} gives
\EQ{
H_0=D_0,
}
and, for $n\ge1$,
\EQn{\label{eq:eikonal-picard-recursion}
H_n
=\sum_{q=1}^{n}\sum_{m\ge1}
\sum_{\substack{n_1,\ldots,n_m\ge0\\ n_1+\cdots+n_m=n-q}}
\mathcal P_{m,q}(H_{n_1},\ldots,H_{n_m}).
}
Thus $q$ is fixed before the inner Picard indices are chosen, and the corresponding term occurs at order $n=q+n_1+\cdots+n_m$.

By \eqref{eq:eikonal-terms-decomp} and \eqref{eq:Rmu-bound},
\EQ{
\norm{N_{\rm res}(s)}_{\A^4_{\rho_1}}
=\norm{R^{(\mu)}(s)}_{\A^4_{\rho_1}}
\le\norm{R^{(\mu)}(s)}_{\A^8_{\rho_1}}
\lsm_{\mu,R}s^{-1-\kappa}.
}
Hence \eqref{eq:eikonal-normal-map-parts} and \eqref{eq:U-isometry} give
\EQn{\label{eq:eikonal-D0-bound}
\norm{D_0}_{\mathcal F_T^\kappa(\rho_1)}
\le\sup_{t\ge T}t^\kappa\int_t^\I\norm{N_{\rm res}(s)}_{\A^4_{\rho_1}}\dd s
\lsm_{\mu,R}\sup_{t\ge T}t^\kappa\int_t^\I s^{-1-\kappa}\dd s
\lsm_{\mu,R}1.
}
Next, we derive the estimates for $\mathcal P_{m,q}$.  Throughout this part, we use \eqref{eq:eikonal-K-flexible-estimate} with its small prefactor bounded by one.

We first estimate the layer $q=1$. In the following, let $a_0=0$.

\begin{prop}[Estimate for $\mathcal P_{m,1}$]\label{prop:eikonal-Pm1-estimate}
Let $0\le\rho'<\rho<\rho_1$. Then
\EQ{%\label{eq:eikonal-Pm1-bound}
\norm{\mathcal P_{m,1}(G_1,\ldots,G_m)}_{\mathcal F_t^\kappa(\rho')} 
&\lsm_{\mu,R}C^ma_m(\rho-\rho')^{-1}
\int_t^\I s^{-1-\al} \prod_{j=1}^m
\norm{G_j}_{\mathcal F_s^\kappa(\rho)} \dd s.
}
\end{prop}

\begin{proof}
By \eqref{eq:eikonal-P-estimate}, \eqref{eq:eikonal-profile-bounds}, \eqref{eq:analytic-first-order-product}, and \eqref{eq:analytic-gradient-product},
\EQn{\label{eq:eikonal-direct-one-index-estimates}
\norm{P_\mu(s)G_1(s)}_{\A^4_{\rho'}}
&\lsm_{\mu,R}(\rho-\rho')^{-1}s^{-1-\al}\norm{G_1(s)}_{\A^4_\rho},
\\
\norm{N_{\rm fir}^{\rm sh}(s,G_1(s))}_{\A^4_{\rho'}}
&\lsm_{\mu,R}(\rho-\rho')^{-1}s^{-2}\norm{G_1(s)}_{\A^4_\rho},
\\
\norm{N_{\rm qua}(s;G_1(s),G_2(s))}_{\A^4_{\rho'}}
&\lsm(\rho-\rho')^{-1}s^{-2}
\norm{G_1(s)}_{\A^4_\rho}\norm{G_2(s)}_{\A^4_\rho}.
}
Using $t_0^\kappa\le s^\kappa$ for $s\ge t_0\ge1$, \eqref{eq:eikonal-direct-one-index-estimates} gives
\EQn{\label{eq:eikonal-direct-one-index-duhamel}
&t_0^\kappa\int_{t_0}^\I
\brko{\norm{P_\mu(s)G_1(s)}_{\A^4_{\rho'}}
+\norm{N_{\rm fir}^{\rm sh}(s,G_1(s))}_{\A^4_{\rho'}}}\dd s
\\
&\quad\lsm_{\mu,R}(\rho-\rho')^{-1}\int_{t_0}^\I s^{-1-\al}
\norm{G_1}_{\mathcal F_s^\kappa(\rho)}\dd s,
\\
&t_0^\kappa\int_{t_0}^\I
\norm{N_{\rm qua}(s;G_1(s),G_2(s))}_{\A^4_{\rho'}}\dd s
\\
&\quad\lsm(\rho-\rho')^{-1}\int_{t_0}^\I s^{-1-\al}
\norm{G_1}_{\mathcal F_s^\kappa(\rho)}
\norm{G_2}_{\mathcal F_s^\kappa(\rho)}\dd s.
}
For $\ell=0$, \eqref{eq:eikonal-K-flexible-estimate} and $\kappa+\al>1$ give
\EQ{%\label{eq:eikonal-K0-proof-bound}
\norm{K_{m,0}(G_1,\ldots,G_m)(s)}_{\A^4_\rho}
&\lsm_{\mu,R}a_m\int_s^\I \theta^{-\al-m\kappa}\dd \theta \prod_{j=1}^m
\norm{G_j}_{\mathcal F_s^\kappa(\rho)} 
\\
&\lsm_{\mu,R}a_m s^{1-\al-m\kappa}\prod_{j=1}^m
\norm{G_j}_{\mathcal F_s^\kappa(\rho)} .
}

With the above estimates in hand, we next deal with $\mathcal P_{m,1}$ in \eqref{eq:eikonal-P11-explicit} and \eqref{eq:eikonal-Pm1-explicit}. The term containing $N_{\rm fir}^{\rm sh}(iK_{m,0})$ in \eqref{eq:eikonal-P11-explicit} and \eqref{eq:eikonal-Pm1-explicit} satisfies
\EQn{\label{eq:eikonal-Pm1-firK-bound}
&(\rho-\rho')^{-1}t_0^\kappa\int_{t_0}^\I s^{-2}
\norm{K_{m,0}(G_1,\ldots,G_m)(s)}_{\A^4_\rho}\dd s
\\
&\quad\lsm_{\mu,R}a_m(\rho-\rho')^{-1}
\int_{t_0}^\I s^{-1-\al-(m-1)\kappa}\prod_{j=1}^m
\norm{G_j}_{\mathcal F_s^\kappa(\rho)}  \dd s
\\
&\quad\lsm_{\mu,R}a_m(\rho-\rho')^{-1}
\int_{t_0}^\I s^{-1-\al} \prod_{j=1}^m
\norm{G_j}_{\mathcal F_s^\kappa(\rho)} \dd s.
}
For $m\ge2$, the mixed term in \eqref{eq:eikonal-Pm1-explicit} satisfies
\EQ{%\label{eq:eikonal-Pm1-mixedK-bound}
&(\rho-\rho')^{-1}t_0^\kappa\int_{t_0}^\I s^{-2}
\norm{G_1(s)}_{\A^4_\rho}
\norm{K_{m-1,0}(G_2,\ldots,G_m)(s)}_{\A^4_\rho}\dd s
\\
&\quad\lsm_{\mu,R}a_{m-1}(\rho-\rho')^{-1}
\int_{t_0}^\I s^{-1-\al-(m-1)\kappa} \prod_{j=1}^m
\norm{G_j}_{\mathcal F_s^\kappa(\rho)} \dd s
\\
&\quad\lsm_{\mu,R}a_{m-1}(\rho-\rho')^{-1}
\int_{t_0}^\I s^{-1-\al} \prod_{j=1}^m
\norm{G_j}_{\mathcal F_s^\kappa(\rho)} \dd s.
}
If $m=m_1+m_2$ with $m_1,m_2\ge1$, the last term in \eqref{eq:eikonal-Pm1-explicit} satisfies
\EQn{\label{eq:eikonal-Pm1-doubleK-bound}
&(\rho-\rho')^{-1}t_0^\kappa\int_{t_0}^\I s^{-2}
\\
&\quad\norm{K_{m_1,0}(G_1,\ldots,G_{m_1})(s)}_{\A^4_\rho}
\norm{K_{m_2,0}(G_{m_1+1},\ldots,G_m)(s)}_{\A^4_\rho}\dd s
\\
&\quad\lsm_{\mu,R}a_{m_1}a_{m_2}(\rho-\rho')^{-1}
\int_{t_0}^\I s^{-2\al-(m-1)\kappa} \prod_{j=1}^m
\norm{G_j}_{\mathcal F_s^\kappa(\rho)} \dd s
\\
&\quad\lsm_{\mu,R}a_{m_1}a_{m_2}(\rho-\rho')^{-1}
\int_{t_0}^\I s^{-1-\al} \prod_{j=1}^m
\norm{G_j}_{\mathcal F_s^\kappa(\rho)} \dd s.
}
The last inequality uses $(m-1)\kappa+\al\ge\kappa+\al>1$. Notice that by \eqref{eq:eikonal-am-definition}, \(\frac{1}{m_1!(m-m_1)!}=\frac{1}{m!}\binom{m}{m_1}\), and \(\sum_{i=1}^{m-1}\binom{m}{i}=2^m-2\),
\EQn{\label{eq:eikonal-am-combinations}
	\sum_{\substack{m_1,m_2\ge1\\m_1+m_2=m}}a_{m_1}a_{m_2}
	&=(2^m-2)a_m,
	\qquad m\ge1.
}
By \eqref{eq:eikonal-P11-explicit}--\eqref{eq:eikonal-Pm1-explicit}, \eqref{eq:eikonal-am-combinations}, \eqref{eq:eikonal-direct-one-index-duhamel}, and \eqref{eq:eikonal-Pm1-firK-bound}--\eqref{eq:eikonal-Pm1-doubleK-bound},
\EQ{
\norm{\mathcal P_{m,1}(G_1,\ldots,G_m)}_{\mathcal F_t^\kappa(\rho')}
&\lsm_{\mu,R}C^ma_m(\rho-\rho')^{-1}
\int_t^\I s^{-1-\al}\prod_{j=1}^m
\norm{G_j}_{\mathcal F_s^\kappa(\rho)} \dd s.
}
\end{proof}

\begin{lem}[Common time-radius estimates]
\label{lem:eikonal-ordered-time-integrals}
For $x>0$ and $k\in\N$, write
\EQ{
(x)_0:=1,
\qquad
(x)_k:=x(x+1)\cdots(x+k-1),
\qquad k\ge1.
}
Assume that $\al,\kappa>0$ and $\kappa+\al>1$.  For $b\in\N_+$ and $n\in\N$, set
\EQ{
D_{b,n}:=b\kappa+\al-1+\frac{\al n}{2}.
}
For $\ell\in\N$, $t_0\ge1$, and $s\ge t_0$, define
\EQ{%\label{eq:eikonal-ordered-time-definition}
\J_{\ell}^{b,n}(s;t_0)
&:=
\int_{s\le s_1\le\cdots\le s_\ell\le\theta<\I}
s_1^{-1-\al}\cdots s_\ell^{-1-\al}
\theta^{-\al-b\kappa}
\brko{\frac\theta{t_0}}^{-\frac{\al n}{2}}
\dd\theta\dd s_\ell\cdots\dd s_1.
}
When $\ell=0$, the intermediate variables and their differentials are absent, and every product indexed by them is one.  Thus the ordered integral is the $\theta$-integral over $s\le\theta<\I$.

The following statements hold, with $C\ge1$ depending only on $\al$ and $\kappa$.
\begin{enumerate}[label=\textup{(\roman*)},leftmargin=*]
\item We have that
\EQn{\label{eq:eikonal-ordered-time-formula}
\J_{\ell}^{b,n}(s;t_0)
&=
\frac{\al^{-(\ell+1)}}
{\brko{\frac{D_{b,n}}\al}_{\ell+1}}
s^{1-\al-b\kappa-\ell\al}
\brko{\frac{s}{t_0}}^{-\frac{\al n}{2}}.
}
\item Let $m_0,N_0,\ell,N_1\in\N$ and $m_1\in\N_+$, and set
\EQ{
m=m_0+m_1,
\qquad
N=N_0+N_1.
}
Then
\EQn{\label{eq:eikonal-one-ordered-time-identity}
&t_0^\kappa\int_{t_0}^\I
s^{-2-m_0\kappa}
\brko{\frac{s}{t_0}}^{-\frac{\al N_0}{2}}
\J_{\ell}^{m_1,N_1}(s;t_0)\dd s
\\
&\quad=
\frac{\al^{-(\ell+1)}
t_0^{-(m-1)\kappa-(\ell+1)\al}}
{\brko{m\kappa+(\ell+1)\al+\frac{\al N}{2}}
\brko{\frac{D_{m_1,N_1}}\al}_{\ell+1}}.
}
Moreover,
\EQn{\label{eq:eikonal-one-ordered-time-bound}
t_0^\kappa\int_{t_0}^\I
s^{-2-m_0\kappa}
\brko{\frac{s}{t_0}}^{-\frac{\al N_0}{2}}
\J_{\ell}^{m_1,N_1}(s;t_0)\dd s \le
\frac{C^{\ell+1}t_0^{-(\ell+1)\al}}
{(N+\ell+2)(N_1+1)_{\ell+1}}.
}
\item Let $m_0,N_0\in\N$, let $m_1,m_2\in\N_+$, and let $\ell_1,\ell_2,N_1,N_2\in\N$.  Set
\EQ{
m=m_0+m_1+m_2,
\qquad
N=N_0+N_1+N_2,
\qquad
L=\ell_1+\ell_2.
}
Then
\EQn{\label{eq:eikonal-two-ordered-time-identity}
&t_0^\kappa\int_{t_0}^\I
s^{-2-m_0\kappa}
\brko{\frac{s}{t_0}}^{-\frac{\al N_0}{2}}
\J_{\ell_1}^{m_1,N_1}(s;t_0)
\J_{\ell_2}^{m_2,N_2}(s;t_0)\dd s
\\
&\quad=
\frac{\al^{-(L+2)}
t_0^{1-(m-1)\kappa-(L+2)\al}}
{\brko{m\kappa+(L+2)\al-1+\frac{\al N}{2}}
\brko{\frac{D_{m_1,N_1}}\al}_{\ell_1+1}
\brko{\frac{D_{m_2,N_2}}\al}_{\ell_2+1}}.
}
Moreover,
\EQn{\label{eq:eikonal-two-ordered-time-bound}
&t_0^\kappa\int_{t_0}^\I
s^{-2-m_0\kappa}
\brko{\frac{s}{t_0}}^{-\frac{\al N_0}{2}}
\J_{\ell_1}^{m_1,N_1}(s;t_0)
\J_{\ell_2}^{m_2,N_2}(s;t_0)\dd s
\\
&\quad\le
\frac{C^{L+2}t_0^{-(L+1)\al}}
{(N+L+2)(N_1+1)_{\ell_1+1}(N_2+1)_{\ell_2+1}}.
}
\item Let $g_0>0$, $N\in\N$, and $q\in\N_+$.  Set
\EQ{
\delta_h:=\frac{g_0}{2(N+q)},
\qquad
1\le h\le q,
\qquad
g_q:=g_0-\sum_{h=1}^q\delta_h.
}
Then $g_q\ge g_0/2$ and
\EQ{
\brko{\prod_{h=1}^q\delta_h^{-1}}
g_q^{-N}(N+1)_q^{-1}
\le C^qg_0^{-N-q}.
}
\item Let $g_0>0$, $B\in\{1,2\}$, and $N_0,N_1,\ldots,N_B,
\ell_1,\ldots,\ell_B\in\N$, and set
\EQ{
N:=N_0+\sum_{i=1}^BN_i,
\qquad
q:=1+\sum_{i=1}^B\ell_i.
}
Define
\EQ{
\delta_0:=\frac{g_0}{2(N+1)},
\qquad
g_1:=g_0-\delta_0.
}
For every $i$ such that $\ell_i\ge1$, set
\EQ{
\delta_{i,h}:=\frac{g_1}{2(N_i+\ell_i)},
\qquad
1\le h\le\ell_i,
\qquad
g_{i,\ell_i}:=g_1-\sum_{h=1}^{\ell_i}\delta_{i,h}.
}
When $\ell_i=0$, set $g_{i,0}:=g_1$ and interpret the corresponding
product as one.  Then
\EQ{
\frac{\delta_0^{-1}}{N+1}g_1^{-N_0}
\prod_{i=1}^B\mbrkbb{
\brko{\prod_{h=1}^{\ell_i}\delta_{i,h}^{-1}}
g_{i,\ell_i}^{-N_i}
(N_i+1)_{\ell_i}^{-1}} \le C^qg_0^{-N-q}.
}
\end{enumerate}
\end{lem}

\begin{proof}
For $\ell=0$, direct integration gives
\EQ{
\J_0^{b,n}(s;t_0)
=
D_{b,n}^{-1}
s^{1-\al-b\kappa}
\brko{\frac{s}{t_0}}^{-\frac{\al n}{2}}.
}
For $\ell\ge0$, the definition also gives the recursion
\EQ{
\J_{\ell+1}^{b,n}(s;t_0)
=
\int_s^\I
\tau^{-1-\al}\J_\ell^{b,n}(\tau;t_0)\dd\tau.
}
If \eqref{eq:eikonal-ordered-time-formula} holds at level $\ell$, this integration contributes the factor
\EQ{
\brko{D_{b,n}+(\ell+1)\al}^{-1}
=
\frac{\al^{-1}}{\frac{D_{b,n}}\al+\ell+1}.
}
The recursion defining the rising factorial therefore proves \eqref{eq:eikonal-ordered-time-formula} by induction.

For \textup{(ii)}, use \eqref{eq:eikonal-ordered-time-formula}. The remaining time integral is
\EQ{
t_0^\kappa\int_{t_0}^\I
s^{-1-m\kappa-(\ell+1)\al}
\brko{\frac{s}{t_0}}^{-\frac{\al N}{2}}\dd s.
}
Its direct evaluation proves \eqref{eq:eikonal-one-ordered-time-identity}.

For \textup{(iii)}, apply \eqref{eq:eikonal-ordered-time-formula} to both ordered integrals.  The remaining time integral is
\EQ{
t_0^\kappa\int_{t_0}^\I
s^{-m\kappa-(L+2)\al}
\brko{\frac{s}{t_0}}^{-\frac{\al N}{2}}\dd s.
}
Its direct evaluation proves \eqref{eq:eikonal-two-ordered-time-identity}.

It remains to prove the two uniform bounds.  For $b\in\N_+$, $n,a\in\N$,
\EQ{
\frac{D_{b,n}}\al+a
=
\frac1\al\mbrko{
\brko{\kappa+\al-1}+(b-1)\kappa
+\frac{\al n}{2}+a\al}
\ge c_{\al,\kappa}(n+a+1).
}
Consequently,
\EQ{
\brko{\frac{D_{b,n}}\al}_{\ell+1}
\ge c_{\al,\kappa}^{\ell+1}(n+1)_{\ell+1}.
}
For \textup{(ii)}, the outer denominator satisfies
\EQ{
&m\kappa+(\ell+1)\al+\frac{\al N}{2}
\\
&\quad=
1+\brko{\kappa+\al-1}+(m-1)\kappa
+\ell\al+\frac{\al N}{2}
\ge c_{\al,\kappa}(N+\ell+2).
}
For \textup{(iii)},
\EQ{
&m\kappa+(L+2)\al-1+\frac{\al N}{2}
\\
&\quad=
1+2\brko{\kappa+\al-1}+(m-2)\kappa
+L\al+\frac{\al N}{2}
\ge c_{\al,\kappa}(N+L+2).
}
Finally,
\EQ{
-(m-1)\kappa-(\ell+1)\al
\le-(\ell+1)\al,
}
and
\EQ{
&1-(m-1)\kappa-(L+2)\al
\\
&\quad=
-(L+1)\al-\brko{\kappa+\al-1}
-(m-2)\kappa
\le-(L+1)\al.
}
Together with $t_0\ge1$, these estimates prove \eqref{eq:eikonal-one-ordered-time-bound} and \eqref{eq:eikonal-two-ordered-time-bound}.

For \textup{(iv)},
\EQ{
g_q=g_0\brko{1-\frac{q}{2(N+q)}}\ge\frac{g_0}{2}.
}
Moreover,
\EQ{
&\brko{\prod_{h=1}^q\delta_h^{-1}}
g_q^{-N}(N+1)_q^{-1}
\\
&\quad=
2^qg_0^{-N-q}
\brko{1-\frac{q}{2(N+q)}}^{-N}
\frac{(N+q)^q}{(N+1)_q}.
}
Since
\EQ{
-N\log\brko{1-\frac{q}{2(N+q)}}\le q,
\qquad
\frac{N+q}{N+a}\le\frac qa,
\quad 1\le a\le q,
}
we obtain
\EQ{
\brko{1-\frac{q}{2(N+q)}}^{-N}\le e^q,
\qquad
\frac{(N+q)^q}{(N+1)_q}
\le\frac{q^q}{q!}\le e^q.
}
This proves \textup{(iv)}.

For \textup{(v)}, first note that $g_1\ge g_0/2$ and
\EQ{
\brko{\frac{g_0}{g_1}}^N
=\brko{1-\frac1{2(N+1)}}^{-N}\le e.
}
If $\ell_i\ge1$, then $g_{i,\ell_i}\ge g_1/2$ and
\EQ{
&\brko{\prod_{h=1}^{\ell_i}\delta_{i,h}^{-1}}
\brko{\frac{g_1}{g_{i,\ell_i}}}^{N_i}
(N_i+1)_{\ell_i}^{-1}
\\
&\quad=
2^{\ell_i}g_1^{-\ell_i}
\brko{1-\frac{\ell_i}{2(N_i+\ell_i)}}^{-N_i}
\frac{(N_i+\ell_i)^{\ell_i}}
{(N_i+1)_{\ell_i}}
\\
&\quad\le C^{\ell_i}g_1^{-\ell_i}.
}
The same estimate holds when $\ell_i=0$, with both sides equal to one.
Therefore,
\EQ{
&\frac{\delta_0^{-1}}{N+1}g_1^{-N_0}
\prod_{i=1}^B\mbrkbb{
\brko{\prod_{h=1}^{\ell_i}\delta_{i,h}^{-1}}
g_{i,\ell_i}^{-N_i}
(N_i+1)_{\ell_i}^{-1}}
\\
&\quad=
\frac2{g_0}g_1^{-N}
\prod_{i=1}^B\mbrkbb{
\brko{\prod_{h=1}^{\ell_i}\delta_{i,h}^{-1}}
\brko{\frac{g_1}{g_{i,\ell_i}}}^{N_i}
(N_i+1)_{\ell_i}^{-1}}
\\
&\quad\le
C^{q}g_0^{-1}g_0^{-N}g_1^{-(q-1)}
\le C^qg_0^{-N-q}.
}
This proves \textup{(v)}.
\end{proof}

\begin{prop}[Estimate for $\mathcal P_{m,q}$, $q\ge2$]\label{prop:eikonal-higher-layer-estimate}
Let $q\ge2$, $0\le\rho<\rho_+<\rho_1$, and assume that
\EQn{\label{eq:eikonal-higher-layer-input}
\norm{G_j}_{\mathcal F_\theta^\kappa(\zeta)}
\le A_j\brko{\frac{\theta^{-\frac\al2}}
{\rho_+-\zeta}}^{n_j},
\qquad \theta\ge t,\quad \rho\le\zeta<\rho_+,
}
where $n_j\in\N$.  Set $N=\sum_{j=1}^m n_j$.  Then there exists $C\ge1$, depending only on the fixed parameters, such that
\EQn{\label{eq:eikonal-higher-layer-bound}
\norm{\mathcal P_{m,q}(G_1,\ldots,G_m)}_{\mathcal F_t^\kappa(\rho)}
&\lsm_{\mu,R} C^{m+q}a_m
(\rho_+-\rho)^{-N-q}
t^{-(q-1)\al-\frac{\al N}{2}}
\prod_{j=1}^m A_j.
}
\end{prop}

\begin{proof}
We begin with the four terms which have to be estimated.  For $q\ge2$, define $V_1,V_3,V_4$ for $m\ge1$ and $V_2$ for $m\ge2$ by
\EQ{
%\label{eq:eikonal-higher-V1}
V_1(t_0)&:=-\int_{t_0}^\I U(t_0,s)N_{\rm fir}^{\rm sh}\brko{s,iK_{m,q-1}(G_1,\ldots,G_m)(s)}\dd s,
\\
%\label{eq:eikonal-higher-V2}
V_2(t_0)&:=-2\int_{t_0}^\I U(t_0,s) N_{\rm qua}\brko{s;G_1(s),iK_{m-1,q-1}(G_2,\ldots,G_m)(s)}\dd s,
\\
%\label{eq:eikonal-higher-V3}
V_3(t_0)&:=\frac i2\sum_{\substack{\ell_1,\ell_2\ge0\\\ell_1+\ell_2+1=q}}
\sum_{\substack{m_1,m_2\ge1\\m_1+m_2=m}}\int_{t_0}^\I U(t_0,s)s^{-2}
\nabla K_{m_1,\ell_1}(G_1,\ldots,G_{m_1})(s)
\\
&\quad\cdot\nabla K_{m_2,\ell_2}(G_{m_1+1},\ldots,G_m)(s)\dd s,
\\
%\label{eq:eikonal-higher-V4}
V_4(t_0)&:=\frac12\int_{t_0}^\I U(t_0,s)s^{-2}\Dy K_{m,q-2}(G_1,\ldots,G_m)(s)\dd s.
}
When $m=1$, let $V_2=0$.  By \eqref{eq:eikonal-P1q-explicit} and \eqref{eq:eikonal-Pmq-explicit},
\EQn{\label{eq:eikonal-higher-V-decomposition}
\mathcal P_{m,q}(G_1,\ldots,G_m)(t_0)
=V_1(t_0)+V_2(t_0)+V_3(t_0)+V_4(t_0).
}

Fix $t_0\ge t$ and set, once for all four terms,
\EQ{
A:=\prod_{j=1}^m A_j,
\qquad
g_0:=\rho_+-\rho.
}
Taking the lower endpoint in \eqref{eq:eikonal-higher-layer-input} to be $\theta$ yields, for $\theta\ge t_0$ and $\rho\le\zeta<\rho_+$,
\EQn{\label{eq:eikonal-input-pointwise}
\norm{G_j(\theta)}_{\A^4_{\zeta}}
&\le A_j(\rho_+-\zeta)^{-n_j}
\theta^{-\kappa-\frac{\al n_j}{2}}
\\
&=A_jt_0^{-\frac{\al n_j}{2}}
(\rho_+-\zeta)^{-n_j}\theta^{-\kappa}
\brko{\frac{\theta}{t_0}}^{-\frac{\al n_j}{2}}.
}
In each estimate below, denote the first radius and its two associated gaps by
\EQn{\label{eq:eikonal-radius-sequence}
\rho=\zeta_0<\zeta_1<\rho_+,
\qquad
\delta_1:=\zeta_1-\rho,
\qquad
g_1:=\rho_+-\zeta_1.
}
For $V_1,V_2,V_4$, extend this sequence by
\EQ{
\zeta_1<\zeta_2<\cdots<\zeta_q<\rho_+,
\qquad
\delta_h:=\zeta_h-\zeta_{h-1},
\qquad
g_h:=\rho_+-\zeta_h,
\quad 2\le h\le q.
}
The gaps are chosen separately in the four estimates.

\emph{Estimate of $V_1$.} The isometry property of $U(t_0,s)$, the definition of $N_{\rm fir}^{\rm sh}$, \eqref{eq:eikonal-profile-bounds}, and \eqref{eq:analytic-first-order-product} give
\EQn{\label{eq:eikonal-V1-outer-estimate}
\norm{V_1(t_0)}_{\A^4_{\rho}}
\lsm_{\mu,R}\delta_1^{-1}
\int_{t_0}^\I s^{-2}
\norm{K_{m,q-1}(G_1,\ldots,G_m)(s)}_{\A^4_{\zeta_1}}\dd s .
}
Applying \eqref{eq:eikonal-K-flexible-estimate} along $\zeta_1<\cdots<\zeta_q$ gives
\EQn{\label{eq:eikonal-V1-K-ordered-estimate}
&\norm{K_{m,q-1}(G_1,\ldots,G_m)(s)}_{\A^4_{\zeta_1}}
\\
&\lsm_{\mu,R}
C^qa_m
\prod_{h=2}^{q}\delta_h^{-1}
\\
&\quad\cdot
\int_{s\le s_1\le\cdots\le s_{q-1}\le \theta<\I}
\brko{\prod_{h=1}^{q-1}s_h^{-1-\al}}\theta^{-\al}
\prod_{j=1}^m\norm{G_j(\theta)}_{\A^4_{\zeta_q}}
\dd \theta\dd s_{q-1}\cdots\dd s_1 .
}
At the input radius $\zeta_q$, \eqref{eq:eikonal-input-pointwise} gives
\EQn{\label{eq:eikonal-V1-terminal-input}
\prod_{j=1}^m\norm{G_j(\theta)}_{\A^4_{\zeta_q}}
\quad\le
A t_0^{-\frac{\al N}{2}}g_q^{-N}\theta^{-m\kappa}
\brko{\frac{\theta}{t_0}}^{-\frac{\al N}{2}}.
}
Combining \eqref{eq:eikonal-V1-K-ordered-estimate} and \eqref{eq:eikonal-V1-terminal-input} with \eqref{eq:eikonal-V1-outer-estimate} gives
\EQ{%\label{eq:eikonal-V1-before-time}
&t_0^\kappa\norm{V_1(t_0)}_{\A^4_{\rho}}
\\
&\quad\lsm_{\mu,R}
C^qa_m A t_0^{-\frac{\al N}{2}}g_q^{-N}
\prod_{h=1}^{q}\delta_h^{-1}
\\
&\qquad\cdot
t_0^\kappa
\int_{t_0\le s\le s_1\le\cdots\le s_{q-1}\le \theta<\I}
s^{-2}\brko{\prod_{h=1}^{q-1}s_h^{-1-\al}}
\theta^{-\al-m\kappa}
\brko{\frac \theta{t_0}}^{-\frac{\al N}{2}}
\dd \theta\dd s_{q-1}\cdots\dd s_1\dd s .
}
We apply Lemma \ref{lem:eikonal-ordered-time-integrals}\textup{(ii)} with
\EQ{
\ell=q-1,
\qquad
(m_0,N_0)=(0,0),
\qquad
(m_1,N_1)=(m,N).
}
The resulting time estimate is
\EQ{%\label{eq:eikonal-V1-time}
&t_0^\kappa
\int_{t_0\le s\le s_1\le\cdots\le s_{q-1}\le \theta<\I}
s^{-2}\brko{\prod_{h=1}^{q-1}s_h^{-1-\al}}
\theta^{-\al-m\kappa}
\brko{\frac \theta{t_0}}^{-\frac{\al N}{2}}
\dd \theta\dd s_{q-1}\cdots\dd s_1\dd s
\\
&\quad=
t_0^\kappa\int_{t_0}^\I
s^{-2}\J_{q-1}^{m,N}(s;t_0)\dd s
\\
&\quad\le
\frac{C^qt_0^{-q\al}}{(N+1)_q}.
}
After the time integration,
\EQ{%\label{eq:eikonal-V1-after-time}
t_0^\kappa\norm{V_1(t_0)}_{\A^4_{\rho}}\lsm_{\mu,R}
C^qa_m A
t_0^{-q\al-\frac{\al N}{2}}g_q^{-N}
\prod_{h=1}^{q}\delta_h^{-1}
(N+1)_q^{-1}.
}

Choose the radii for $V_1$ by taking
\EQn{\label{eq:eikonal-V1-radius-choice}
\delta_h:=\frac{g_0}{2(N+q)},
\qquad
1\le h\le q.
}
Then
\EQ{
g_q=g_0-\sum_{h=1}^q\delta_h
=g_0\brko{1-\frac{q}{2(N+q)}}\ge\frac{g_0}{2}.
}
Moreover,
\EQn{\label{eq:eikonal-V1-radius-bound}
\brko{\prod_{h=1}^{q}\delta_h^{-1}}
g_q^{-N}
(N+1)_q^{-1} \le
C^qg_0^{-N-q}\frac{(N+q)^q}{(N+1)_q}
\le C^qg_0^{-N-q}\frac{q^q}{q!}
\le C^qg_0^{-N-q}.
}
Here $-N\log(1-q/(2(N+q)))\le q$, while $(N+q)/(N+a)\le q/a$ for $1\le a\le q$. Thus
\EQ{%\label{eq:eikonal-V1-pointwise-bound}
t_0^\kappa\norm{V_1(t_0)}_{\A^4_{\rho}}
\lsm_{\mu,R}
C^qa_m A g_0^{-N-q}
t_0^{-q\al-\frac{\al N}{2}}.
}
Since $t_0\ge t$, taking the supremum over $t_0$ gives
\EQ{%\label{eq:eikonal-V1-bound}
\norm{V_1}_{\mathcal F_t^\kappa(\rho)}
\lsm_{\mu,R}
C^qa_m A
(\rho_+-\rho)^{-N-q}
t^{-q\al-\frac{\al N}{2}}.
}

\emph{Estimate of $V_2$.} Assume $m\ge2$ and use the radii in \eqref{eq:eikonal-radius-sequence}. The isometry property of $U$ and \eqref{eq:analytic-gradient-product} give
\EQ{%\label{eq:eikonal-V2-outer-estimate}
&\norm{V_2(t_0)}_{\A^4_{\rho}}
\\
&\quad\lsm
\delta_1^{-1}\int_{t_0}^\I s^{-2}
\norm{G_1(s)}_{\A^4_{\zeta_1}}
\norm{K_{m-1,q-1}(G_2,\ldots,G_m)(s)}_{\A^4_{\zeta_1}}
\dd s .
}
The quadratic term uses the gap from $\rho$ to $\zeta_1$, and the $K$ estimate uses $\zeta_1<\cdots<\zeta_q$.  Hence \eqref{eq:eikonal-K-flexible-estimate} and \eqref{eq:eikonal-input-pointwise} yield
\EQ{%\label{eq:eikonal-V2-before-time}
t_0^\kappa\norm{V_2(t_0)}_{\A^4_{\rho}} 
&\lsm_{\mu,R}
C^qa_{m-1}A t_0^{-\frac{\al N}{2}}
\prod_{h=1}^{q}\delta_h^{-1}
g_1^{-n_1}g_q^{-(N-n_1)}
\\
&\qquad\cdot
t_0^\kappa
\int_{t_0\le s\le s_1\le\cdots\le s_{q-1}\le \theta<\I}
s^{-2-\kappa}
\brko{\frac{s}{t_0}}^{-\frac{\al n_1}{2}}
\\
&\qquad\qquad\cdot
\brko{\prod_{h=1}^{q-1}s_h^{-1-\al}}
\theta^{-\al-(m-1)\kappa}
\brko{\frac \theta{t_0}}^{-\frac{\al(N-n_1)}{2}} 
\dd \theta\dd s_{q-1}\cdots\dd s_1\dd s .
}
We apply Lemma \ref{lem:eikonal-ordered-time-integrals}\textup{(ii)} with
\EQ{
\ell=q-1,
\qquad
(m_0,N_0)=(1,n_1),
\qquad
(m_1,N_1)=(m-1,N-n_1).
}
The resulting time estimate is
\EQ{%\label{eq:eikonal-V2-time}
&t_0^\kappa
\int_{t_0\le s\le s_1\le\cdots\le s_{q-1}\le \theta<\I}
s^{-2-\kappa}
\brko{\frac{s}{t_0}}^{-\frac{\al n_1}{2}}
\\
&\qquad\cdot
\brko{\prod_{h=1}^{q-1}s_h^{-1-\al}}
\theta^{-\al-(m-1)\kappa}
\brko{\frac \theta{t_0}}^{-\frac{\al(N-n_1)}{2}} 
\dd \theta\dd s_{q-1}\cdots\dd s_1\dd s
\\
&\quad=
t_0^\kappa\int_{t_0}^\I
s^{-2-\kappa}\brko{\frac{s}{t_0}}^{-\frac{\al n_1}{2}}
\J_{q-1}^{m-1,N-n_1}(s;t_0)\dd s
\\
&\quad\le
\frac{C^qt_0^{-q\al}}
{(N+1)(N-n_1+1)_{q-1}}.
}
After the time integration,
\EQ{%\label{eq:eikonal-V2-after-time}
t_0^\kappa\norm{V_2(t_0)}_{\A^4_{\rho}} \lsm_{\mu,R}
C^qa_{m-1}A
t_0^{-q\al-\frac{\al N}{2}}
\frac{\delta_1^{-1}}{N+1}
g_1^{-n_1}
\brko{\prod_{h=2}^{q}\delta_h^{-1}}
g_q^{-(N-n_1)}
(N-n_1+1)_{q-1}^{-1}.
}
Separate the two radius scales as
\EQ{%\label{eq:eikonal-V2-radius-split}
g_1^{-n_1}g_q^{-(N-n_1)}
=
g_1^{-N}
\brko{\frac{g_1}{g_q}}^{N-n_1}.
}
Choose the outer gap by
\EQ{
\delta_1:=\frac{g_0}{2(N+1)}.
}
Then $g_1=g_0-\delta_1\ge g_0/2$ and
\EQ{
\frac{\delta_1^{-1}}{N+1}
g_1^{-N}
\lsm g_0^{-N-1}.
}
For the remaining gaps, take
\EQ{
\delta_h:=\frac{g_1}{2(N-n_1+q-1)},
\qquad 2\le h\le q.
}
Then
\EQ{
g_q=g_1-\sum_{h=2}^q\delta_h\ge\frac{g_1}{2},
}
and the same elementary estimate used in \eqref{eq:eikonal-V1-radius-bound} gives
\EQ{
\brko{\prod_{h=2}^{q}\delta_h^{-1}}
\brko{\frac{g_1}{g_q}}^{N-n_1}
(N-n_1+1)_{q-1}^{-1}
\le C^{q-1}g_1^{-(q-1)}.
}
Combining the last two bounds, using $g_1\ge g_0/2$, and taking the supremum over $t_0\ge t$, we obtain
\EQ{%\label{eq:eikonal-V2-bound}
\norm{V_2}_{\mathcal F_t^\kappa(\rho)}
\lsm_{\mu,R}
C^qa_{m-1}A
(\rho_+-\rho)^{-N-q}
t^{-q\al-\frac{\al N}{2}}.
}

For $m\ge2$,
\EQ{
\frac{a_{m-1}}{a_m}=\frac{m}{C_4p}\le C^m,
}
so the coefficient $a_{m-1}$ in the $V_2$ estimate is absorbed into
$C^{m+q}a_m$.

\emph{Estimate of $V_3$.} When $m=1$, the sum defining $V_3$ is empty and $V_3=0$.  Assume $m\ge2$ and fix $m_1,m_2\ge1$ and $\ell_1,\ell_2\ge0$ with
\EQ{
m_1+m_2=m,
\qquad
\ell_1+\ell_2+1=q.
}
Set
\EQ{
N_1:=\sum_{j=1}^{m_1}n_j,
\qquad
N_2:=\sum_{j=m_1+1}^{m}n_j.
}
Thus $N_1+N_2=N$.  Use the first radius from \eqref{eq:eikonal-radius-sequence} as the common outer radius. For each $i$ with $\ell_i\ge1$, choose an increasing radius sequence
\EQ{
&\zeta_{i,0}=\zeta_1<\zeta_{i,1}<\cdots
<\zeta_{i,\ell_i}<\rho_+,
\\
&
\delta_{i,h}:=\zeta_{i,h}-\zeta_{i,h-1},
\qquad
g_{i,h}:=\rho_+-\zeta_{i,h},
\quad 1\le h\le\ell_i.
}
When $\ell_i=0$, set $\zeta_{i,0}:=\zeta_1$ and $g_{i,0}:=g_1$. No internal radius or gap is introduced, and the corresponding products below are empty. Denote the corresponding summand in the definition of $V_3$ by $V_3^{m_1,m_2,\ell_1,\ell_2}$. The isometry property of $U$ and \eqref{eq:analytic-gradient-product} give
\EQ{%\label{eq:eikonal-V3-outer-estimate}
t_0^\kappa
\norm{V_3^{m_1,m_2,\ell_1,\ell_2}(t_0)}_{\A^4_{\rho}} 
&\lsm
\delta_1^{-1}t_0^\kappa\int_{t_0}^\I s^{-2}
\norm{K_{m_1,\ell_1}(G_1,\ldots,G_{m_1})(s)}_{\A^4_{\zeta_1}}
\\
&\quad\cdot
\norm{K_{m_2,\ell_2}(G_{m_1+1},\ldots,G_m)(s)}_{\A^4_{\zeta_1}}
\dd s .
}
The two gradients use the common outer gap.  After that, the two $K$ estimates use independent radius and time variables.  Hence \eqref{eq:eikonal-K-flexible-estimate} and \eqref{eq:eikonal-input-pointwise} yield
\EQ{%\label{eq:eikonal-V3-before-time}
&t_0^\kappa
\norm{V_3^{m_1,m_2,\ell_1,\ell_2}(t_0)}_{\A^4_{\rho}}
\\
&\quad\lsm_{\mu,R}
C^qa_{m_1}a_{m_2}A t_0^{-\frac{\al N}{2}}
\delta_1^{-1}
\prod_{i=1}^2\mbrkbb{
\brko{\prod_{h=1}^{\ell_i}\delta_{i,h}^{-1}}
g_{i,\ell_i}^{-N_i}} 
t_0^\kappa\int_{t_0}^\I s^{-2}
\\
&\qquad\cdot
\prod_{i=1}^2\mbrkbb{
\int_{s\le s_{i,1}\le\cdots\le s_{i,\ell_i}\le \theta_i<\I}
\brko{\prod_{h=1}^{\ell_i}s_{i,h}^{-1-\al}}
\theta_i^{-\al-m_i\kappa}
\brko{\frac{\theta_i}{t_0}}^{-\frac{\al N_i}{2}}
\dd \theta_i\dd s_{i,\ell_i}\cdots\dd s_{i,1}}
\dd s .
}
We apply Lemma \ref{lem:eikonal-ordered-time-integrals}\textup{(iii)} with
\EQ{
(m_0,N_0)=(0,0),
\qquad
L=\ell_1+\ell_2=q-1.
}
The resulting time estimate is
\EQ{%\label{eq:eikonal-V3-time}
&t_0^\kappa\int_{t_0}^\I s^{-2} 
\prod_{i=1}^2\mbrkbb{
\int_{s\le s_{i,1}\le\cdots\le s_{i,\ell_i}\le \theta_i<\I}
\brko{\prod_{h=1}^{\ell_i}s_{i,h}^{-1-\al}}
\theta_i^{-\al-m_i\kappa}
\brko{\frac{\theta_i}{t_0}}^{-\frac{\al N_i}{2}}
\dd \theta_i\dd s_{i,\ell_i}\cdots\dd s_{i,1}}
\dd s
\\
&\quad=
t_0^\kappa\int_{t_0}^\I s^{-2}
\prod_{i=1}^2\J_{\ell_i}^{m_i,N_i}(s;t_0)\dd s
\\
&\quad\le
\frac{C^qt_0^{-q\al}}
{(N+1)\displaystyle\prod_{i=1}^2(N_i+1)_{\ell_i}}.
}
After the time integration,
\EQn{\label{eq:eikonal-V3-after-time}
&t_0^\kappa
\norm{V_3^{m_1,m_2,\ell_1,\ell_2}(t_0)}_{\A^4_{\rho}}
\\
&\quad\lsm_{\mu,R}
C^qa_{m_1}a_{m_2}A
t_0^{-q\al-\frac{\al N}{2}}
\frac{\delta_1^{-1}}{N+1}
\prod_{i=1}^2\mbrkbb{
\brko{\prod_{h=1}^{\ell_i}\delta_{i,h}^{-1}}
g_{i,\ell_i}^{-N_i}
(N_i+1)_{\ell_i}^{-1}}.
}
The radius factors split as
\EQ{
\prod_{i=1}^2
g_{i,\ell_i}^{-N_i}
=g_1^{-N}
\prod_{i=1}^2
\brko{\frac{g_1}{g_{i,\ell_i}}}^{N_i}.
}
Choose
\EQ{
\delta_1:=\frac{g_0}{2(N+1)}.
}
Then $g_1=g_0-\delta_1\ge g_0/2$ and
\EQ{
\frac{\delta_1^{-1}}{N+1}
g_1^{-N}
\lsm g_0^{-N-1}.
}
For each $i$ with $\ell_i\ge1$, choose
\EQ{
\delta_{i,h}:=\frac{g_1}{2(N_i+\ell_i)},
\qquad
1\le h\le\ell_i.
}
Then
\EQ{
g_{i,\ell_i}
=g_1-\sum_{h=1}^{\ell_i}\delta_{i,h}
\ge\frac{g_1}{2}.
}
The corresponding radius factor satisfies
\EQ{
\brko{\prod_{h=1}^{\ell_i}\delta_{i,h}^{-1}}
\brko{\frac{g_1}{g_{i,\ell_i}}}^{N_i}
(N_i+1)_{\ell_i}^{-1} \le C^{\ell_i}g_1^{-\ell_i},
\qquad i=1,2.
}
Since $1+\ell_1+\ell_2=q$, the complete radius factor satisfies
\EQn{\label{eq:eikonal-V3-radius-bound}
\frac{\delta_1^{-1}}{N+1}
g_1^{-N}
\prod_{i=1}^2\mbrkbb{
\brko{\prod_{h=1}^{\ell_i}\delta_{i,h}^{-1}}
\brko{\frac{g_1}{g_{i,\ell_i}}}^{N_i}
(N_i+1)_{\ell_i}^{-1}}
\le C^qg_0^{-N-q}.
}
Combining \eqref{eq:eikonal-V3-after-time} and \eqref{eq:eikonal-V3-radius-bound}, taking the supremum over $t_0\ge t$, and summing over the $q$ choices of $(\ell_1,\ell_2)$ gives
\EQ{%\label{eq:eikonal-V3-bound}
\norm{V_3}_{\mathcal F_t^\kappa(\rho)}
\lsm_{\mu,R} C^qq
\sum_{\substack{m_1,m_2\ge1\\m_1+m_2=m}}a_{m_1}a_{m_2} A(\rho_+-\rho)^{-N-q}
t^{-q\al-\frac{\al N}{2}}.
}

\emph{Estimate of $V_4$.} Return to the radius notation in \eqref{eq:eikonal-radius-sequence}.  The two radius losses in \eqref{eq:analytic-second-derivative} are used first:
\EQ{%\label{eq:eikonal-V4-Laplacian-estimate}
\norm{\Dy K_{m,q-2}(G_1,\ldots,G_m)(s)}_{\A^4_{\rho}}
\lsm
\delta_1^{-1}\delta_2^{-1}
\norm{K_{m,q-2}(G_1,\ldots,G_m)(s)}_{\A^4_{\zeta_2}}.
}
The remaining $q-2$ gaps are allocated to the $K$ estimate.  Hence \eqref{eq:eikonal-K-flexible-estimate} and \eqref{eq:eikonal-input-pointwise} yield
\EQ{%\label{eq:eikonal-V4-before-time}
&t_0^\kappa\norm{V_4(t_0)}_{\A^4_{\rho}}
\\
&\quad\lsm_{\mu,R}
C^qa_mA t_0^{-\frac{\al N}{2}}g_q^{-N}
\prod_{h=1}^{q}\delta_h^{-1}
\\
&\qquad\cdot
t_0^\kappa
\int_{t_0\le s\le s_1\le\cdots\le s_{q-2}\le \theta<\I}
s^{-2}\brko{\prod_{h=1}^{q-2}s_h^{-1-\al}} 
\theta^{-\al-m\kappa}
\brko{\frac \theta{t_0}}^{-\frac{\al N}{2}}
\dd \theta\dd s_{q-2}\cdots\dd s_1\dd s .
}
We apply Lemma \ref{lem:eikonal-ordered-time-integrals}\textup{(ii)} with
\EQ{
\ell=q-2,
\qquad
(m_0,N_0)=(0,0),
\qquad
(m_1,N_1)=(m,N).
}
The resulting time estimate is
\EQ{%\label{eq:eikonal-V4-time}
&t_0^\kappa
\int_{t_0\le s\le s_1\le\cdots\le s_{q-2}\le \theta<\I}
s^{-2}\brko{\prod_{h=1}^{q-2}s_h^{-1-\al}} 
\theta^{-\al-m\kappa}
\brko{\frac \theta{t_0}}^{-\frac{\al N}{2}}
\dd \theta\dd s_{q-2}\cdots\dd s_1\dd s
\\
&\quad=
t_0^\kappa\int_{t_0}^\I
s^{-2}\J_{q-2}^{m,N}(s;t_0)\dd s
\\
&\quad\le
\frac{C^qt_0^{-(q-1)\al}}{(N+1)_q}.
}
After the time integration,
\EQ{%\label{eq:eikonal-V4-after-time}
t_0^\kappa\norm{V_4(t_0)}_{\A^4_{\rho}} \lsm_{\mu,R}
C^qa_mA
t_0^{-(q-1)\al-\frac{\al N}{2}}g_q^{-N}
\prod_{h=1}^{q}\delta_h^{-1}
(N+1)_q^{-1}.
}
Choose the gaps exactly as in \eqref{eq:eikonal-V1-radius-choice}.  Then \eqref{eq:eikonal-V1-radius-bound} gives
\EQ{%\label{eq:eikonal-V4-bound}
\norm{V_4}_{\mathcal F_t^\kappa(\rho)}
\lsm_{\mu,R} C^qa_mA
(\rho_+-\rho)^{-N-q}
t^{-(q-1)\al-\frac{\al N}{2}}.
}

It remains to combine the four estimates.  The first three terms have the stronger decay $t^{-q\al}$, whereas $V_4$ has the decay $t^{-(q-1)\al}$ in the statement.  By \eqref{eq:eikonal-am-combinations},
\EQ{
\sum_{\substack{m_1,m_2\ge1\\m_1+m_2=m}}a_{m_1}a_{m_2}
=(2^m-2)a_m\le2^ma_m.
}
Using also $q\le2^q$, the decomposition \eqref{eq:eikonal-higher-V-decomposition} and the four estimates above give
\EQ{
\norm{\mathcal P_{m,q}(G_1,\ldots,G_m)}_{\mathcal F_t^\kappa(\rho)}
&\lsm_{\mu,R} C^{m+q}a_mA
(\rho_+-\rho)^{-N-q}
t^{-(q-1)\al-\frac{\al N}{2}}.
}
This proves \eqref{eq:eikonal-higher-layer-bound}.
\end{proof}

\begin{cor}[Estimate for one Picard layer]\label{cor:eikonal-picard-layer}
There exists $C_0\ge1$, depending only on the fixed parameters, with the following property.  Let $q\ge1$, $0\le\rho<\rho_+<\rho_1$, and assume \eqref{eq:eikonal-higher-layer-input}.  If $N=\sum_{j=1}^m n_j$, then
\EQn{\label{eq:eikonal-picard-layer-bound}
\norm{\mathcal P_{m,q}(G_1,\ldots,G_m)}_{\mathcal F_t^\kappa(\rho)}
\le C_0^{m+q}a_m
(\rho_+-\rho)^{-N-q}t^{-\frac{\al(N+q)}2}
\prod_{j=1}^m A_j.
}
\end{cor}

\begin{proof}
For $q\ge2$, Proposition \ref{prop:eikonal-higher-layer-estimate} gives \eqref{eq:eikonal-picard-layer-bound}, since $t\ge1$ and $(q-1)\al+\frac{\al N}{2}\ge\frac{\al(N+q)}2$, provided that $C_0$ is sufficiently large.

Let $q=1$ and set
\EQ{
\zeta=\rho+\frac{\rho_+-\rho}{N+2}.
}
The time-radius factor in Proposition \ref{prop:eikonal-Pm1-estimate}, applied at the radius $\zeta$, is
\EQ{
&(\zeta-\rho)^{-1}\int_t^\I s^{-1-\al}
\brko{\frac{s^{-\frac\al2}}
{\rho_+-\zeta}}^N\dd s
\\
&\quad=
\frac2\al\brko{\frac{N+2}{N+1}}^N
(\rho_+-\rho)^{-N-1}
t^{-\frac{\al(N+2)}2}
\\
&\quad\lsm
(\rho_+-\rho)^{-N-1}
t^{-\frac{\al(N+1)}2}.
}
Increasing $C_0$ proves \eqref{eq:eikonal-picard-layer-bound} for $q=1$ as well.
\end{proof}
\subsection{The normal-form fixed point}%\label{sec:eikonal-picard-fixed}

We now prove convergence of the Picard series defined by \eqref{eq:eikonal-picard-recursion}. Its sum is $\widetilde Z$, and the correction $Z$ is then recovered by \eqref{eq:eikonal-normal-decomposition}.
\begin{lem}[Graded analytic iteration]\label{lem:eikonal-picard-convergence}
Assume the layer estimate \eqref{eq:eikonal-picard-layer-bound}.
\begin{enumerate}[label=\textup{(\roman*)},leftmargin=*]
\item \emph{(Convergence of the graded Picard series)}.
Let $0\le\rho<\rho_+<\rho_1$, let $H_0=D_0$, and define $H_n$, $n\ge1$, successively by \eqref{eq:eikonal-picard-recursion}. Assume that
\EQ{
\norm{D_0}_{\mathcal F_T^\kappa(\rho_+)}\le M_0.
}
The sums defining every $H_n$ converge absolutely, and there are nonnegative coefficients $B_n$ such that
\EQn{\label{eq:eikonal-picard-order-bound}
\norm{H_n}_{\mathcal F_t^\kappa(\rho')}
\le B_n\brko{\frac{t^{-\frac\al2}}{\rho_+-\rho'}}^n,
\qquad t\ge T,\qquad 0\le\rho'<\rho_+,
\qquad n\ge0.
}
There exists $z_0>0$ such that, if
\EQn{\label{eq:eikonal-picard-smallness}
\frac{T^{-\frac\al2}}{\rho_+-\rho}\le z_0,
}
then $\sum_{n\ge0}H_n$ converges absolutely in $\mathcal F_T^\kappa(\rho)$ and
\EQn{\label{eq:eikonal-picard-absolute-bound}
\sum_{n\ge0}\norm{H_n}_{\mathcal F_T^\kappa(\rho)}
\le 2M_0,
}
where the bound is independent of $T$.

\item \emph{(Uniqueness of bounded fixed points)}.
Let
\EQ{
0\le\rho_-<\rho<\rho_+<\rho_1,
}
and let $\widetilde Z_1,\widetilde Z_2\in
\mathcal F_T^\kappa(\rho_+)$ satisfy
\EQ{
\widetilde Z_j=\widetilde{\TT}\widetilde Z_j
\quad\hbox{in }\mathcal F_T^\kappa(\rho),
\qquad j=1,2,
}
as well as
\EQ{
\norm{\widetilde Z_j}_{\mathcal F_T^\kappa(\rho_+)}
\le M,
\qquad j=1,2,
}
for some $M\ge1$. Assume that, at every radius below $\rho$, the series
\EQ{
D_0+\sum_{q\ge1}\sum_{m\ge1}
\mathcal P_{m,q}(\widetilde Z_j,\ldots,\widetilde Z_j)
}
converges absolutely and represents $\widetilde{\TT}\widetilde Z_j$.
Set
\EQ{
d_q(M)&:=C\sum_{m\ge1}mC_0^{m+q}a_mM^{m-1}
=C C_0^{q+1}C_4p e^{C_4pC_0M},
\\
d_M(z)&:=\sum_{q\ge1}d_q(M)z^q
=\frac{C C_0^2C_4p e^{C_4pC_0M}z}{1-C_0z},
\qquad |z|<C_0^{-1}.
}
Fix $0<z_M<C_0^{-1}$. If
\EQn{\label{eq:eikonal-uniqueness-series-condition}
\frac{T^{-\frac\al2}}{\rho_+-\rho}<z_M,
\qquad
d_M\brko{\frac{T^{-\frac\al2}}{\rho-\rho_-}}<1,
}
then $\widetilde Z_1=\widetilde Z_2$ at every radius below $\rho_-$.
\end{enumerate}
\end{lem}

\begin{proof}
For \textup{(i)}, set $B_0=M_0$. We prove
\eqref{eq:eikonal-picard-order-bound}, with the remaining scalar
coefficients fixed recursively below.
Suppose that \eqref{eq:eikonal-picard-order-bound} holds at all orders below $n$.  If $n_1+\cdots+n_m+q=n$, then $q\ge1$ implies $n_j<n$ for every $j$.  Hence Corollary \ref{cor:eikonal-picard-layer} gives
\EQn{\label{eq:eikonal-picard-term-induction}
\norm{\mathcal P_{m,q}(H_{n_1},\ldots,H_{n_m})}_{\mathcal F_t^\kappa(\rho')}
\le C_0^{m+q}a_mB_{n_1}\cdots B_{n_m}
\brko{\frac{t^{-\frac\al2}}{\rho_+-\rho'}}^n.
}
The induction closes with the recursive choice
\EQn{\label{eq:eikonal-picard-coefficient-recursion}
B_0&:=M_0,
\\
B_n&:=\sum_{q=1}^{n}\sum_{m\ge1}C_0^{m+q}a_m
\sum_{\substack{n_1,\ldots,n_m\ge0\\ n_1+\cdots+n_m=n-q}}
B_{n_1}\cdots B_{n_m},
\qquad n\ge1.
}
It remains to show that the constants in \eqref{eq:eikonal-picard-coefficient-recursion} are finite and summable. Recall that $C_4$ is the constant in the analytic Wiener algebra estimate \eqref{eq:fourier-algebra}. Since $a_m=(C_4 p)^m/m!$ by \eqref{eq:eikonal-am-definition}, introduce the real analytic function
\EQ{
\mathcal C(w,z)
:=\sum_{q\ge1}\sum_{m\ge1}C_0^{m+q}a_mw^mz^q=\frac{C_0z}{1-C_0z}\brko{e^{C_4 p C_0w}-1}.
}
This function is real analytic for $w\in\R$ and $|z|<C_0^{-1}$.  The function
\EQ{
\mathcal G(w,z):=w-B_0-\mathcal C(w,z)
}
satisfies
\EQ{
\mathcal G(B_0,0)=0,
\qquad
\pd_w\mathcal G(B_0,0)=1.
}
Lemma \ref{lem:real-analytic-ift} therefore gives a real analytic solution $B(z)$ of
\EQ{%\label{eq:eikonal-picard-generating-equation}
B(z)=B_0+\mathcal C(B(z),z),
\qquad B(0)=B_0,
}
for $|z|<z_0$, where $z_0>0$ is sufficiently small.  Write its convergent Taylor series as $B(z)=\sum_{n\ge0}\widehat B_nz^n$.  Absolute convergence permits the expansion
\EQ{%\label{eq:eikonal-scalar-coefficient-expansion}
\sum_{n\ge0}\widehat B_nz^n
=B_0+
\sum_{n\ge1}z^n\sum_{q=1}^{n}\sum_{m\ge1}
C_0^{m+q}a_m
\sum_{\substack{n_1,\ldots,n_m\ge0\\ n_1+\cdots+n_m=n-q}}
\widehat B_{n_1}\cdots\widehat B_{n_m}.
}
Comparing coefficients shows that $\widehat B_0=B_0$ and that $\widehat B_n$ satisfies \eqref{eq:eikonal-picard-coefficient-recursion}.  Since $q\ge1$, the right-hand side at order $n$ involves only the coefficients of lower order. The recursion is therefore uniquely determined, and induction gives $\widehat B_n=B_n$ for every $n$.  In particular, the coefficients $B_n$ are finite and nonnegative.  After decreasing $z_0$ if necessary, we have
\EQn{\label{eq:eikonal-picard-coefficient-sum}
\sum_{n\ge0}B_nz^n=B(z)\le 2B_0,
\qquad 0\le z\le z_0.
}

We now establish \eqref{eq:eikonal-picard-order-bound} and the absolute
convergence of the sums defining $H_n$. For $n=0$, the assumed bound on
$D_0$ and monotonicity in time and radius give
\EQ{
\norm{H_0}_{\mathcal F_t^\kappa(\rho')}
=\norm{D_0}_{\mathcal F_t^\kappa(\rho')}
\le M_0=B_0.
}
Assume that $H_0,\ldots,H_{n-1}$ have been defined and satisfy \eqref{eq:eikonal-picard-order-bound}.  Summing \eqref{eq:eikonal-picard-term-induction} and using \eqref{eq:eikonal-picard-coefficient-recursion}, we obtain
\EQ{%\label{eq:eikonal-picard-order-absolute}
\sum_{q=1}^{n}\sum_{m\ge1}
\sum_{\substack{n_1,\ldots,n_m\ge0\\ n_1+\cdots+n_m=n-q}}
\norm{\mathcal P_{m,q}(H_{n_1},\ldots,H_{n_m})}_{\mathcal F_t^\kappa(\rho')} \le
B_n\brko{\frac{t^{-\frac\al2}}{\rho_+-\rho'}}^n<\I.
}
Hence the sum defining $H_n$ in \eqref{eq:eikonal-picard-recursion} converges absolutely, and its sum satisfies \eqref{eq:eikonal-picard-order-bound}.  Induction proves both assertions for every $n\ge0$.

Evaluating \eqref{eq:eikonal-picard-order-bound} at $t=T$ and $\rho'=\rho$ and using \eqref{eq:eikonal-picard-smallness} and \eqref{eq:eikonal-picard-coefficient-sum} gives
\EQ{
\sum_{n\ge0}\norm{H_n}_{\mathcal F_T^\kappa(\rho)}\le 2B_0,
}
This proves \eqref{eq:eikonal-picard-absolute-bound} and the absolute convergence of the Picard series.

For \textup{(ii)}, \eqref{eq:eikonal-am-sum} and the choice of $z_M$ give
\EQ{%
\sum_{q\ge1}\sum_{m\ge1}
C_0^{m+q}a_mM^mz_M^q
=\frac{C_0z_M}{1-C_0z_M}
\brko{e^{C_4pC_0M}-1}<\I.
}
Hence Corollary \ref{cor:eikonal-picard-layer}, with degree zero
assigned to every input, gives, for $j=1,2$ and every
$\rho'\le\rho$,
\EQ{%
\sum_{q\ge1}\sum_{m\ge1}
\norm{\mathcal P_{m,q}(\widetilde Z_j,\ldots,\widetilde Z_j)}
_{\mathcal F_T^\kappa(\rho')}
\le\sum_{q\ge1}\sum_{m\ge1}C_0^{m+q}a_mM^m
\brko{\frac{T^{-\frac\al2}}{\rho_+-\rho'}}^q<\I.
}
By hypothesis, these series represent the two fixed-point equations.
For $q\ge1$, define
\EQ{%
\mathcal L_qF
:=\sum_{m\ge1}\sum_{h=1}^{m}
\mathcal P_{m,q}\brko{
\underbrace{\widetilde Z_2,\ldots,\widetilde Z_2}_{h-1},
F,
\underbrace{\widetilde Z_1,\ldots,\widetilde Z_1}_{m-h}}.
}
If $V=\widetilde Z_1-\widetilde Z_2$, the multilinear identity
\EQ{
&\mathcal M(A_1,\ldots,A_m)-\mathcal M(B_1,\ldots,B_m)
\\
&\quad=\sum_{h=1}^{m}
\mathcal M(B_1,\ldots,B_{h-1},A_h-B_h,A_{h+1},\ldots,A_m)
}
and absolute convergence permit termwise subtraction. Thus
\EQn{\label{eq:eikonal-uniqueness-difference-expansion}
V=\sum_{q\ge1}\mathcal L_qV
}
with absolute convergence at every radius below $\rho$.

Let $\rho_-\le\rho'<\rho$ and suppose that
\EQ{
\norm{F}_{\mathcal F_{t_0}^\kappa(\zeta)}
\le A\brko{\frac{t_0^{-\frac\al2}}
{\rho-\zeta}}^n,
\qquad
t_0\ge t,\qquad \rho'\le\zeta<\rho.
}
Since the two fixed points are bounded by $M$ at all lower radii,
Corollary \ref{cor:eikonal-picard-layer}, with degree zero assigned to
the remaining entries, yields
\EQn{\label{eq:eikonal-uniqueness-Lq-bound}
\norm{\mathcal L_qF}_{\mathcal F_t^\kappa(\rho')}
\le d_q(M)A\brko{\frac{t^{-\frac\al2}}
{\rho-\rho'}}^{n+q}.
}
For $\rho_-\le\zeta_0<\zeta_1<\rho$, the same estimate with
upper radius $\zeta_1$ and $n=0$ gives
\EQn{\label{eq:eikonal-uniqueness-Lq-continuity}
\norm{\mathcal L_qF}_{\mathcal F_t^\kappa(\zeta_0)}
\le d_q(M)\brko{\frac{t^{-\frac\al2}}
{\zeta_1-\zeta_0}}^q
\norm{F}_{\mathcal F_t^\kappa(\zeta_1)}.
}
In particular,
$\mathcal L_q:\mathcal F_t^\kappa(\zeta_1)
\ra
\mathcal F_t^\kappa(\zeta_0)$ is continuous.

For $q_1,\ldots,q_k\ge1$, induction using
\eqref{eq:eikonal-uniqueness-Lq-bound} gives, for
$\rho_-\le\rho'<\rho$ and $t\ge T$,
\EQn{\label{eq:eikonal-uniqueness-composition-bound}
\norm{\mathcal L_{q_1}\cdots\mathcal L_{q_k}V}
_{\mathcal F_t^\kappa(\rho')} \le d_{q_1}(M)\cdots d_{q_k}(M)
\brko{\frac{t^{-\frac\al2}}
{\rho-\rho'}}^{q_1+\cdots+q_k}
\norm{V}_{\mathcal F_T^\kappa(\rho_+)}.
}
Indeed, the case $k=1$ follows from
\eqref{eq:eikonal-uniqueness-Lq-bound} with $n=0$. If the bound holds
with $k-1$ factors, then, for $t_0\ge t$ and
$\rho'\le\zeta<\rho$,
\EQ{
\norm{\mathcal L_{q_2}\cdots\mathcal L_{q_k}V}
_{\mathcal F_{t_0}^\kappa(\zeta)} \le d_{q_2}(M)\cdots d_{q_k}(M)
\brko{\frac{t_0^{-\frac\al2}}
{\rho-\zeta}}^{q_2+\cdots+q_k}
\norm{V}_{\mathcal F_T^\kappa(\rho_+)},
}
and another application of \eqref{eq:eikonal-uniqueness-Lq-bound}
proves the induction step. Summing
\eqref{eq:eikonal-uniqueness-composition-bound} gives
\EQn{\label{eq:eikonal-uniqueness-composition-series}
\sum_{q_1,\ldots,q_k\ge1}
\norm{\mathcal L_{q_1}\cdots\mathcal L_{q_k}V}
_{\mathcal F_T^\kappa(\rho_-)} \le
d_M\brko{\frac{T^{-\frac\al2}}
{\rho-\rho_-}}^k
\norm{V}_{\mathcal F_T^\kappa(\rho_+)}<\I.
}
For $k\ge2$, choose
\EQ{
\rho_-=\zeta_0<\zeta_1<\cdots<\zeta_{k-1}<\rho.
}
By \eqref{eq:eikonal-uniqueness-difference-expansion},
\eqref{eq:eikonal-uniqueness-Lq-continuity}, and
\eqref{eq:eikonal-uniqueness-composition-series},
\EQ{%
V
=\sum_{q_1\ge1}\mathcal L_{q_1}V
=\cdots
=\sum_{q_1,\ldots,q_k\ge1}
\mathcal L_{q_1}\cdots\mathcal L_{q_k}V
\quad\hbox{in }\mathcal F_T^\kappa(\rho_-).
}
Consequently,
\EQn{\label{eq:eikonal-uniqueness-iteration-bound}
\norm{V}_{\mathcal F_T^\kappa(\rho_-)}
\le d_M\brko{\frac{T^{-\frac\al2}}
{\rho-\rho_-}}^k
\norm{V}_{\mathcal F_T^\kappa(\rho_+)},
\qquad k\ge1.
}
The second condition in
\eqref{eq:eikonal-uniqueness-series-condition} allows us to let
$k\ra\I$ in \eqref{eq:eikonal-uniqueness-iteration-bound}; hence $V=0$ in
$\mathcal F_T^\kappa(\rho_-)$. Radius monotonicity proves the same
conclusion at every lower radius.
\end{proof}
\begin{prop}[Normal-form fixed point around the Fuchsian profile]\label{prop:eikonal-fixed}
Under the assumptions of Theorem \ref{thm:eikonal}, set
\EQ{
\rho_-:=\frac{\rho_*+\rho_1}{2},
\qquad
\rho:=\frac{2\rho_-+\rho_1}{3},
\qquad
\rho_+:=\frac{\rho+\rho_1}{2}.
}
There exist $\ep_0>0$, $T\ge2$, and $R_0>1$ such that, if $0<\ep\le\ep_0$, there is $\widetilde Z\in\mathcal F_T^\kappa(\rho)$ satisfying
\EQ{
\normb{\widetilde Z}_{\mathcal F_T^\kappa(\rho)}\le R_0.
}
The fixed-point identity holds at every lower radius:
\EQ{
\widetilde Z=\widetilde{\TT}\widetilde Z
\quad\hbox{in }\mathcal F_T^\kappa(\rho')
\quad\hbox{for every }\rho_*<\rho'<\rho_-.
}
The reconstructed correction
\EQ{
Z=\widetilde Z+iK_\mu(\re\widetilde Z)
}
is a final-state correction in the sense of Definition \ref{def:eikonal-asymptotic} and satisfies
\EQn{\label{eq:eikonal-main-Z-bound}
t^\kappa\norm{\re Z(t)}_{\A^4_{\rho_-}}+t^{\kappa+\al-1}\norm{\im Z(t)}_{\A^4_{\rho_-}}\lsm_{\mu,R}1,\qquad t\ge T.
}
Suppose that, for some $T_0\ge2$ and $\rho_*<\zeta<\rho_1$, two functions $\widetilde Z_j$, $j=1,2$, satisfy
\EQ{
\widetilde Z_j&\in\mathcal F_{T_0}^\kappa(\zeta),
\\
\widetilde Z_j&=\widetilde{\TT}\widetilde Z_j
\quad\hbox{in }\mathcal F_{T_0}^\kappa(\zeta')
\quad\hbox{for every }\rho_*<\zeta'<\zeta.
}
If both fixed-point identities are associated with the same truncated Fuchsian profile, then there exists $T_*\ge T_0$ such that
\EQ{
\widetilde Z_1(t)=\widetilde Z_2(t),
\qquad t\ge T_*.
}
\end{prop}

\begin{proof}
Let
\EQ{
\delta_*=\min\fbrko{\rho-\rho_-,\rho_+-\rho}.
}
Let $H_n$ be defined by \eqref{eq:eikonal-picard-recursion}. By
\eqref{eq:eikonal-D0-bound} and radius monotonicity, fix $M_0\ge1$ such that
$\norm{D_0}_{\mathcal F_T^\kappa(\rho_+)}\le M_0$. Let $z_0$ be
the constant in Lemma \ref{lem:eikonal-picard-convergence}\textup{(i)}.
Choose $\ep_0>0$ small enough for the preceding estimates, choose
$R_0>2M_0$, and then choose $T\ge2$ so that
\EQn{\label{eq:eikonal-normal-smallness}
\frac{T^{-\frac\al2}}{\delta_*}\le z_0,
\qquad
\frac{C_{\mu,R}T^{-\al}}{\delta_*}\le\frac14.
}
Since $\delta_*\le\rho_+-\rho$, Lemma
\ref{lem:eikonal-picard-convergence}\textup{(i)} gives
\EQ{
\sum_{n\ge0}\norm{H_n}_{\mathcal F_T^\kappa(\rho)}
\le2M_0<R_0.
}
Let
\EQ{
\widetilde Z:=\sum_{n\ge0}H_n.
}
Corollary \ref{cor:eikonal-picard-layer} and the order-by-order estimates
in Lemma \ref{lem:eikonal-picard-convergence}\textup{(i)} give
\EQn{\label{eq:eikonal-fixed-nonlinear-sum}
\sum_{q\ge1}\sum_{m\ge1}\sum_{n_1,\ldots,n_m\ge0}
\norm{\mathcal P_{m,q}(H_{n_1},\ldots,H_{n_m})}_{\mathcal F_T^\kappa(\rho_-)}<\I.
}
The second inequality in \eqref{eq:eikonal-normal-smallness} and Lemma \ref{lem:eikonal-K-tail}\textup{(ii)}, with lower radius $\rho_-$ and upper radius $\rho$, show that the double-series expansion \eqref{eq:eikonal-K-double-series} for $K_\mu(\re\widetilde Z)$ converges absolutely. The absolute convergence in \eqref{eq:eikonal-fixed-nonlinear-sum} permits the resulting products to be rearranged as
\EQn{\label{eq:eikonal-fixed-series-product}
\sum_{q\ge1}\sum_{m\ge1}
\mathcal P_{m,q}(\widetilde Z,\ldots,\widetilde Z)=\sum_{q\ge1}\sum_{m\ge1}\sum_{n_1,\ldots,n_m\ge0}
\mathcal P_{m,q}(H_{n_1},\ldots,H_{n_m})
=\sum_{n\ge1}H_n
}
in $\mathcal F_T^\kappa(\rho_-)$. Thus \eqref{eq:eikonal-fixed-series-product}, \eqref{eq:eikonal-q-index-expansion}, and \eqref{eq:eikonal-picard-recursion} identify the formal expansion with the actual map and give
\EQn{\label{eq:eikonal-normal-fixed-equation}
\widetilde{\TT}\widetilde Z
=D_0+\sum_{q\ge1}\sum_{m\ge1}\mathcal P_{m,q}(\widetilde Z,\ldots,\widetilde Z)
=H_0+\sum_{n\ge1}H_n
=\widetilde Z
}
in $\mathcal F_T^\kappa(\rho_-)$. By radius monotonicity, the identity holds in $\mathcal F_T^\kappa(\rho')$ for every $\rho_*<\rho'<\rho_-$. Moreover,
\EQn{\label{eq:eikonal-fixed-point-bound}
\normb{\widetilde Z}_{\mathcal F_T^\kappa(\rho)}\le R_0.
}
Lemma \ref{lem:eikonal-K-tail}\textup{(ii)}, with lower radius $\rho_-$ and upper radius $\rho$, gives
\EQn{\label{eq:eikonal-K-final-bound}
\normb{K_\mu(\re\widetilde Z)(t)}_{\A^4_{\rho_-}}
\lsm_{\mu,R}\ep t^{1-\al-\kappa}.
}
Let $Z=\widetilde Z+iK_\mu(\re\widetilde Z)$.  Since $K_\mu(\re\widetilde Z)$ is real-valued,
\EQ{
\re Z=\re\widetilde Z,
\qquad
\im Z=\im\widetilde Z+K_\mu(\re\widetilde Z).
}
Combining this identity with \eqref{eq:eikonal-fixed-point-bound} and \eqref{eq:eikonal-K-final-bound} gives \eqref{eq:eikonal-main-Z-bound}.

Fix $\rho_*<\rho'<\rho''<\rho_-$. Lemma \ref{lem:eikonal-K-tail}\textup{(ii)} gives \eqref{eq:eikonal-K-integral} at radius $\rho''$. Hence Lemma \ref{lem:eikonal-Duhamel-conjugacy} and \eqref{eq:eikonal-normal-fixed-equation} give
\EQ{
Z=\TT_\mu Z
\quad\hbox{in }\mathcal Z_T^\kappa(\rho').
}
Since $\rho'$ is arbitrary, \eqref{eq:eikonal-main-Z-bound} proves that $Z$ is a final-state correction with $\rho=\rho_-$ in Definition \ref{def:eikonal-asymptotic}.

For uniqueness, let $\widetilde Z_j$, $j=1,2$, satisfy the assumptions in the proposition. Choose
\EQ{
\rho_*<\zeta_-<\zeta'<\zeta_+<\zeta.
}
Set
\EQ{
M:=1+\sum_{j=1}^2\norm{\widetilde Z_j}_{\mathcal F_{T_0}^\kappa(\zeta_+)}.
}
Fix $0<z_M<C_0^{-1}$. Increase $T$ so that
\EQ{
T\ge T_0,\quad
\frac{C_{\mu,R}T^{-\al}}{\zeta_+-\zeta'}\le\frac14,\quad
\frac{T^{-\frac\al2}}{\zeta_+-\zeta'}<z_M,\quad
d_M\brko{\frac{T^{-\frac\al2}}{\zeta'-\zeta_-}}<1.
}
The assumptions give the two fixed-point identities in $\mathcal F_T^\kappa(\zeta')$. Lemma \ref{lem:eikonal-K-tail}\textup{(ii), (v)}, \eqref{eq:eikonal-q-index-expansion}, and Corollary \ref{cor:eikonal-picard-layer} show that both transformed-map series converge absolutely and represent $\widetilde{\TT}$ below $\zeta'$. Lemma \ref{lem:eikonal-picard-convergence}\textup{(ii)}, applied with $(\rho_-,\rho,\rho_+)=(\zeta_-,\zeta',\zeta_+)$, gives $\widetilde Z_1=\widetilde Z_2$ below $\zeta_-$. Hence the two functions agree on $[T,\I)$.
\end{proof}

\begin{lem}[Automatic normal-form reduction]\label{lem:eikonal-automatic-normal-form}
Under the assumptions of Theorem \ref{thm:eikonal}, let $(L^{(\mu)},S^{(\mu)})$ be the truncated Fuchsian profile constructed in Proposition \ref{prop:eikonal-prep}. Suppose that $Z$ is a final-state correction associated with this profile and that $(T_0,\zeta_+)$ realizes Definition \ref{def:eikonal-asymptotic}. Then there exist $T\ge T_0$ and $\rho_*<\rho<\zeta_+$ such that the functions
\EQ{
K:=K_\mu(\re Z),
\qquad
\widetilde Z:=Z-iK
}
satisfy
\EQ{
\widetilde Z&\in\mathcal F_T^\kappa(\rho),
\\
K&\in\mathcal F_T^{\kappa+\al-1}(\rho),
\\
\widetilde Z&=\widetilde{\TT}\widetilde Z
\quad\hbox{in }\mathcal F_T^\kappa(\rho')
\quad\hbox{for every }\rho_*<\rho'<\rho.
}
\end{lem}

\begin{proof}
Set
\EQ{
N_*:=\lceil\gamma\rceil.
}
Choose
\EQ{
\rho_*<\zeta_0<\zeta_1<\cdots<\zeta_{N_*+2}<\zeta_+.
}
Writing $Z=X+iY$, the defining bound gives
\EQ{
X\in\mathcal F_{T_0}^\kappa(\zeta_+),
\qquad
Y\in\mathcal F_{T_0}^{\kappa+\al-1}(\zeta_+).
}
After choosing $T\ge T_0$ sufficiently large, Lemma \ref{lem:eikonal-K-tail}\textup{(ii)} gives
\EQ{
K_\mu(X)\in\mathcal F_T^{\kappa+\al-1}(\zeta_{N_*+2})
}
and \eqref{eq:eikonal-K-integral} at every lower radius. Hence
\EQ{
\widetilde Z=Z-iK_\mu(X)\in\mathcal F_T^{\kappa+\al-1}(\zeta_{N_*+2}),
\qquad
\re\widetilde Z=X\in\mathcal F_T^\kappa(\zeta_{N_*+2}).
}
The defining identity $Z=\TT_\mu Z$ and Lemma \ref{lem:eikonal-Duhamel-conjugacy} give
\EQ{
\widetilde Z=\widetilde{\TT}\widetilde Z
}
at every radius below $\zeta_{N_*+1}$.

We now record the decay improvement obtained from this integral equation. For $1\le j\le N_*$, if
\EQ{
\widetilde Z\in\mathcal F_T^\sigma(\zeta_j),
\qquad
\kappa+\al-1\le\sigma\le\kappa,
}
then \eqref{eq:U-isometry}, \eqref{eq:Rmu-bound}, \eqref{eq:eikonal-P-estimate}, \eqref{eq:eikonal-direct-one-index-estimates}, and the radius-gap estimate give
\EQn{\label{eq:eikonal-automatic-normal-step}
\normb{\widetilde Z(t)}_{\A^4_{\zeta_{j-1}}}
&\lsm t^{-\min\{\kappa,\sigma+\al\}}.
}
Starting with $\sigma_0=\kappa+\al-1$ and applying \eqref{eq:eikonal-automatic-normal-step} successively along the radius chain gives
\EQ{
\widetilde Z\in\mathcal F_T^{\sigma_j}(\zeta_{N_*-j}),
\qquad
\sigma_j:=\min\fbrko{\kappa,\kappa+(j+1)\al-1},
\qquad
j=0,\ldots,N_*.
}
Since $1-\al=\al\gamma$,
\EQ{
\kappa+(N_*+1)\al-1
=\kappa+\al(N_*-\gamma)\ge\kappa.
}
Thus $\widetilde Z\in\mathcal F_T^\kappa(\zeta_0)$. Together with the fixed-point identity already obtained at every radius below $\zeta_{N_*+1}$, this proves the assertion with $\rho=\zeta_0$.
\end{proof}

\subsection{Construction, convergence, and uniqueness for a fixed truncation order}%\label{sec:eikonal-construction}

\begin{proof}[Proof of Theorem \ref{thm:eikonal}\textup{(1)--(3)}]

Proposition \ref{prop:eikonal-prep} gives the coefficients and the truncated Fuchsian profile in Theorem \ref{thm:eikonal}.
Let $\widetilde Z$ be the fixed point in Proposition \ref{prop:eikonal-fixed}, set
\EQ{
K=K_\mu(\re\widetilde Z),
\qquad
Z=\widetilde Z+iK,
\qquad
a=e^{\Phi^{(\mu)}+Z},
}
and define $u$ from $a$ by the ray transform \eqref{eq:ray-transform}.

On every compact interval $I\subset[T,\I)$, \eqref{eq:eikonal-amplitude-comparable}, \eqref{eq:wiener-L2-product}, and the $L^2$ isometry of the ray transform give
\EQ{
a=e^{\Phi^{(\mu)}+Z}
&\in C(I;L^2_y).
}
Moreover, $u\in C(I;L^2_x)$ solves \eqref{eq:nls} on $[T,\I)$, satisfying
\EQ{
\norm{u(t)}_{L^2_x}=\norm{a(t)}_{L^2_y}
&\lsm_{\mu,R}e^{\norm{\re Z(t)}_{L^\I_y}}\norm{W}_{L^2_y}.
}
The defining identity $Z=\TT_\mu Z$ is the integral form of the phase correction equation. By Proposition \ref{prop:eikonal-fixed}, $Z$ is a final-state correction in the sense of Definition \ref{def:eikonal-asymptotic}.

Since $W\in L^2\cap\A^6_{\rho_0}$, one has $W\in L^q$ for $2\le q\le\I$. The bound \eqref{eq:eikonal-main-Z-bound} and the Wiener embedding give
\EQn{\label{eq:eikonal-Z-Linfty-decay}
\norm{\re Z(t)}_{L^\I_y}
&\lsm_{\mu,R}t^{-\kappa},
\\
\norm{Z(t)}_{L^\I_y}
&\lsm_{\mu,R}t^{-(\kappa+\al-1)}.
}
By \eqref{eq:eikonal-amplitude-comparable},
\EQn{\label{eq:eikonal-profile-Lq}
\norm{e^{\Phi^{(\mu)}(t)}}_{L^q_y}
=\norm{e^{\re L^{(\mu)}(t^{-\al})}}_{L^q_y}
\lsm_{\mu,R}\norm{W}_{L^q_y}.
}
Since $|e^z-1|\le e^{|\re z|}|z|$, \eqref{eq:ray-Lq-scaling}, \eqref{eq:eikonal-Z-Linfty-decay}, and \eqref{eq:eikonal-profile-Lq} give
\EQ{
\norm{u(t)-u_{\rm eik}^{(\mu)}(t)}_{L^q_x}
&\le t^{-d\brko{\frac12-\frac1q}}\norm{e^{\Phi^{(\mu)}(t)}}_{L^q_y}
e^{\norm{\re Z(t)}_{L^\I_y}}\norm{Z(t)}_{L^\I_y}
\\
&\lsm_{\mu,R}\norm{W}_{L^q_y}t^{-d\brko{\frac12-\frac1q}-(\kappa+\al-1)}.
}
Moreover,
\EQ{
|u(t,x)|-|u_{\rm eik}^{(\mu)}(t,x)|
=t^{-\frac d2}e^{\re L^{(\mu)}(t^{-\al},x/t)}
\brko{e^{\re Z(t,x/t)}-1},
}
and hence
\EQ{
\norm{|u(t)|-|u_{\rm eik}^{(\mu)}(t)|}_{L^q_x}
\lsm_{\mu,R}\norm{W}_{L^q_y}t^{-d\brko{\frac12-\frac1q}-\kappa}.
}
Finally, $L^{(\mu)}=\Psi+\Omega^{(\mu)}$, so \eqref{eq:eikonal-positive-order} and \eqref{eq:ray-Lq-scaling} yield
\EQ{
\norm{|u_{\rm eik}^{(\mu)}(t)|
-t^{-\frac d2}|W(x/t)|}_{L^q_x} \lsm_{\mu,R}\norm{W}_{L^q_y}
t^{-d\brko{\frac12-\frac1q}-\al}(1+\log t)^{M_\mu}.
}
Since $\kappa>1-\al\ge\al$,
\EQ{
\absb{t^{d\brko{\frac12-\frac1q}}\norm{u(t)}_{L^q_x}-\norm{W}_{L^q_y}}
&\le t^{d\brko{\frac12-\frac1q}}\norm{|u(t)|-|u_{\rm eik}^{(\mu)}(t)|}_{L^q_x}
\\
&\quad+t^{d\brko{\frac12-\frac1q}}\norm{|u_{\rm eik}^{(\mu)}(t)|-t^{-\frac d2}|W(x/t)|}_{L^q_x}
\\
&\lsm_{\mu,R}\norm{W}_{L^q_y}t^{-\al}(1+\log t)^{M_\mu}.
}
This proves \eqref{eq:eikonal-Lq-remainder}--\eqref{eq:eikonal-Lq-leading}, and hence the existence and convergence assertions of Theorem \ref{thm:eikonal} follow.

We now prove uniqueness for the fixed truncated Fuchsian profile $(L^{(\mu)},S^{(\mu)})$. Let $Z_j$, $j=1,2$, be two final-state corrections associated with this profile.
Define
\EQ{
a_j(t,y)&:=e^{\Phi^{(\mu)}(t,y)+Z_j(t,y)},
\\
u_j(t,x)&:=\frac{1}{(it)^{\frac d2}}e^{\frac{i|x|^2}{2t}}a_j\brko{t,\frac{x}{t}}.
}
Applying Lemma \ref{lem:eikonal-automatic-normal-form} to each $Z_j$, then increasing $T$ and decreasing the radius, we obtain
\EQ{
T\ge2,
\qquad
\rho_*<\rho<\rho_1,
}
and, for $j=1,2$,
\EQ{
K_j&:=K_\mu(\re Z_j),
\qquad
\widetilde Z_j:=Z_j-iK_j,
\\
\widetilde Z_j&\in\mathcal F_T^\kappa(\rho),
\qquad
K_j\in\mathcal F_T^{\kappa+\al-1}(\rho),
\\
\widetilde Z_j&=\widetilde{\TT}\widetilde Z_j
\quad\hbox{in }\mathcal F_T^\kappa(\rho')
\quad\hbox{for every }\rho_*<\rho'<\rho.
}
After increasing $T$ if necessary, Proposition \ref{prop:eikonal-fixed} gives
\EQ{
\widetilde Z_1(t)=\widetilde Z_2(t),
\qquad t\ge T.
}
Fix $\rho_*<\rho'<\rho''<\rho$. Since $K_j$ is real-valued,
\EQ{
X:=\re Z_1=\re\widetilde Z_1=\re\widetilde Z_2=\re Z_2
\in\mathcal F_T^\kappa(\rho'').
}
Lemma \ref{lem:eikonal-K-tail}\textup{(v)} therefore gives
\EQ{
K_1=K_2
\quad\hbox{in }C([T,\I);\A^4_{\rho'}).
}
Consequently,
\EQ{
Z_1=\widetilde Z_1+iK_1
=\widetilde Z_2+iK_2=Z_2
\quad\hbox{in }C([T,\I);\A^4_{\rho'}),
\\
a_1=a_2,
\qquad
u_1=u_2,
\qquad t\ge T.
}
Since $\rho'$ is arbitrary, this proves the uniqueness assertion of Theorem \ref{thm:eikonal}.
\end{proof}

\subsection{Compatibility of truncation orders}

\begin{lem}[Compatibility under a change of Fuchsian truncation]
\label{lem:eikonal-truncation-change}
Let $\mu_j>2\gamma$, $j=1,2$, be integers, and let $(L^{(\mu_j)},S^{(\mu_j)})$ be the truncated Fuchsian profiles of orders $\mu_j$ constructed from the same parameters $(\Psi,C_+)$. Let $\Phi^{(\mu_j)}$ be the corresponding logarithmic phases, and assume that
\EQ{
1-\al<\kappa<\al\brko{\min\fbrko{\mu_1,\mu_2}+1}-1.
}
Assume that $Z_2$ satisfies \eqref{eq:eikonal-final-state-correction} with $\mu=\mu_2$ for the profile $(L^{(\mu_2)},S^{(\mu_2)})$. Then
\EQ{
Z_{21}:=Z_2+\Phi^{(\mu_2)}-\Phi^{(\mu_1)}
}
satisfies \eqref{eq:eikonal-final-state-correction} with $\mu=\mu_1$ for the profile $(L^{(\mu_1)},S^{(\mu_1)})$ and the same exponent $\kappa$.
\end{lem}

\begin{proof}
Choose $T_0\ge2$ and $\rho_*<\zeta_+<\rho_1$ so that \eqref{eq:eikonal-final-state-correction} holds for $Z=Z_2$ and $\mu=\mu_2$. Choose $T\ge T_0$ sufficiently large and fix $\rho_*<\rho<\zeta_+$. Set
\EQ{
\bar\mu&:=\min\fbrko{\mu_1,\mu_2},
\qquad
\sigma_0:=\al\bar\mu,
\\
\Phi_{21}&:=\Phi^{(\mu_2)}-\Phi^{(\mu_1)},
\\
L_{21}(t)&:=L^{(\mu_2)}(t^{-\al})-L^{(\mu_1)}(t^{-\al}),
\\
S_{21}(t)&:=S^{(\mu_2)}(t^{-\al})-S^{(\mu_1)}(t^{-\al}).
}
Then
\EQ{
\re\Phi_{21}=\re L_{21},
\qquad
t^{\al-1}\im\Phi_{21}=t^{\al-1}\im L_{21}-S_{21}.
}
By \eqref{eq:eikonal-separated-profile-diff} and radius monotonicity, there exists $N\in\N$ such that
\EQn{\label{eq:eikonal-Phi-truncation-diff}
\norm{\re\Phi_{21}(t)}_{\A^4_\rho}
+t^{\al-1}\norm{\im\Phi_{21}(t)}_{\A^4_\rho}
\lsm t^{-\sigma_0}(1+\log t)^N.
}
Since $\bar\mu>2\gamma$, one has $\sigma_0>2(1-\al)$. The upper bound for $\kappa$ gives $\kappa<\sigma_0+\al-1<\sigma_0$, and hence
\EQ{
\Phi_{21}\in\mathcal Z_T^\kappa(\rho),
\qquad
Z_{21}=Z_2+\Phi_{21}\in\mathcal Z_T^\kappa(\rho).
}
Fix $\rho_*<\rho'<\rho$. By \eqref{eq:Rmu-def}, \eqref{eq:Nmu-def}, and the definition of $Z_{21}$, one has
\EQ{
\pd_t\Phi_{21}
=\frac{i}{2t^2}\Dy\Phi_{21}
+N^{(\mu_1)}(t,Z_{21})-N^{(\mu_2)}(t,Z_2).
}
Differentiating the finite Fuchsian expansions gives, after increasing $N$ if necessary,
\EQ{
\norm{\pd_t\Phi_{21}(t)}_{\A^4_{\rho'}}
\lsm t^{-\al-\sigma_0}(1+\log t)^N.
}
On the other hand, \eqref{eq:radius-gap-estimate} and \eqref{eq:eikonal-Phi-truncation-diff} give
\EQ{
t^{-2}\norm{\Dy\Phi_{21}(t)}_{\A^4_{\rho'}}
\lsm \frac{t^{-2}}{(\rho-\rho')^2}
\norm{\Phi_{21}(t)}_{\A^4_\rho}
\lsm t^{-1-\al-\sigma_0}(1+\log t)^N.
}
Since $\al+\sigma_0>1$, the preceding estimates and the differential identity imply
\EQ{
N^{(\mu_1)}(\cdot,Z_{21})-N^{(\mu_2)}(\cdot,Z_2)
\in L^1([T,\I);\A^4_{\rho'}).
}
For $T'>t$, Duhamel's formula gives
\EQ{
\Phi_{21}(t)
=U(t,T')\Phi_{21}(T') -\int_t^{T'}U(t,s)\mbrko{
N^{(\mu_1)}(s,Z_{21}(s))-N^{(\mu_2)}(s,Z_2(s))}\dd s.
}
Moreover, \eqref{eq:eikonal-Phi-truncation-diff} and \eqref{eq:U-isometry} give
\EQ{
\norm{U(t,T')\Phi_{21}(T')}_{\A^4_{\rho'}}
\lsm (T')^{1-\al-\sigma_0}(1+\log T')^N\ra0
}
as $T'\ra\I$. Therefore
\EQ{
\Phi_{21}(t)
=-\int_t^\I U(t,s)\mbrko{
N^{(\mu_1)}(s,Z_{21}(s))-N^{(\mu_2)}(s,Z_2(s))}\dd s.
}
Adding the defining Duhamel identity for $Z_2$ yields
\EQ{
Z_{21}=\TT_{\mu_1}Z_{21}
\quad\hbox{in }\mathcal Z_T^\kappa(\rho').
}
Since $\rho'$ is arbitrary, this identity and $Z_{21}\in\mathcal Z_T^\kappa(\rho)$ prove \eqref{eq:eikonal-final-state-correction} with $Z=Z_{21}$ and $\mu=\mu_1$.
\end{proof}

\begin{prop}[Compatibility of truncation orders]
\label{prop:eikonal-truncation-compatibility}
For $j=1,2$, let $\mu_j>2\gamma$ be an integer, let $\kappa_j$ be an admissible exponent for the construction of order $\mu_j$, and let $u_j$ be the resulting solution for the same Fuchsian profile parameters. Then there exists $T\ge2$ such that $u_1=u_2$ on $[T,\I)$.
\end{prop}

\begin{proof}
Let $Z_j$ be the final-state correction defining $u_j$. After interchanging the indices if necessary, assume that $\kappa_1\le\kappa_2$. Then
\EQ{
1-\al<\kappa_1
<\al\brko{\min\fbrko{\mu_1,\mu_2}+1}-1.
}
After restricting the two corrections to a common interval $[T,\I)$ and a common analytic radius, $Z_2$ is also a final-state correction with exponent $\kappa_1$. Lemma \ref{lem:eikonal-truncation-change} gives
\EQ{
Z_{21}:=Z_2+\Phi^{(\mu_2)}-\Phi^{(\mu_1)}
}
as a final-state correction for the truncation of order $\mu_1$ with exponent $\kappa_1$. After increasing $T$ if necessary, the uniqueness assertion in Theorem \ref{thm:eikonal}\textup{(3)}, which has already been proved, gives $Z_1=Z_{21}$. By the definition of $Z_{21}$,
\EQ{
\Phi^{(\mu_1)}+Z_1
=\Phi^{(\mu_2)}+Z_2.
}
Hence $u_1=u_2$ on $[T,\I)$.
\end{proof}

\begin{proof}[Proof of Theorem \ref{thm:eikonal}\textup{(4)}]
Let $u_j$, $j=1,2$, be the solutions corresponding to $(\mu_j,\kappa_j)$ in \textup{(4)}. Proposition \ref{prop:eikonal-truncation-compatibility} gives $T\ge2$ such that
\EQ{
u_1(t)=u_2(t),
\qquad t\ge T.
}
Since $0<p\le1/d<4/d$, the equation is mass-subcritical. The standard $L^2(\R^d)$ theory \cite{Tsutsumi1987} extends each $u_j$ uniquely to $[0,\I)$. Uniqueness of the $L^2$ flow gives
\EQ{
u_1=u_2
\quad\hbox{in }C([0,\I);L^2).
}
In particular, $u_1(0)=u_2(0)$, which proves the independence of the truncation order $\mu$ and the exponent $\kappa$.
\end{proof}

\subsection{Independence of the final-data decomposition}

\begin{prop}[Independence of the parametrization]
\label{prop:eikonal-final-datum-invariance}
Fix $d,p,\rho_*,\rho_1,\rho_0,R$ as in Theorem \ref{thm:eikonal}. Let $\Theta_+$ admit two representations
\EQ{
\Theta_+=\Psi_j-iC_{+,j},
\qquad j=1,2,
}
where $C_{+,j}$ is real-valued. For each $j$, choose an integer $\mu_j>2\gamma$ and an exponent $\kappa_j$ satisfying
\EQ{
1-\al<\kappa_j<\al(\mu_j+1)-1,
}
and assume that $W_j=e^{\Psi_j}$ and $C_{+,j}$ satisfy the data and smallness assumptions of Theorem \ref{thm:eikonal} for $(\mu_j,\kappa_j)$. Let $(L_j^{(\mu_j)},S_j^{(\mu_j)})$ be the truncated Fuchsian profile and let $Z_j$ be the final-state correction given by that theorem. Set
\EQ{
\Phi_j(t,y)
:=L_j^{(\mu_j)}(t^{-\al},y)
-it^{1-\al}S_j^{(\mu_j)}(t^{-\al},y).
}
Define the associated NLS solution for large times by
\EQ{
u_j(t,x)
:=\frac1{(it)^{\frac d2}}e^{\frac{i|x|^2}{2t}}
\exp\mbrko{
\Phi_j\brko{t,\frac{x}{t}}
+Z_j\brko{t,\frac{x}{t}}},
}
and use the same notation for its unique global $L^2$ extension. Then there exists $T\ge2$ such that
\EQ{
\Phi_1(t)+Z_1(t)=\Phi_2(t)+Z_2(t),
\qquad t\ge T,
}
and
\EQ{
u_1=u_2
\quad\hbox{in }C([0,\I);L^2).
}
Consequently, the solution obtained from the Fuchsian construction depends only on
\EQ{
\Theta_+=\Psi-iC_+.
}
For real-valued $\chi$, the change
\EQ{
(\Psi,C_+)\longmapsto(\Psi+i\chi,C_++\chi)
}
leaves the global solution unchanged whenever both parameter pairs satisfy the assumptions above.
\end{prop}

\begin{proof}
After interchanging the indices if necessary, assume that $\kappa_1\le\kappa_2$, and set
\EQ{
\bar\mu:=\min\fbrko{\mu_1,\mu_2},
\qquad
\chi:=C_{+,2}-C_{+,1}.
}
Then $\Psi_2=\Psi_1+i\chi$.

\emph{Step 1: relation between the finite profiles.} For $j=1,2$, let $(L_{j,<\bar\mu},S_{j,<\bar\mu})$ be the coefficient truncation of $(L_j^{(\mu_j)},S_j^{(\mu_j)})$ obtained by retaining the terms with $\tau$-exponent below $\bar\mu$. For $b>0$, let $\Pi_{<b}$ retain the terms with $\tau$-exponent strictly below $b$, including the exponent-zero term. We construct a finite real-valued expansion $h$ such that
\EQ{
L_{2,<\bar\mu}&=L_{1,<\bar\mu}+ih,
\\
S_{2,<\bar\mu}&=S_{1,<\bar\mu}
+\tau^\gamma\Pi_{<\bar\mu-\gamma}h.
}
The coefficient construction of Proposition \ref{prop:eikonal-prep} gives
\EQ{
\mbrk{\mathcal R_L(L_{1,<\bar\mu},S_{1,<\bar\mu})}_q
=\mbrk{\mathcal R_S(L_{1,<\bar\mu},S_{1,<\bar\mu})}_q=0,
\qquad q\in E_{\gamma,\bar\mu}.
}
For real-valued $h$, define
\EQ{
\E(h)
:=\al\tau\pd_\tau h
+\tau\nabla S_{1,<\bar\mu}\cdot\nabla h
+\frac12\tau^{\gamma+1}|\nabla h|^2.
}
Since $\al\gamma=1-\al$, direct substitution into \eqref{eq:eikonal-ES}--\eqref{eq:eikonal-EL} gives
\EQ{
\mathcal R_S\brko{L_{1,<\bar\mu}+ih,S_{1,<\bar\mu}+\tau^\gamma h}
&=\mathcal R_S\brko{L_{1,<\bar\mu},S_{1,<\bar\mu}}+\tau^\gamma\E(h),
\\
\mathcal R_L\brko{L_{1,<\bar\mu}+ih,S_{1,<\bar\mu}+\tau^\gamma h}
&=\mathcal R_L\brko{L_{1,<\bar\mu},S_{1,<\bar\mu}}+i\E(h),
}
while
\EQ{
\brko{L_{1,<\bar\mu}+ih}
-i\tau^{-\gamma}\brko{S_{1,<\bar\mu}+\tau^\gamma h}
=L_{1,<\bar\mu}-i\tau^{-\gamma}S_{1,<\bar\mu}.
}
Write
\EQ{
h(\tau,y)=\chi(y)
+\sum_{q\in E_{\gamma,\bar\mu}}
\tau^qh_q(\log\tau,y).
}
The exponents in $\E(h)$ have the forms
\EQ{
q,
\qquad
1+a+b,
\qquad
1+\gamma+b+c,
\qquad
q\in E_{\gamma,\bar\mu},
\quad
a,b,c\in E_{\gamma,\bar\mu}\cup\fbrko{0}.
}
Thus every such exponent below $\bar\mu$ belongs to $E_{\gamma,\bar\mu}$. For $q\in E_{\gamma,\bar\mu}$, suppose that the coefficients below $q$ have been determined and set
\EQ{
h_{<q}:=\chi+
\sum_{\substack{q'\in E_{\gamma,\bar\mu}\\q'<q}}
\tau^{q'}h_{q'}(r,y),
\qquad r=\log\tau.
}
The equation $\mbrk{\E(h)}_q=0$ is
\EQ{
\al\brko{q+\pd_r}h_q
=-\mbrk{
\tau\nabla S_{1,<\bar\mu}\cdot\nabla h_{<q}
+\frac12\tau^{\gamma+1}|\nabla h_{<q}|^2
}_q.
}
The right-hand side contains only the previously determined coefficients. Since $q>0$, \eqref{eq:eikonal-poly-inversion-transport} determines a unique real polynomial $h_q$. Proceeding through the finite set $E_{\gamma,\bar\mu}$ determines all the coefficients and gives
\EQ{
\mbrk{\E(h)}_q=0,
\qquad q\in E_{\gamma,\bar\mu}.
}
For every $\rho<\rho_1$, the finite allocation of analytic radii used in Proposition \ref{prop:eikonal-prep} gives
\EQ{
h_q(r,y)=\sum_{j=0}^{J_q^h}h_{q,j}(y)r^j,
\qquad
h_{q,j}\in\A^6_\rho,
\quad q\in E_{\gamma,\bar\mu},\quad 0\le j\le J_q^h.
}
The residual identities and the coefficient cancellation above show that $(L_{1,<\bar\mu}+ih,S_{1,<\bar\mu}+\tau^\gamma h)$ has no residual coefficient below $\bar\mu$. Removing from its second component the terms with exponents at least $\bar\mu$ gives
\EQ{
L_{1,<\bar\mu}+ih,
\qquad
S_{1,<\bar\mu}+\tau^\gamma\Pi_{<\bar\mu-\gamma}h
}
and does not create a residual coefficient below $\bar\mu$. Moreover,
\EQ{
\Psi_2=\Psi_1+i\chi,
\qquad
C_{+,2}=C_{+,1}+\chi,
\qquad
S_*(\Psi_2)=S_*(\Psi_1).
}
The exponent-zero term of $h$ changes the homogeneous coefficient from $C_{+,1}$ to $C_{+,2}$, while the normalization $\Sigma_{\gamma,0}=0$ is preserved. The uniqueness of the coefficient recursion in Proposition \ref{prop:eikonal-prep} now proves the claimed relations between $(L_{1,<\bar\mu},S_{1,<\bar\mu})$ and $(L_{2,<\bar\mu},S_{2,<\bar\mu})$.

\emph{Step 2: decay of the difference of the logarithmic phases.} Define
\EQ{
\Phi_{j,<\bar\mu}(t)
:=L_{j,<\bar\mu}(t^{-\al})
-it^{1-\al}S_{j,<\bar\mu}(t^{-\al}).
}
The relations from Step 1 give
\EQ{
\Phi_{2,<\bar\mu}(t)-\Phi_{1,<\bar\mu}(t)
=i\mbrk{(I-\Pi_{<\bar\mu-\gamma})h}(t^{-\al}).
}
Set
\EQ{
D:=\Phi_2-\Phi_1.
}
Then
\EQ{
D
=\brko{\Phi_2-\Phi_{2,<\bar\mu}}
+\brko{\Phi_{2,<\bar\mu}-\Phi_{1,<\bar\mu}}
+\brko{\Phi_{1,<\bar\mu}-\Phi_1}.
}
The first and third terms contain amplitude exponents at least $\bar\mu$ and phase exponents at least $\bar\mu-\gamma$, while every exponent in the middle term is at least $\bar\mu-\gamma$. Hence, with
\EQ{
\lambda:=\al(\bar\mu-\gamma)
=\al(\bar\mu+1)-1,
}
there exists an integer $J\ge0$ such that, for every fixed $\rho<\rho_1$,
\EQ{
\norm{D(t)}_{\A^6_\rho}
+t\norm{\pd_tD(t)}_{\A^4_\rho}
\lsm_{\mu_1,\mu_2,R,\rho}
t^{-\lambda}(1+\log t)^J.
}
Choose $j_0\in\fbrko{1,2}$ such that $\mu_{j_0}=\bar\mu$. Then
\EQ{
1-\al<\kappa_1
\le\kappa_{j_0}
<\al(\bar\mu+1)-1
=\lambda.
}
Consequently, after increasing $T$ if necessary,
\EQ{
t^{\kappa_1}\norm{\re D(t)}_{\A^4_\rho}
+t^{\kappa_1+\al-1}\norm{\im D(t)}_{\A^4_\rho}
\lsm_{\mu_1,\mu_2,R,\rho}
t^{\kappa_1-\lambda}(1+\log t)^J,
}
and hence
\EQ{
D\in\mathcal Z_T^{\kappa_1}(\rho).
}

\emph{Step 3: transfer of the final-state correction and uniqueness.} Let $N_j$ and $\TT_j$ denote the nonlinearity and the Duhamel map associated with $\Phi_j$. Since $\kappa_1\le\kappa_2$,
\EQ{
\mathcal Z_T^{\kappa_2}(\rho)
\subset\mathcal Z_T^{\kappa_1}(\rho).
}
By Definition \ref{def:eikonal-asymptotic}, after increasing $T$ if necessary, choose $\rho_*<\rho<\rho_1$ such that, for $j=1,2$,
\EQ{
Z_j&\in\mathcal Z_T^{\kappa_1}(\rho),
\\
Z_j&=\TT_jZ_j
\quad\hbox{in }\mathcal Z_T^{\kappa_1}(\rho')
\quad\hbox{for every }\rho_*<\rho'<\rho.
}
Set
\EQ{
Z_{21}:=Z_2+D.
}
Then $Z_{21}\in\mathcal Z_T^{\kappa_1}(\rho)$. For every $\rho_*<\rho'<\rho$, the definitions \eqref{eq:Rmu-def} and \eqref{eq:Nmu-def} give
\EQ{
\pd_tD
=\frac{i}{2t^2}\Dy D
+N_1(t,Z_{21})-N_2(t,Z_2).
}
By the estimate from Step 2,
\EQ{
\norm{N_1(t,Z_{21})-N_2(t,Z_2)}_{\A^4_{\rho'}}
&\le\norm{\pd_tD(t)}_{\A^4_{\rho'}}
+\frac1{2t^2}\norm{\Dy D(t)}_{\A^4_{\rho'}}
\\
&\lsm_{\mu_1,\mu_2,R,\rho}
t^{-1-\lambda}(1+\log t)^J.
}
The right-hand side is integrable on $[T,\I)$, and $D(t)\ra0$ in $\A^4_{\rho'}$. The Duhamel formula at infinity therefore gives
\EQ{
D(t)=-\int_t^\I U(t,s)
\mbrko{N_1(s,Z_{21}(s))-N_2(s,Z_2(s))}\dd s.
}
Combining this identity with $Z_2=\TT_2Z_2$ gives
\EQ{
Z_{21}(t)
&=-\int_t^\I U(t,s)N_2(s,Z_2(s))\dd s
-\int_t^\I U(t,s)
\mbrko{N_1(s,Z_{21}(s))-N_2(s,Z_2(s))}\dd s
\\
&=-\int_t^\I U(t,s)N_1(s,Z_{21}(s))\dd s
=\TT_1Z_{21}(t).
}
Thus $Z_{21}$ is a final-state correction for the first truncated profile with exponent $\kappa_1$. The uniqueness argument in the proof of Theorem \ref{thm:eikonal}\textup{(3)} gives, after increasing $T$ once more,
\EQ{
Z_1=Z_{21}=Z_2+D,
\qquad t\ge T.
}
It follows that
\EQ{
\Phi_1+Z_1=\Phi_2+Z_2,
\qquad t\ge T.
}
Hence $u_1=u_2$ for $t\ge T$. Uniqueness of the global $L^2$ flow propagates this equality to $[0,\I)$, as in the proof of Theorem \ref{thm:eikonal}\textup{(4)}.
\end{proof}

Taking $C_+=0$ in Theorem \ref{thm:eikonal} proves the Fuchsian assertion of Theorem \ref{thm:main}.
\appendix

\section{On the admissible data}\label{app:admissible-data}

For $\nu\in\R$ and $z>0$, define the modified Bessel function of the second kind by
\EQn{\label{eq:modified-bessel-def}
K_\nu(z):=\int_0^\I e^{-z\cosh r}\cosh(\nu r)\dd r.
}

\begin{lem}[A Fourier-Bessel family in analytic Wiener spaces]
\label{lem:admissible-fourier-bessel}
Let $a>0$ and $L>0$, and set
\EQ{
f_{a,L}(x):=\brko{1+\frac{|x|^2}{L^2}}^{-a},
\qquad x\in\R^d.
}
Then $\wh f_{a,L}$ is well-defined as tempered distribution, and it is locally integrable satisfying
\EQn{\label{eq:admissible-bessel-formula}
\widehat{f_{a,L}}(\xi)
=\frac{2^{1-a}}{\Gamma(a)}L^d
\brko{L|\xi|}^{a-\frac d2}
K_{a-\frac d2}\brko{L|\xi|},
\qquad \xi\ne0,
}
Moreover,
\EQn{\label{eq:admissible-bessel-low}
\absb{\widehat{f_{a,L}}(\xi)}
\lsm_{a,d}L^d
\begin{cases}
\brko{L|\xi|}^{2a-d},&0<a<d/2,\\
1+\abs{\log\brko{L|\xi|}},&a=d/2,\\
1,&a>d/2,
\end{cases}
\qquad 0<L|\xi|\le1,
}
and
\EQn{\label{eq:admissible-bessel-high}
\absb{\widehat{f_{a,L}}(\xi)}
\lsm_{a,d}L^d
\brko{L|\xi|}^{a-\frac{d+1}{2}}e^{-L|\xi|},
\qquad L|\xi|\ge1.
}
Consequently, for every $\sigma\ge0$ and $0\le\rho<L$,
\EQn{\label{eq:admissible-bessel-wiener}
f_{a,L}\in\A^\sigma_\rho.
}
Moreover,
\EQn{\label{eq:admissible-bessel-L2}
f_{a,L}\in L^2(\R^d)
\quad\Longleftrightarrow\quad
a>\frac d4,
}
\end{lem}

\begin{proof}
Set $\nu=a-d/2$. The Gamma integral representation and the Fourier transform of a Gaussian give
\EQ{
f_{a,L}(x)
&=\frac1{\Gamma(a)}
\int_0^\I t^{a-1}e^{-t}e^{-t|x|^2/L^2}\dd t,
\\
\widehat{e^{-t|\cdot|^2/L^2}}(\xi)
&=\frac{L^d}{2^{d/2}t^{d/2}}
e^{-L^2|\xi|^2/(4t)},
}
and hence
\EQ{
\widehat{f_{a,L}}(\xi)
=\frac{L^d}{2^{d/2}\Gamma(a)}
\int_0^\I t^{\nu-1}e^{-t-L^2|\xi|^2/(4t)}\dd t.
}
Under the change of variables $t=(L|\xi|/2)e^r$, the definition \eqref{eq:modified-bessel-def} yields
\EQ{
\int_0^\I t^{\nu-1}e^{-t-L^2|\xi|^2/(4t)}\dd t
&=\brko{\frac{L|\xi|}{2}}^{\nu}
\int_{-\I}^{\I}e^{-L|\xi|\cosh r}e^{\nu r}\dd r
\\
&=2\brko{\frac{L|\xi|}{2}}^{\nu}K_\nu\brko{L|\xi|}.
}
This proves \eqref{eq:admissible-bessel-formula} formally. For $a>d/2$, the calculation for Fourier transform is justified in $f_{a,L}\in L^1$. For general $a>0$, one first restricts the $t$-integral to $[\ep,R]$ and then passes to the limit in $\mathcal S'(\R^d)$. 

Next, we prove that the right hand side of \eqref{eq:admissible-bessel-formula} is integrable. The identity $K_{-\nu}=K_\nu$, which follows directly from \eqref{eq:modified-bessel-def}, and the classical bounds, which follow from asymptotic expansions in \cite{Temme1996},
\EQ{
K_\nu(z)&\lsm_\nu z^{-|\nu|},
\qquad 0<z\le1,\quad \nu\ne0,
\\
K_0(z)&\lsm 1+|\log z|,
\qquad 0<z\le1,
\\
K_\nu(z)&\lsm_\nu z^{-\frac12}e^{-z},
\qquad z\ge1,
}
imply \eqref{eq:admissible-bessel-low} and \eqref{eq:admissible-bessel-high}. These bounds also yield that right hand side of \eqref{eq:admissible-bessel-formula} is locally integrable. Now, we prove \eqref{eq:admissible-bessel-wiener}. To verify the low-frequency integrability for $\{|\xi|\le L^{-1}\}$, observe first that the weight $e^{\rho\jb{\xi}}\jb{\xi}^{\sigma}$ is bounded. By \eqref{eq:admissible-bessel-low} and polar coordinates, we can prove that for every $a>0$,
\EQ{
\int_{|\xi|\le L^{-1}} e^{\rho\jb{\xi}}\jb{\xi}^{\sigma}
\absb{\widehat{f_{a,L}}(\xi)}\dd\xi <\I.
}
On $\{|\xi|\ge L^{-1}\}$, the exponential decay $e^{-L|\xi|}$ in \eqref{eq:admissible-bessel-high} dominates the same weight whenever $\rho<L$. This proves \eqref{eq:admissible-bessel-wiener}. Finally, the change of variables $x=Ly$ shows that
\EQ{
\norm{f_{a,L}}_{L^2}^2
=L^d\int_{\R^d}(1+|y|^2)^{-2a}\dd y.
}
The last integral is finite exactly when $4a>d$, which proves \eqref{eq:admissible-bessel-L2}.
\end{proof}

\begin{prop}[Non-emptiness of the admissible class in every dimension]
\label{prop:admissible-data-nonempty}
Fix $d\ge1$, $p>0$, $\rho_0>0$, $R>0$, and $0<\ep\le1$. Let $A>d/4$. Then there exist $L>\rho_0$ and $M>0$ such that
\EQ{%\label{eq:admissible-log-profile}
\Psi_{A,L,M}(x)
:=-M-A\log\brko{1+\frac{|x|^2}{L^2}}
}
satisfies
\EQn{\label{eq:admissible-log-gradient-bound}
\norm{\nabla\Psi_{A,L,M}}_{\A^6_{\rho_0}}
\le R
}
and
\EQn{\label{eq:admissible-log-smallness}
\norm{e^{\Psi_{A,L,M}}}_{L^2}
+\norm{e^{\Psi_{A,L,M}}}_{\A^6_{\rho_0}}
+\norm{e^{p\re\Psi_{A,L,M}}}_{\A^6_{\rho_0}}
\le\ep.
}
In particular, the admissible final-profile class in Definition \ref{def:final-data} is nonempty for every $d\ge1$.
\end{prop}

\begin{proof}
We first control the logarithmic derivative. Let
\EQ{
G_A(y):=-\frac{2Ay}{1+|y|^2}.
}
Then
\EQn{\label{eq:admissible-log-gradient-scaling}
\nabla\Psi_{A,L,M}(x)
=L^{-1}G_A\brko{\frac{x}{L}},
\qquad
\widehat{\nabla\Psi_{A,L,M}}(\xi)
=L^{d-1}\widehat G_A(L\xi).
}
For $1\le j\le d$,
\EQ{
G_{A,j}(y)=-2Ay_jf_{1,1}(y),
\qquad
\widehat{G_{A,j}}=-2Ai\partial_{\eta_j}\widehat{f_{1,1}}
\quad\text{in }\mathcal S'(\R^d).
}
By \eqref{eq:admissible-bessel-formula} with $a=L=1$ and the identity $K_{-\nu}=K_\nu$,
\EQ{
\widehat{f_{1,1}}(\eta)
=|\eta|^{1-\frac d2}K_{\frac d2-1}(|\eta|),
\qquad \eta\ne0.
}
Using
\EQ{
\frac{\dd}{\dd r}\brko{r^{-\nu}K_\nu(r)}
=-r^{-\nu}K_{\nu+1}(r)
}
gives
\EQ{
\partial_{\eta_j}\widehat{f_{1,1}}(\eta)
=-\eta_j|\eta|^{-\frac d2}K_{\frac d2}(|\eta|),
\qquad \eta\ne0.
}
For $\varphi\in\mathcal S(\R^d)$, integration by parts on $\{|\eta|>\varepsilon\}$ gives
\EQ{
-\int_{|\eta|>\varepsilon}
\widehat{f_{1,1}}(\eta)\partial_{\eta_j}\varphi(\eta)\dd\eta
&=\int_{|\eta|>\varepsilon}
\partial_{\eta_j}\widehat{f_{1,1}}(\eta)\varphi(\eta)\dd\eta
\\
&\quad+
\int_{|\eta|=\varepsilon}
\widehat{f_{1,1}}(\eta)\varphi(\eta)
\frac{\eta_j}{|\eta|}\dd S(\eta).
}
The low-frequency estimates in Lemma \ref{lem:admissible-fourier-bessel}, together with the radiality of $\widehat{f_{1,1}}$ when $d=1$, show that the boundary term tends to zero as $\varepsilon\downarrow0$. The Bessel estimates also show that the interior term has a locally integrable limit. Therefore,
\EQ{%\label{eq:admissible-log-gradient-bessel}
\widehat G_A(\eta)
=2Ai\eta|\eta|^{-\frac d2}K_{\frac d2}(|\eta|)
\quad\text{in }\mathcal S'(\R^d).
}
The Bessel estimates imply
\EQn{\label{eq:admissible-log-gradient-frequency}
\absb{\widehat G_A(\eta)}
\lsm_d A
\begin{cases}
|\eta|^{1-d},&0<|\eta|\le1,\\
|\eta|^{\frac{1-d}{2}}e^{-|\eta|},&|\eta|\ge1.
\end{cases}
}
By \eqref{eq:admissible-log-gradient-scaling} and the change of variables $\eta=L\xi$,
\EQ{
\norm{\nabla\Psi_{A,L,M}}_{\A^6_{\rho_0}}
=\frac1L\int_{\R^d}
e^{\rho_0\jb{\eta/L}}\jb{\eta/L}^{6}
\abs{\widehat G_A(\eta)}\dd\eta.
}
For $L\ge\max\fbrko{1,2\rho_0}$,
\EQ{
e^{\rho_0\jb{\eta/L}}
\le e^{\rho_0}e^{\frac12|\eta|},
\qquad
\jb{\eta/L}^{6}\le\jb{\eta}^{6}.
}
By \eqref{eq:admissible-log-gradient-frequency} and polar coordinates,
\EQ{
\int_{|\eta|\le1}
e^{\rho_0\jb{\eta/L}}\jb{\eta/L}^{6}
\absb{\widehat G_A(\eta)}\dd\eta
\lsm_{A,d,\rho_0}1,
}
and
\EQ{
\int_{|\eta|\ge1}
e^{\rho_0\jb{\eta/L}}\jb{\eta/L}^{6}
\absb{\widehat G_A(\eta)}\dd\eta
\lsm_{A,d,\rho_0}
\int_1^\I \jb{r}^{6}r^{\frac{d-1}{2}}
e^{-\frac r2}\dd r
\lsm_{A,d,\rho_0}1.
}
Consequently,
\EQn{\label{eq:admissible-log-gradient-decay}
\norm{\nabla\Psi_{A,L,M}}_{\A^6_{\rho_0}}
\lsm_{A,d,\rho_0}L^{-1}.
}
We now choose $L\ge\max\fbrko{1,2\rho_0}$ sufficiently large that the right-hand side of \eqref{eq:admissible-log-gradient-decay} is at most $R$. This fixes $L$ and proves \eqref{eq:admissible-log-gradient-bound}.

For this fixed $L$, Lemma \ref{lem:admissible-fourier-bessel} and $A>d/4$ show that the finite numbers
\EQ{
B_L&:=
\norm{f_{A,L}}_{L^2}
+\norm{f_{A,L}}_{\A^6_{\rho_0}},
\\
D_L&:=
\norm{f_{pA,L}}_{\A^6_{\rho_0}}
}
are well defined. Notice that only $pA>0$ is needed for the finiteness of $D_L$. Since
\EQ{
e^{\Psi_{A,L,M}}=e^{-M}f_{A,L},
\qquad
e^{p\re\Psi_{A,L,M}}=e^{-pM}f_{pA,L},
}
we have
\EQn{\label{eq:admissible-log-smallness-reduced}
\norm{e^{\Psi_{A,L,M}}}_{L^2}
+\norm{e^{\Psi_{A,L,M}}}_{\A^6_{\rho_0}}
+\norm{e^{p\re\Psi_{A,L,M}}}_{\A^6_{\rho_0}}
=e^{-M}B_L+e^{-pM}D_L.
}
After $L$ has been fixed, choose
\EQ{
M\ge
\max\fbrko{
1,\,
\log\frac{2B_L}{\ep},\,
\frac1p\log\frac{2D_L}{\ep}}.
}
Then \eqref{eq:admissible-log-smallness-reduced} is at most $\ep$. Thus \eqref{eq:admissible-log-smallness} follows, and the proof is complete.
\end{proof}

\begin{cor}[Infinite-dimensional families of admissible final data]\label{cor:admissible-data-richness}
Fix $d\ge1$, $p>0$, $\rho_0>0$, $R>0$, $0<\ep\le1$, and $A>d/4$. There exist $L>\rho_0$, $M>0$, and $\delta>0$ such that, for every complex-valued $h\in\A^7_{\rho_0}$ with $\norm{h}_{\A^7_{\rho_0}}<\delta$, the final datum $W_h=e^{\Psi_h}$ defined by
\EQ{
\Psi_h(x)&:=\Psi_{A,L,M}(x)+h(x)
=-M-A\log\brko{1+\frac{|x|^2}{L^2}}+h(x),
\\
W_h(x)&:=e^{-M}\brko{1+\frac{|x|^2}{L^2}}^{-A}e^{h(x)}
}
is admissible in the sense of Definition \ref{def:final-data} with the same parameters $(\rho_0,R,\ep)$. The radius $\delta$ can be chosen depending only on $d,p,R$, and hence independently of $\ep$, while $M$ and the seed $W_0:=e^{\Psi_{A,L,M}}$ may depend on $\ep$. The map $h\mapsto W_h$ is injective. Moreover, for $h,k$ in this ball,
\EQ{
&\norm{W_h-W_k}_{L^2}
+\norm{W_h-W_k}_{\A^6_{\rho_0}}
+\norm{|W_h|^p-|W_k|^p}_{\A^6_{\rho_0}}
\lsm_{d,p,\rho_0,R}\ep\norm{h-k}_{\A^7_{\rho_0}}.
}
\end{cor}

\begin{proof}
Apply Proposition \ref{prop:admissible-data-nonempty} with $R/2$ and $\ep/2$ in place of $R$ and $\ep$. This gives $L,M$. Set $\Psi_0:=\Psi_{A,L,M}$ and $W_0:=e^{\Psi_0}$. Then
\EQ{
\norm{\nabla\Psi_0}_{\A^6_{\rho_0}}
&\le R/2,
\\
\norm{W_0}_{L^2}
+\norm{W_0}_{\A^6_{\rho_0}}
+\norm{|W_0|^p}_{\A^6_{\rho_0}}
&\le\ep/2.
}

Choose $c=c_d\ge1$ large enough to dominate the sixth-order algebra and embedding constants. Then
\EQ{
\norm{h}_{\A^6_{\rho_0}}
+\norm{h}_{L^\I}
+\norm{\nabla h}_{\A^6_{\rho_0}}
\le c\norm{h}_{\A^7_{\rho_0}}.
}
The exponential is used as a multiplier rather than as an element of $\A^6_{\rho_0}$. Indeed, expanding $Be^f=B+\sum_{n\ge1}Bf^n/n!$ and using \eqref{eq:fourier-algebra} gives
\EQ{
\norm{Be^f}_{\A^6_{\rho_0}}
\le e^{c\norm{f}_{\A^6_{\rho_0}}}
\norm{B}_{\A^6_{\rho_0}},
\qquad B,f\in\A^6_{\rho_0}.
}
Since $|W_0e^h|^p=|W_0|^pe^{p\re h}$, after increasing $c$ by a constant depending only on $d$, the preceding estimates give
\EQ{
\norm{W_h}_{L^2}
+\norm{W_h}_{\A^6_{\rho_0}}
+\norm{|W_h|^p}_{\A^6_{\rho_0}}
&\le e^{c(1+p)\norm{h}_{\A^7_{\rho_0}}}
\mbrko{
\norm{W_0}_{L^2}
+\norm{W_0}_{\A^6_{\rho_0}}
+\norm{|W_0|^p}_{\A^6_{\rho_0}}},
\\
\norm{\nabla(\Psi_0+h)}_{\A^6_{\rho_0}}
&\le R/2+c\norm{h}_{\A^7_{\rho_0}}.
}
Set
\EQ{
\delta:=\min\fbrko{1,\frac{R}{2c},\frac{\log2}{c(1+p)}}.
}
Then every $h$ in the stated ball satisfies \eqref{eq:R-bound} and \eqref{eq:eps-small} with $(\Psi,W)=(\Psi_h,W_h)$. Since $\rho_0>0$, the condition $h\in\A^7_{\rho_0}$ implies that $h$ is smooth. Thus $W_h=e^{\Psi_h}$ is admissible. The displayed choice of $\delta$ is independent of $\ep$. When $\ep$ changes, the value of $M$ and hence the seed $W_0$ may change.

If $W_h=W_k$, then the positivity of $W_0$ gives $e^{h-k}=1$, and hence
\EQ{
h(x)-k(x)\in2\pi i\Z,
\qquad x\in\R^d.
}
On the other hand,
\EQ{
\norm{h-k}_{L^\I}<2c\delta\le2\log2<2\pi.
}
Therefore $h-k=0$, which proves injectivity.

Finally, use
\EQ{
e^h-e^k
&=(h-k)\int_0^1e^{k+\theta(h-k)}\dd\theta,
\\
e^{p\re h}-e^{p\re k}
&=p\re(h-k)\int_0^1e^{p\re k+\theta p\re(h-k)}\dd\theta.
}
The multiplier estimate above, the algebra property, and the $L^\I$ embedding yield
\EQ{
&\norm{W_h-W_k}_{L^2}
+\norm{W_h-W_k}_{\A^6_{\rho_0}}
+\norm{|W_h|^p-|W_k|^p}_{\A^6_{\rho_0}}
\lsm_{d,p,\rho_0,R}\ep\norm{h-k}_{\A^7_{\rho_0}},
}
which proves the stated Lipschitz estimate.
\end{proof}

Writing $h=h_1+ih_2$ with $h_1,h_2$ real-valued, the preceding family allows independent amplitude and phase perturbations with no symmetry assumption on either component. This is a neighborhood in the multiplicative logarithmic variable $W_h=W_0e^h$.

\section{Low-order coefficients in the eikonal recursion}
\label{app:first-eikonal-coefficients}

This appendix writes out some low-order coefficients produced by Proposition \ref{prop:eikonal-prep}.  All coefficients are functions of $y$, and $r=\log\tau$.  The order-zero pair is
\EQ{
L_0=\Psi,
\qquad
S_0=\frac1{1-\al}e^{p\re\Psi}.
}

\subsection{\texorpdfstring{The case $\gamma=1$}{The case gamma=1}}

Here $\gamma=1$, equivalently $\al=1/2$ and $p=1/d$.  The homogeneous coefficient $C_+$ and the resonant forced correction both occur at order one.  Write
\EQ{
L&=L_0+\tau L_1+\tau^2\brko{L_{2,0}+rL_{2,1}}
+\tau^3\brko{L_{3,0}+rL_{3,1}}+\cdots,
\\
S&=S_0+\tau\brko{C_++rS_{1,1}}
+\tau^2\brko{S_{2,0}+rS_{2,1}}
+\tau^3\brko{S_{3,0}+rS_{3,1}+r^2S_{3,2}}+\cdots .
}
At order one,
\EQ{
L_1&=-2\brko{\nabla S_0\cdot\nabla L_0+\frac12\Dy S_0},
\\
S_{1,1}
&=-2\brko{\frac12\abs{\nabla S_0}^2
+p e^{p\re L_0}\re L_1}.
}
The order-two coefficient of $L$ is
\EQ{
L_{2,1}
&=-\brko{\nabla S_{1,1}\cdot\nabla L_0
+\frac12\Dy S_{1,1}},
\\
L_{2,0}
&=-\brko{
\nabla S_0\cdot\nabla L_1
+\nabla C_+\cdot\nabla L_0
+\frac12\Dy C_+
}
\\
&\quad
-\frac i2\brko{\Dy L_0+\nabla L_0\cdot\nabla L_0}
-\frac12L_{2,1}.
}
The order-two coefficient of $S$ is
\EQ{
S_{2,1}
&=-2\brko{
\nabla S_0\cdot\nabla S_{1,1}
+p e^{p\re L_0}\re L_{2,1}},
\\
S_{2,0}
&=-2\brko{
\nabla S_0\cdot\nabla C_+
+e^{p\re L_0}
\brko{p\re L_{2,0}+\frac{p^2}{2}(\re L_1)^2}}
-S_{2,1}.
}
At order three, the coefficient of $L$ is
\EQ{
L_{3,1}
&=-\frac23\brko{
\nabla S_0\cdot\nabla L_{2,1}
+\nabla S_{1,1}\cdot\nabla L_1
+\nabla S_{2,1}\cdot\nabla L_0
+\frac12\Dy S_{2,1}},
\\
L_{3,0}
&=-\frac23\brko{
\nabla S_0\cdot\nabla L_{2,0}
+\nabla C_+\cdot\nabla L_1
+\nabla S_{2,0}\cdot\nabla L_0
+\frac12\Dy S_{2,0}
}
\\
&\quad
-\frac i3\brko{\Dy L_1+2\nabla L_0\cdot\nabla L_1}
-\frac13L_{3,1}.
}
The order-three coefficient of $S$ is
\EQ{
S_{3,2}
&=-\frac12\abs{\nabla S_{1,1}}^2,
\\
S_{3,1}
&=-\nabla S_0\cdot\nabla S_{2,1}
-\nabla C_+\cdot\nabla S_{1,1}
\\
&\quad
-p e^{p\re L_0}\re L_{3,1}
-p^2e^{p\re L_0}\re L_1\re L_{2,1} -S_{3,2},
\\
S_{3,0}
&=-\nabla S_0\cdot\nabla S_{2,0}
 -\frac12\abs{\nabla C_+}^2
-p e^{p\re L_0}\re L_{3,0}
\\
&\quad
-p^2e^{p\re L_0}\re L_1\re L_{2,0} -\frac{p^3}{6}e^{p\re L_0}(\re L_1)^3
-\frac12S_{3,1}.
}
\subsection{\texorpdfstring{The case $1<\gamma<2$}{The case 1<gamma<2}}

This range is equivalent to \(2/(3d)<p<1/d\).  The prescribed homogeneous term is \(C_+\tau^\gamma\).  Apart from this term, no correction is generated at exponent $\gamma$; if $2\gamma<3$, no correction is generated at exponent $2\gamma$ either:
\EQ{
L_\gamma=0,
\qquad
\Sigma_\gamma=0,
\qquad
L_{2\gamma}=0,
\qquad
\Sigma_{2\gamma}=0\quad(2\gamma<3).
}
The nonzero coefficients through the third integer order are written as
\EQ{
L&=L_0+\tau L_1+\tau^2L_2+\tau^{1+\gamma}L_{1+\gamma}
+\tau^3L_3+\cdots,
\\
S&=S_0+C_+\tau^\gamma+\tau S_1+\tau^2S_2
+\tau^{1+\gamma}S_{1+\gamma}+\tau^3S_3+\cdots .
}
The integer coefficients of orders one and two are
\EQ{
L_1&=-\frac1\al\brko{\nabla S_0\cdot\nabla L_0+\frac12\Dy S_0},
\\
S_1&=-\frac1{2\al-1}
\brko{\frac12\abs{\nabla S_0}^2+e^{p\re L_0}p\re L_1},
\\
L_2&=-\frac1{2\al}
\brko{\nabla S_0\cdot\nabla L_1
+\nabla S_1\cdot\nabla L_0+\frac12\Dy S_1},
\\
S_2&=-\frac1{3\al-1}
\brko{\nabla S_0\cdot\nabla S_1
+e^{p\re L_0}
\brko{p\re L_2+\frac{p^2}{2}(\re L_1)^2}}.
}
Using $\al(1+\gamma)=1$, the coefficients at exponent $1+\gamma$ are
\EQ{
L_{1+\gamma}
&=-\brko{
\nabla C_+\cdot\nabla L_0+\frac12\Dy C_+
+\frac{i}{2}\brko{\Dy L_0+\nabla L_0\cdot\nabla L_0}},
\\
S_{1+\gamma}
&=-\frac1\al
\brko{
\nabla S_0\cdot\nabla C_+
+e^{p\re L_0}p\re L_{1+\gamma}}.
}
The third integer coefficients are
\EQ{
L_3&=-\frac1{3\al}
\brko{\nabla S_0\cdot\nabla L_2
+\nabla S_1\cdot\nabla L_1
+\nabla S_2\cdot\nabla L_0+\frac12\Dy S_2},
\\
S_3&=-\frac1{4\al-1}\brko{
\nabla S_0\cdot\nabla S_2+\frac12\abs{\nabla S_1}^2
\\
&\quad
+e^{p\re L_0}\brko{p\re L_3+p^2\re L_1\re L_2
+\frac{p^3}{6}(\re L_1)^3}}.
}
Beyond the prescribed homogeneous term $C_+\tau^\gamma$, the first forced coefficient depending on $C_+$ occurs at order $1+\gamma$.

\end{document}